\documentclass[reqno]{amsart}
\usepackage[top=1.1in, bottom=1in, left=1in, right=1in]{geometry}
\usepackage{amsfonts}
\usepackage{amssymb}
\usepackage{amsthm}
\usepackage{amsmath}
\usepackage{mathrsfs}
\usepackage{color}
\usepackage{enumerate}
\usepackage[numbers,sort&compress]{natbib}
\usepackage{pdfsync}
\usepackage{esint}
\usepackage{graphicx}
\usepackage{float}
\usepackage{caption}
\usepackage{subfigure}

\allowdisplaybreaks

\usepackage{tikz} 
\usetikzlibrary{calc}
\usetikzlibrary{intersections}
\usetikzlibrary{patterns}

\usepackage[colorlinks,
linkcolor=black,
anchorcolor=blue,
citecolor=blue,
]{hyperref}

\usepackage{tikz} 
\usetikzlibrary{calc}
\usetikzlibrary{intersections}
\usetikzlibrary{patterns}

\makeatother

\numberwithin{equation}{section}

\newtheorem{theorem}{Theorem}[section]
\newtheorem{definition}[theorem]{Definition}

\newtheorem{lemma}[theorem]{Lemma}
\newtheorem{remark}[theorem]{Remark}

\newtheorem{proposition}[theorem]{Proposition}

\numberwithin{equation}{section}

\newcommand{\T}{{\mathbb{T}}}
\newcommand{\Z}{{\mathbb{Z}}}

\newcommand{\ga}{\gamma}

\newcommand{\ru}{\mathring{R}_q}
\newcommand{\run}{\mathring{R}_n}
\newcommand{\runq}{\mathring{R}_{n,q}}
\newcommand{\Ru}{\mathring{R}}

\newcommand{\G}{\mathring{G}^{B}}

\newcommand{\rw}{r_{\perp}}

\newcommand{\p}{\partial}
\newcommand{\Bb}{\widetilde{B}}
\newcommand{\Bu}{\widetilde{u}}
\newcommand{\ui}{v}
\newcommand{\bi}{H}
\renewcommand{\P}{\mathbb{P}}

\renewcommand{\div}{{\mathrm{div}}}
\newcommand{\curl}{{\mathrm{curl}}}

\newcommand{\om}{\Omega}
\newcommand{\omw}{\widetilde{w}}
\newcommand{\dw}{\widetilde{d}}

\newcommand{\la}{\lambda_q}
\newcommand{\laq}{\lambda_{q+1}}

\newcommand{\dq}{\delta_q}
\newcommand{\dqq}{\delta_{q+1}}
\newcommand{\dqqq}{\delta_{q+2}}

\newcommand{\J}{J_{q}}
\newcommand{\JJ}{\widetilde{J}_{q}}

\newcommand{\C}{ C_{t}}
\newcommand{\Lo}{L}

\newcommand{\ft}{\mathcal{F}_{t}}
\newcommand{\F}{\mathcal{F}}

\def\a{{\alpha}}

\def\wt{\widetilde}
\def\9{{\infty}}

\def\bbr{{\mathbb{R}}}
\def\bbn{{\mathbb{N}}}
\def\bbp{{\mathbb{P}}}
\def\bbq{{\mathbb{Q}}}

\def\({\left(}
\def\){\right)}	
\def\[{\left[}	
\def\]{\right]}	

\begin{document}

\title[] {Existence and non-uniqueness of 
probabilistically strong solutions to 
3D stochastic magnetohydrodynamic equations}

\author{Wenping Cao}
\address{School of Mathematical Sciences, Shanghai Jiao Tong University, China.}
\email[Wenping Cao]{caowenping1@sjtu.edu.cn}
\thanks{}

\author{Yachun Li}
\address{School of Mathematical Sciences, CMA-Shanghai, MOE-LSC, and SHL-MAC,  Shanghai Jiao Tong University, China.}
\email[Yachun Li]{ycli@sjtu.edu.cn}
\thanks{}

\author{Deng Zhang}
\address{School of Mathematical Sciences, CMA-Shanghai, Shanghai Jiao Tong University, China.}
\email[Deng Zhang]{dzhang@sjtu.edu.cn}
\thanks{}



\begin{abstract}
 We are concerned with the 3D stochastic magnetohydrodynamic (MHD) equations
 driven by additive noise on torus.
For arbitrarily prescribed divergence-free initial data in $L^{2}_x$,
we construct infinitely many  
probabilistically strong and analitically
weak solutions in the class $L^{r}_{\om}L_{t}^{\gamma}W_{x}^{s,p}$,
where $r>1$ and $(s, \gamma, p)$ lie in a supercritical regime with respect to the Lady\v{z}henskaya-Prodi-Serrin (LPS) criteria.
In particular,
we get the non-uniqueness of probabilistically strong 
solutions, which is sharp at one LPS endpoint space.
Our proof utilizes intermittent flows 
which are different  from those of Navier-Stokes equations 
and derives the non-uniqueness even in the high 
viscous and resistive regime beyond the Lions exponent $5/4$.
Furthermore, we prove that as the noise intensity tends to zero,
the accumulation points of stochastic MHD solutions contain all deterministic solutions to
MHD solutions,
{which include the recently constructed solutions in \cite{ZZL-MHD,ZZL-MHD-sharp}} 
to deterministic MHD systems. 
\end{abstract}

\maketitle
\tableofcontents

\section{Introduction and main results}\label{Introduction and main results}

\subsection{Classical stochastic MHD} \label{Subsec-intro}
We are concerned with the classical three-dimensional stochastic magnetohydrodynamic equations (MHD)
 on the torus $\T^{3}:=\[-\pi,\pi\]^{3}$,
\begin{eqnarray}\label{equa-SMHD}
	\left\{\begin{array}{lc}
		\mathrm{d}u-\nu_{1}\Delta u \mathrm{d}t+\big((u\cdot \nabla)u-(B\cdot \nabla )B\big)\mathrm{d}t +\nabla P \mathrm{d}t = \text{d}W^{(1)},\\
		\mathrm{d}B-\nu_{2}\Delta B\mathrm{d}t+\big((u\cdot \nabla )B-(B\cdot \nabla )u\big)\mathrm{d}t= \text{d}W^{(2)},\\
		\div u=0,\quad\div B=0. 
	\end{array}\right.
\end{eqnarray} 
Here $u=(u_{1},u_{2},u_{3})^{\top}$ and
$B=(B_{1},B_{2},B_{3})^{\top}$ represent the 
velocity and  magnetic fields, respectively,
$P$ is the pressure of fluid,
and $ \nu_{1}$, $\nu_{2} > 0$ are the viscous and resistive coefficients, respectively.
The noises $W^{(i)}$, $i=1,2$, are independent $G_{i}G_{i}^{*}$-Wiener processes
on a probability space $\(\om, \F, (\F_t), \textbf{P}\)$
of the form
\begin{align}\label{Wiener}
   W^{(i)}=\sum\limits_{k\in \bbn}c_{k}^{(i)}e_{k}\beta_{k}^{(i)}, 
\end{align} 
where $G_{i}$ is a Hilbert-Schmidt operator from some Hilbert space $U$ to $H_{\sigma}^{4}$, 
$H_{\sigma}^{4}$ denotes the subspace of all mean-free and divergence-free functions in $H^{4}_x$,  
$\{e_{k}\}_{k\in \bbn} \subseteq H_{\sigma}^{4}$
is an orthonormal basis consisting of the eigenvectors of $G_{i}G_{i}^{*}$ with corresponding eigenvalues $(c_{k}^{(i)})^{2}$ satisfying
$ \sum_{k\in \bbn}(c_{k}^{(i)})^{2}<\infty$, 
and $\{\beta_{k}^{(i)}\}$ are standard independent Brownian motions.

In particular, 
without the magnetic field $B\equiv 0$,
\eqref{equa-SMHD} reduces to the classical incompressible Navier-Stokes equations (NSE).
In the absence of the viscosity and resistivity $\nu_{1}=\nu_{2}=0$,
\eqref{equa-SMHD} becomes exactly 
the ideal MHD equations,
which further reduces to the incompressible Euler equations when $B\equiv 0$.

MHD equations have wide-ranging applications in various fields, including astrophysics, geophysics,
plasma physics and engineering. 
It is also of great mathematical interest, as it couples Navier-Stokes equations of fluid dynamics and Maxwell equations of electromagnetism. 
There is an extensive literature on MHD systems. 
See, for instance, \cite{ST} for 
deterministic MHD systems.  
For stochastic MHD systems, 
we refer to \cite{Sritharan-Sundar,
Yamazaki-2D-STMHD-well,Yamazaki-3D-STMHD-well} 
for the well-posedness, 
\citep{Barbu-Daprato-SMHD}  
for ergodicity, 
and \cite{CM10} for large deviation principle. 
As in the context of Navier-Stokes equations,
the uniqueness of Leray-Hopf solutions to  stochastic MHD still remains a challenging open problem.

In the last decade,
there have been significant progresses towards the non-uniqueness of weak solutions to hydrodynamic models.
In the seminal paper \cite{BV-NS},  Buckmaster-Vicol
proved the non-uinqueness of weak solutions 
in  $C_{t}L_{x}^{2}$ to the 
3D NSE,
by exploiting intermittency
in convex integration constructions,
which date back to the pioneering work \citep{LS-Elur} by De Lellis and Sz\'{e}kelyhidi, Jr.
on Euler equations.
Afterwards,
Luo-Titi \cite{LT-hyperviscous NS}
and Buckmaster-Colombo-Vicol \cite{BCV}
proved that the Lions exponent $5/4$ is the sharp threshold
for the $C_tL^2_x$ well-posedness of the 
3D hyper-dissipative NSE.
Moreover,
Cheskidov and Luo proved
the sharp non-uniquness near the endpoints of
Lady\v{z}henskaya-Prodi-Serrin (LPS) criteria for NSE \cite{CL-NSE1,CL-NSE2}. 
The sharp non-uniqueness in the case of (hyper-)dissipative NSE and MHD, 
with viscosity even beyond the Lions exponent, 
has been 
recently studied in   \cite{LQZZ-NSE,ZZL-MHD,ZZL-MHD-sharp}.  
See also \cite{Dai21} for the 
non-uniqueness in the Leray-Hopf class for the Hall-MHD system,  
and \cite{MY24} for the 
non-uniqueness 
for MHD equations. 

In the stochastic setting, 
Hofmanov\'{a}, Zhu and Zhu \citep{HZZ-1} first 
proved the non-uniqueness in law of the 3-D stochastic NSE,
and developed a stochastic counterpart of convex integration scheme.
The probabilistically strong solutions 
was constructed later in $L^q_t L^2_x$ 
with  $1\leq q<\infty$ \cite{HZZ-3}. 
Chen, Dong and Zhu \citep{CDZ} proved the sharp non-uniqueness at one endpoint
of the LPS criteria.
See also \cite{HZZ23-arma} 
for the non-uniqueness in the white noise case, 
and \cite{HZZ-2} for the non-uniqueness of ergodic measures.

\medskip 
It remains open whether weak solutions to 
the classical stochastic MHD system (with $(-\Delta)$) is unique or not. 
Furthermore, 
even the existence of global-in-time probabilistically strong solutions is difficult, it 
was an open problem for the 3D stochastic NSE which has been recently proved 
in the spaces $L^q_t L^2_x$ 
with  $1\leq q<\infty$ in \cite{HZZ-3}.   
The existence 
of global-in-time probabilistically strong solutions to the 3D stochastic MHD
is still open.

The main distinction between NSE and MHD lies in that MHD systems contain  strong couplings between velocity and magnetic fields, 
and the nonlinearity is symmetric in the NS component, 
but anti-symmetric in the Maxwell component. 
Thus, 
the crucial intermittent velocity flows constructed in convex integration for NSE 
are not applicable to MHD systems. 
This particular geometrical structure of MHD systems 
restricts the choice of oscillation directions for intermittent flows, 
and thus poses a major obstacle to 
constructing velocity and magnetic flows with sufficiently strong intermittency to control the viscosity and resistivity 
$(-\Delta)$.   
For the hypo-dissipative MHD system 
with $(-\Delta)^\alpha$, $\alpha\in (0,1)$, 
the non-uniqueness
in law of martingale solutions was proved by Yamazaki 
 \citep{Yamazaki-SMHD}.

\medskip 
The aim of this paper is to address these two problems.   
For arbitrary prescribed initial data in $L^2_\sigma$, 
where $L_{\sigma}^{2}$ denotes the subspace of all mean-free and divergence-free functions in $L^{2}_x$,    
we construct infinitely many global-in-time probabilistically strong and analytically weak solutions
to stochastic MHD \eqref{equa-SMHD} 
in the supertical regime with respect to the Lady\v{z}enskaja-Prodi-Serrin (LPS) criterion.  
In particular, 
the non-uniqueness is derived for the 3D stochastic MHD equations, 
which is sharp at one endpoint 
of the LPS criterion.  

\begin{theorem}  \label{Thm-Nonuniq-SMHD}
For any divergence-free   initial data $u_{0}, B_{0}\in L_{\sigma}^{2}$ 
and any  
$(\gamma,p)\in  [1, 2) \times[1, \infty]$, 
there exist infinitely many global-in-time
probabilistically strong and analytically weak solutions  
in the space $L^\gamma_t L^p_x$ to \eqref{equa-SMHD}.   
In particular, 
the uniqueness fails in the space $L^\gamma_t L^p_x$ 
with $(\gamma,p)\in  [1, 2) \times[1, \infty]$. 
\end{theorem}

\begin{remark}
$(i)$ To the best of our knowledge, 
Theorem \ref{Thm-Nonuniq-SMHD} provides the first example 
of global-in-time probabilistically strong 
solutions to the 3D stochastic MHD equations.  

$(ii)$ The non-uniqueness in Theorem \ref{Thm-Nonuniq-SMHD} is sharp at one endpoint of the LSP criterion.  
Actually,
the deterministic MHD is invariant under the scaling
\begin{equation}
	u(t,x)\mapsto \lambda u(\lambda^{2}t, \lambda x),
    \quad B(t,x)\mapsto \lambda B(\lambda^{2}t, \lambda x),
    \quad P(t,x)\mapsto \lambda^{2} P(\lambda^{2}t, \lambda x),
\end{equation}
which leads to the scaling-invariant space 
$L^\gamma_t L^p_x$
when
\begin{equation} \label{LPS}
	\frac{2}{\gamma}+\frac{3}{p}= 1. 
\end{equation} 
In particular, 
$L^2_t L^\infty_x$ is one endpoint space of LPS criteria. 
On one hand, 
it is well-known that weak solutions to NSE in the (sub)critical space $L_{t}^{\gamma}L_{x}^{p}$ with $2/\gamma+3/p\leq 1$
are unique and in the Leray-Hopf class (see \citep{CL-NSE1,ESS, S}). 
This criterion also holds for MHD systems, 
see Theorem \ref{uniq-sharp} below.  
On the other hand, 
Theorem \ref{Thm-Nonuniq-SMHD} 
implies the non-uniqueness of probabilistically strong solutions to  stochastic  MHD
in the space $L^\gamma_t L^\infty$ for any $1\leq \gamma<2$. 
Thus, the non-uniqueness 
is sharp as $L^2_t L^\infty_x$ is the critical LPS space. 

$(iii)$ 
Because solutions to stochastic MHD equations 
conserve the mean, 
the main result also holds for divergence-free solutions with non-zero mean. 
In this paper we focus on the case 
where initial data are mean free 
for simplicity. 
\end{remark}

Differently from the NSE, 
our proof utilizes intermittent flows which respect the geometry of MHD equations 
and leads to non-uniqueness 
in mixed Lebesgue spaces 
for more general stochastic MHD models 
with viscosity and resisitivity even larger than the
Lions exponent $5/4$, 
in which case, however, 
it is well-known that 
solutions to NSE and MHD are well-posed in $C_tL^2_x$ due to the classical result by Lions \cite{Lions} 
and \cite{W-MHD}.  
Furthermore, 
the relationship, 
via vanishing noise limit, 
between solutions to stochastic and deterministic MHD systems is 
also derived.

\subsection{Hyper-dissipative stochastic MHD: beyond the Lions exponent}  

Let us consider a general model of stochastic MHD
with high viscosity and resistivity 
\begin{eqnarray}\label{1.1}
	\left\{\begin{array}{lc}
		\mathrm{d} u+\big(\nu_{1}(-\Delta)^{\a}u+(u\cdot \nabla)u-(B\cdot \nabla )B\big) \mathrm{d} t+\nabla P \mathrm{d} t =\mathrm{d} W^{(1)}(t),\\
		\mathrm{d}  B+\big(\nu_{2}(-\Delta)^{\a}B+(u\cdot \nabla )B-(B\cdot \nabla )u\big)\mathrm{d} t =\mathrm{d} W^{(2)}(t),\\
		\div u=0, \quad\div B=0,\\
		u(0)=u_{0},\quad B(0)=B_{0},
	\end{array}\right.
\end{eqnarray}
where $\nu_1,\nu_2>0$, 
$W^{(i)}$ is as in \eqref{Wiener},  
$i=1,2$, 
the exponent $\a\in [1,3/2)$, 
and
$(-\Delta)^{\a}$ is the fractional Laplacian defined via Fourier transform on the flat torus
\begin{align*}
	\mathscr{F}\((-\Delta)^{\a}u\)(\xi)=|\xi|^{2\a}\mathscr{F} (u)(\xi),\quad\xi\in \Z^{3}.
\end{align*}

The solutions to \eqref{1.1} is 
taken in the probabilistically strong and analytically weak sense. 

\begin{definition}  \label{Def-Sol}
Given $T\in (0,\infty)$.
We say that $(u,B)$
	is a probabilistically strong
	and analytically weak solution to the SMHD system (\ref{1.1}) with the initial data $u_{0}$, $B_{0}\in L_{\sigma}^{2}$, $\mathbf{P}$-a.s. if
 the following holds:
\begin{enumerate}
  \item[(i)] $u, B \in  L^{2}(0, T; L^{2})$ ;\\
  \item[(ii)]  for all $t\in [0, T]$, $u(t, \cdot)$ and $B(t, \cdot)$ are divergence-free and have zero spatial mean;\\
  \item[(iii)]  for any divergence-free test function $\phi\in C_{0}^{\infty}(\T^{3})$ and any $t\in [0, T]$, 
	\begin{align*}
	&\langle u(t),\phi\rangle = \langle u_{0} ,\phi\rangle  +\int_{0}^{t}\left(\langle  -\nu_{1}(-\Delta)^{\a}u(s), \phi\rangle -\langle (u(s)\cdot\nabla)u(s),\phi\rangle +\langle (B(s)\cdot\nabla)B(s),\phi \rangle\right)  \mathrm{d}s
		+ \langle W^{(1)}(t), \phi \rangle,\\
	&\langle B(t),\phi\rangle
	= \langle B_{0} ,\phi\rangle+\int_{0}^{t}\left(\langle -\nu_{2}(-\Delta)^{\a}B(s),\phi \rangle-\langle (u(s)\cdot\nabla)B(s),\phi \rangle+\langle (B(s)\cdot\nabla)u(s),\phi\rangle\right) \mathrm{d}s
	+ \langle W^{(2)}(t),\phi\rangle.
	\end{align*}
\end{enumerate}
\end{definition}

It is well-known that high dissipation usually helps to establish well-posedness theory of (stochastic) PDEs. 
For the 3D NSE, 
one classical result by Lions \citep{Lions} 
is that 
energy solutions are well-posed in 
$C_tL^2_x \cap L^2_t H^1_x$ 
in the high dissipate regime where the viscosity exponent is larger than $5/4$.
Similar well-posedness result also holds for hyper-viscous and resistive MHD \cite{W-MHD}, 
and for the 3D stochastic hyper-dissipative NSE 
\cite{BJLZ23}.  

However, 
Theorem \ref{Thm-Nonuniq-Hyper} below shows that, even in the high-dissipative regime beyond the Lions exponent $5/4$,
the uniqueness would still fail in certain mixed Lebesgue spaces.

\begin{theorem}   \label{Thm-Nonuniq-Hyper}
Let $\alpha \in [1,3/2)$ and 
consider the superctitical regime 
\begin{align*}
   \mathcal{S}:=\left\{(s,\gamma,p)\in [0, 3)\times[1, \infty]\times[1, \infty]:0\leq s< \frac{2\a}{\ga}+\frac{2\a-2}{p}+1-2\a\right\}.
\end{align*}
Given any initial data $ u_{0}$, $B_{0}\in L_{\sigma}^{2} $  and any deterministic, divergence-free and mean-free smooth vector fields $ \ui,~ \bi\in  C_{0}^{\infty}((0,T]\times\T^{3})$. Then, 
for any $(s,\gamma,p)\in\mathcal{S}$, 
$r>1$ and $\epsilon>0 $, 
there exists a solution $ (u, B)$ to system (\ref{1.1}) with the initial data $ (u_{0}, B_{0})$ 
such that
	\begin{equation}\label{1.4}
		u-z_{1}, B-z_{2}\in \Lo^{2r}_{\om} L_{t}^{2}L_{x}^{2}\cap\Lo^{r}_{\om} L_{t}^{1}L_{x}^{2}\cap\Lo^{r}_{\om}L_{t}^{\ga}W_{x}^{s,p},
	\end{equation}
	 and $ (u,B) $ is close to $ (\ui+z_{1}, \bi+z_{2}):$
	\begin{equation}\label{1.5}
		\|\big(u-(\ui+z_{1}), B-(\bi+z_{2})\big)\|_{ \Lo^{r}_{\om} L_{t}^{1}L_{x}^{2}}+	\|\big(u-(\ui+z_{1}),B-(\bi+z_{2})\big)\|_{\Lo^{r}_{\om}L_{t}^{\ga}W_{x}^{s,p}}<\epsilon,
	\end{equation}
where  $z_{i}$, $i=1,2$, are the stochastic convolutions
associated to $W^{(i)}$ as in (\ref{1.4.1}) below.	

In particular,
the uniqueness fails for probabilistically strong and analytically weak solutions
to hyper-viscous and resistive MHD (\ref{1.1}) when $\alpha \in [1,\frac 32)$.
\end{theorem}

As for the classical MHD when $\alpha=1$,
$L^\gamma_t W^{s,p}_x$ is 
the scaling-invariant space for \eqref{1.1}
when
\begin{align}\label{LPS-2}
     \frac{2\alpha}{\gamma} + \frac{3}{p} = 2\alpha -1 + s.
\end{align} 
For the convenience of readers, 
{Figure 1} below shows the visual 
location of the non-unique $\mathcal{S}$ regime
when $s=0$
and the LPS critical line.
We note that $\mathcal{S}$ lies in the supercritical regime
of LPS criteria,
and the boundary of $\mathcal{S}$ contains one LPS endpoint $(\frac{2\alpha-1}{2\alpha}, 0)$. 
In view of Theorem \ref{Thm-Nonuniq-Hyper}, 
the uniqueness fails in the space
$L^{\gamma}_t L^\infty_x$
for any $1\leq \gamma <\frac{2\alpha}{2\alpha -1}$,
which is sharp with respect to the LPS scaling.

\begin{figure}[H]\label{fig}
	\centering
	\begin{tikzpicture}[scale=4,>=stealth,align=center]
		\small
		
		
		\draw[->] (0,0) --node[pos=1,left] {$\frac{1}{p}$} node[pos=0,below]{$0$} (0,1.35);
		\draw[->] (0,0) --node[pos=1,below] {$\frac{1}{\gamma}$} (1.35,0);
		
		
		\coordinate [label=below:$1$] (gamma1) at (1,0);
		\coordinate [label=left:$1$] (p1) at (0,1);
		\coordinate (11) at (1,1);

		\coordinate [label=left:$\frac{2\alpha -1}{3} $] (A) at (0,6/7); 
		\coordinate [label=below:$\frac{2\alpha-1}{2\alpha}$ \\ $\uparrow$ \\ \text{Endpoint space $L_t^{\frac{2\alpha}{2\alpha-1}}L_x^{\infty}$}] (B) at (6/7,0);
		\draw (A) --node[pos=0.5,below left]{$\;\nearrow$\\LPS\;} (B);
		\draw (p1) -- (11);
		\draw (11) -- (gamma1);
		\coordinate [label=left:$ \frac{2\alpha-1}{2\alpha-2} $] (B1) at (0,1.2);

		
		\draw[white, pattern =north west lines, pattern color = blue!50] (1/7,1) -- (11) -- (gamma1) -- (B) -- cycle;
		\draw (1/7,1) -- (11) -- (gamma1) -- (B) ;
		\draw[dashed, name path = l4] (B) --node[pos=0.6,above right] { \; $\mathcal{S}$\\ $\swarrow$ \;} (B1);
		
	\end{tikzpicture}
\caption{The case $\alpha \in [1,\frac{3}{2}), s=0$}
\label{fig:my_label}
\end{figure}
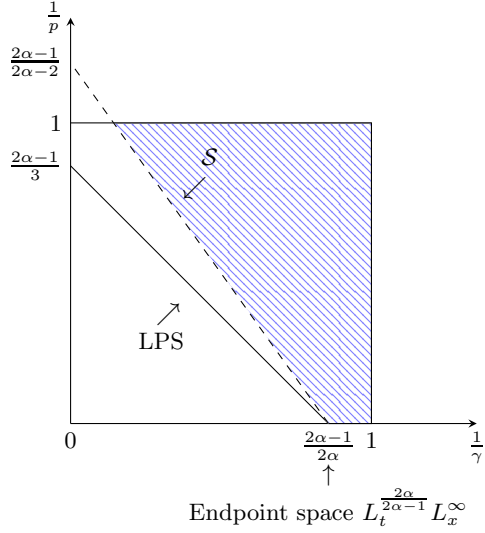

Moreover, in the (sub)critical LPS regime, 
we have the uniqueness result of solutions 
in Theorem \ref{uniq-sharp}  below.  
Let 
\begin{equation}\label{8.1}
	\begin{aligned}
		X_{(0,T)}^{s,\gamma,p}:=\left\{\begin{array}{ll}
			L^{\gamma}(0,T;W_{x}^{s,p}),\quad \text{if} ~ \frac{2\a}{\gamma}+\frac{3}{p}=2\a-1+s,~1\leq \gamma<\infty,~1\leq p\leq \infty,~s\geq 0, \\
			C([0,T];W_{x}^{s,p}),\quad \text{if} ~ \frac{3}{p}=2\a-1+s,~ \gamma=\infty,~1\leq p\leq \infty,~s\geq 0.
		\end{array} \right.
\end{aligned}\end{equation} 
We have 
\begin{theorem}\label{uniq-sharp} 
	Let $u$, $B\in X_{(0,T)}^{s,\gamma,p}$, $\mathbf{P}$-a.s., and  $(s,\gamma,p)$ satisfy  (\ref{LPS-2}) with $\a\geq 1$, $2\leq \gamma\leq \infty$, $1\leq p\leq \infty$, $s\geq 0$ and $0\leq \frac{1}{p}-\frac{s}{3}\leq \frac{1}{2}$. If $(u,B)$ is a solution to (\ref{1.1}) with $\nu_{1}=\nu_{2}$ in the sense of Definition \ref{Def-Sol}, then $(u,B)$ is unique and $u$, $B\in  C_{w}([0,T];L_{x}^{2})\cap L^{2}(0,T; \dot{H}_{x}^{\a})$, $\mathbf{P}$-a.s.
\end{theorem}

\subsection{Vanishing noise limit}
Our last result is concerned with the vanishing noise limit problem,
which relates both stochastic and deterministic solutions to MHD.

Let $\mathscr{D}_{MHD}^{\widetilde{\beta}}$, $\wt \beta>0$,
denote the set of all mean-free solutions
$(u,B)\in H_{t,x}^{\widetilde{\beta}}\times H_{t,x}^{\widetilde{\beta}}$
to the deterministic MHD model
	\begin{eqnarray}\label{1.3.1}
	\left\{\begin{array}{lc}
		\mathrm{d} u+\big(\nu_{1}(-\Delta)^{\a}u+(u\cdot \nabla)u-(B\cdot \nabla )B\big) \mathrm{d} t+\nabla P \mathrm{d} t =0,\\
		\mathrm{d}  B+\big(\nu_{2}(-\Delta)^{\a}B+(u\cdot \nabla )B-(B\cdot \nabla )u\big)\mathrm{d} t =0,\\
		\div u=0, \quad\div B=0,\\
		u(0)=u_{0},\quad B(0)=B_{0},
	\end{array}\right.
\end{eqnarray}
where $\alpha \in [1,3/2)$ and
$(u_{0}, B_{0})\in L_{\sigma}^{2}\times L_{\sigma}^{2}$. 
Let $\mathscr{A}^{r,\wt \beta}_{SMHD}$  denote the accumulation points,
where the limit is taken in the space $\Lo^{2r}_{\om}L_{t}^{2}L_{x}^{2}$,
of global solutions $(u^{(\epsilon_{n})}, B^{(\epsilon_{n})})$
with finite $r$-th moment to the stochastic MHD
(\ref{1.1}),
where $W^{(i)}$ is replaced by $\epsilon_{n} W^{(i)}$, $i=1,2$.

The following result shows that
the accumulation points of stochastic MHD contain all deterministic solutions
to MHD in the class $H_{t,x}^{\widetilde{\beta}}\times H_{t,x}^{\widetilde{\beta}}$ 
for some $\wt \beta>0$.
In particular, it contains the recently constructed solutions  
to (hyper-dissipative) MHD  in \cite{ZZL-MHD,ZZL-MHD-sharp}.

\begin{theorem}
[Vanishing noise limit] \label{Theorem Vanishing noise}.
	Let $\alpha\in [1,3/2)$, $r>1$ and $\wt \beta\in (0,1)$.
Then, we have
	\begin{equation}
		 \mathscr{D}_{MHD}^{\widetilde{\beta}}\subseteq  \mathscr{A}^{2r, \widetilde{\beta}}_{SMHD}.
    \end{equation}
	That is, for any mean-free weak solution $(u,B)\in H_{t,x}^{\widetilde{\beta}}\times H_{t,x}^{\widetilde{\beta}}$ to the deterministic MHD (\ref{1.3.1}), there exists
	a sequence of stochastic solutions $\left\{(u^{(\epsilon_{n})}, B^{(\epsilon_{n})})\right\}_{{n}}$ to the stochastic MHD (\ref{1.1}) with $\epsilon_{n}W^{(i)}$ replacing $W^{(i)}$, $i=1,2$,
	such that
	\begin{equation}\label{1.3.3}
		(u^{(\epsilon_{n})}, B^{(\epsilon_{n})})\rightarrow (u, B) \quad\text{strongly in } \ \
          \Lo^{2r}_{\om}L_{t}^{2}L_{x}^{2} \times \Lo^{2r}_{\om}L_{t}^{2}L_{x}^{2},
		\ \ as\ \epsilon_n \to 0.
	\end{equation}
\end{theorem}

\subsection{Comments on intermittent flows adapted to stochastic MHD}

As mentioned above, 
the specific geometry of MHD system 
requires a different construction of intermittent flows from those of NSE. 

More precisely, 
for the NSE, 
the intermittent Beltrami flows in \cite{BV-NS} and 
intermittent jets in \cite{BCV} have 3D intermittency, 
which can control the viscosity up to the Lions exponent. 

However, for MHD equations,  
the Maxwell component has anti-symmetric nonlinearity which is different from the symmetric one in the NS component.  
Thus, the NSE intermittent flows cannot be applied to MHD. 

Intermittent flows for MHD 
were first constructed by Beekie, Buckmaster and Vicol 
\cite{BBV-IMHD}
and have 1D spatial intermittency, 
which can control the fractional viscosity and resistivity $(-\Delta)^{\alpha}$ with $\alpha\in [0,1/2)$. 
Moreover, 
the intermittent flows in 
\cite{Yamazaki-3D-STMHD-well} 
have 2D spatial intermittency 
and can control $(-\Delta)^{\alpha}$ with $\alpha\in [0,1)$. 
For stronger intermittent flows for MHD system, we refer to \cite{MY24}. 
We also would like to mention the 
$L^\infty$ convex integration schemes 
recently developed for the ideal MHD in a series of works {\citep{FLS-IMHD2021-A, FLS-IMHD2021-ar, FLS-IMHD2018-ar}}. 

In \cite{ZZL-MHD,ZZL-MHD-sharp}, 
the viscosity and magnetic flows with 
 2D spatial intermittency and 1D temporal intermittency were constructed.  
The gained temporal intermittency is crucial to control the classical viscosity and resistivity 
$(-\Delta)$, 
and even can control the high dissipation 
$(-\Delta)^\alpha$ with $\alpha$ larger than the Lions exponent $5/4$.

The construction of temporal intermittency, on the other hand, gives rise to high temporal oscillations.
In order to balance such oscillations,
extra temporal correctors are used for both velocity and magnetic flows (see {(\ref{4.68})} below).
We also introduce 
temporal shifts in the temporal building blocks in 
\eqref{4.14} below,  
so that the asscociated temporal supports 
are disjoint. 
This permits to decouple different temporal interactions
and to avoid the complicate interaction estimates in 
the previous works \cite{ZZL-MHD,ZZL-MHD-sharp}. 

Because of the anti-symmetric nonlinearity
in the magnetic equation, 
the construction of magnetic amplitudes 
requires  a second Geometric Lemma \ref{Lemma Second Geometric} 
in small neighborhoods of the null matrix,  
which is different from the symmetric Geometric Lemma 
\ref{Lemma First Geometric} in the NSE case. 
This structural difference also leads to the construction of 
an additional matrix $\G$ (see \eqref{4.42} below) 
in the construction of velocity amplitudes, 
as well as different algebraic identities 
in Lemmas \ref{Lemma Magnetic amplitudes} and 
\ref{Lemma Velocity amplitudes}. 

Let us also mention that, 
handling prescribed initial data  
is quite delicate 
in convex integrations.  
Actually, 
in order to preserve the initial data of approximate 
relaxed solutions, 
temporal truncations of perturbations 
are used in the convex integration scheme. 
The input of noise makes the situation more complicated, 
and the associated Reynolds 
and magnetic stresses tend to explode 
in the super-norm 
near the initial time. 
In order to control such large errors, 
insipred by \cite{CDZ,HZZ-3}, 
we consider the Reynolds errors in two different 
temporal regimes: 
the regime near the initial time 
and the regime away from the initial time. 
The detailed analysis is contained 
in Subsection \ref{away from zero} and \ref{near zero} 
below.

\medskip 
The remainder of this 
paper is organized as follows. 
In Section \ref{Sec-Main-Iteration}, 
we present the main iterative estimates in the convex integration scheme, which is crucial to the construction of distinct probabilistically strong solutions. 
Then, Section \ref{Sec-Regul-Noise} treats the stochastic convolutions mainly in the hyper-dissipative case. 
Sections \ref{Sec-Vel-Mag} 
and \ref{Sec-Reynolds} constitute the main part of this paper, which contain the constructions of velocity and magnetic intermittent flows and of the Reynolds and magnetic stresses. 
The key iterative estimates are also verified there. 
In Section \ref{Sec-Proof-Main}, we prove the main results in Theorems \ref{Thm-Nonuniq-Hyper}, 
\ref{uniq-sharp}, and \ref{Theorem Vanishing noise}. 
Finally, the Appendix contains some standard tools in convex integration 
used in this paper.

\medskip
\paragraph{\bf Notations}
Let $\bbn_{0}:=\bbn\cup \{0\}$.
For $p \in[1,+\infty]$, $N\in \bbn_{0} \cup \{\infty\}$ and $s \in \mathbb{R}$. For simplicity, we use some shorthand notations as follows.
\begin{itemize}
\item $C_{c,\sigma}^{N}:=	C_{c,\sigma}^{N}(\[0, T\]\times \T^{3}),$ where the indices $c$ and $\sigma$ mean ``compact support in time'' and ``divergence-free'', respectively;\\
\item $L_x^p:=L_x^p(\mathbb{T}^3), \,\, H_x^s:=H_x^s(\mathbb{T}^3), \,\, W_x^{s, p}:=W_x^{s, p}(\mathbb{T}^3)$ represent usual Sobolev spaces;\\
\item $\fint_{\T^n} u \mathrm{~d} x :=\left|\T^n\right|^{-1} \int_{\T^n} u \mathrm{~d} x$;\\
\item $\langle\cdot,\cdot\rangle$ denotes the inner product in $L_{x}^{2}$ and also the duality between $H^{-\a}_{x}$ and $H^{\a}_{x}$;\\
\item $L^p(I; X)$ represents the space of integrable functions from $I$ to $X$, equipped with the usual $L^p$-norm (with the usual adaptation when $p=\infty$) for any Banach space $X$, and we write $L^p_{I}X :=L^p(I; X)$ for simplicity;\\
\item $C(I; X)$ stands for the space of continuous functions from $I$ to $X$
equipped with the supremum norm
and write $C_{I}X:= C(I; X)$ and
$C_{I,x}:=C_{I}C_{x}$;\\
\item When $I=[0, T]$ 
with $T\in (0, \infty)$, we simply
write $L^{p}_{t}:=L^{p}_{[0, T]}$, and $C_{t}:=C_{[0, T]}$;\\
\item We use the following notations of norms:
\begin{align*}
		\|u\|_{L_{t, x}^p}:=\|u\|_{L_{t}^pL_{ x}^p}, \quad
	\|u\|_{\dot{C}_{t, x}^N}:=\sum_{0 \leq m+|\zeta|= N}\|\partial_t^m \nabla^\zeta u\|_{C_{t, x}},\quad
	\|u\|_{C_{t, x}^N}:=\sum_{m=0}^{N}\|u\|_{\dot{C}_{t, x}^m},
\end{align*}
and
\begin{align*}
\|u\|_{W_{x}^{N, p}}:=\sum_{0 \leq |m| \leq N}\|\nabla^m u\|_{L_{x}^p},\quad
	\|u\|_{W_{t,x}^{N, p}}:=\sum_{0 \leq n+|\zeta| \leq N}\|\p_{t}^{n}\nabla^{\zeta} u\|_{L_{t}^pL_{x}^p},
\end{align*}
where $\zeta=\left(\zeta_1, \zeta_2, \zeta_3\right)$ is the multi-index and $\nabla^\zeta:=\partial_{x_1}^{\zeta_1} \partial_{x_2}^{\zeta_2} \partial_{x_3}^{\zeta_3}$;\\
\item We also use the product space $\mathcal{L}_{x}^{p}:=L_{x}^{p}\times L_{x}^{p}$. Similar notations apply to $\mathcal{L}_{t,x}^{p}$, 
$\mathcal{W}_{x}^{s,p}$, $\mathcal{H}_{t,x}^{s}$, $\mathcal{C}_{x}^{N}$, $\mathcal{C}_{t,x}^{N}$;\\
\item For any $T\in (0, \infty)$, $I\subset \[0, T \]$,
and $p$, $q\geq 1$,
let
\begin{equation}\label{1.6}
	\begin{aligned}
		\Lo^{p}_{\om}L^{q}_{I}X:=\left\{u: \om\times I\rightarrow X\mid \|u\|_{\Lo^{p}_{\om}L^{q}_{I}X}:=\(\mathbb{E}\|u\|_{L^{q}_{I}X}^{p}\)^{\frac{1}{p}}<\infty\right\},
	\end{aligned}
\end{equation}
and define $\Lo^{p}_{\om}C_{I}^{N}X$ and $\Lo^{p}_{\om} C_{I,x}^{N}$
in an analogous way;\\
\item For any vector fields $u$ and $v$, 
the corresponding second order tensor product is defined by
$$
u \otimes v:=(u_i v_j)_{1 \leq i, j \leq 3},
$$
and for any second-order tensor $A=(a_{i j})_{1 \leq i, j \leq 3}$ 
let
$$
\operatorname{div} A:=\(\sum_{j=1}^3 \partial_{x_j} a_{1 j}, \sum_{j=1}^3 \partial_{x_j} a_{2 j}, \sum_{j=1}^3 \partial_{x_j} a_{3 j}\)^{\top};
$$ 
\item We use the notation $a\lesssim b$ if there exists a constant $C>0$ such that $a\leq Cb$.
\end{itemize}

\section{Main iterative estimates} \label{Sec-Main-Iteration}

This section contains the main iterative estimates which is crucial in the proof of Theorem \ref{Thm-Nonuniq-Hyper}. 
For the convenience of simplicity, 
we consider the case where $\nu_{1}=\nu_{2}=\nu$ 
in \eqref{1.1}. 
The general case 
where $\nu_{1}\neq\nu_{2}$ 
in Theorems \ref{Thm-Nonuniq-SMHD}, \ref{Thm-Nonuniq-Hyper} and  \ref{Theorem Vanishing noise} can be proved analogously with only slight modifications. 

Let $z_{i}$, $i=1,2$, denote the stochastic convolutions of $W^{(i)}$:
\begin{equation}\label{1.4.1}
	\begin{aligned}
		&z_{1}(t)=e^{t\(-\nu(-\Delta)^{\a}-I\)}u_{0}+\int_{0}^{t}e^{(t-s)\(-\nu(-\Delta)^{\a}-I\)}\mathrm{d} W^{(1)}(s),\\
		&z_{2}(t)=e^{t\(-\nu(-\Delta)^{\a}-I\)}B_{0}+\int_{0}^{t}e^{(t-s)\(-\nu(-\Delta)^{\a}-I\)}\mathrm{d} W^{(2)}(s),\quad t\in [0, T].
	\end{aligned}
\end{equation}
Or, equivalently,
$z_{1}$ and $z_{2}$ solve the linear stochastic equations
\begin{eqnarray}\label{1.2}
	\left\{\begin{array}{lc}
		\mathrm{d} z_{1}+\(\nu(-\Delta)^{\a}+1\)z_{1} \mathrm{d} t=\mathrm{d} W^{(1)}(t),\\
		\mathrm{d}  z_{2}+\(\nu(-\Delta)^{\a}+1\)z_{2} \mathrm{d} t=\mathrm{d} W^{(2)}(t),\ \ t\in [0,T], \\
		\div\, z_{1}=0, \quad\div\, z_{2}=0, \\
		z_{1}(0)=u_{0}, \quad z_{2}(0)=B_{0}. 
	\end{array}\right.
\end{eqnarray}
For simplicity, set
\begin{align}
	&z_{1}^{u}(t):=e^{t\(-\nu(-\Delta)^{\a}-I\)}u_{0},\quad
	z_{2}^{B}(t):=e^{t\(-\nu(-\Delta)^{\a}-I\)}B_{0},\label{1.4.2}\\
	&Z_{i}(t):=\int_{0}^{t}e^{(t-s)\(-\nu(-\Delta)^{\a}-I\)}\mathrm{d} W^{(i)}(s),\quad i=1,2 \quad t\in [0, T], \label{1.4.3}
\end{align}
and thus
\begin{align*}
  & z_1(t) = z^u_1(t) + Z_1(t), \ \
    z_2(t) = z^B_2(t) + Z_2(t).
\end{align*}
We extend $z_{i}(t)$ to zero for $t<0$ and $t>T$, $i=1,2$. 

The frequency truncated versions of the stochastic convolutions are defined by 
\begin{equation}\label{2.3-1}
	Z_{i,q}:=\bbp_{\leq \la^{15}}~Z_{i},\ \ i=1,2,
\end{equation}
and define the approximate solutions to $z_1$ and $z_2$ by
\begin{align*}
   z_{1,q}:=z_{1}^{u}+Z_{1,q},\quad z_{2,q}:=z_{2}^{B}+Z_{2,q}. 
\end{align*}  
Note that by the Sobolev embedding $H_{x}^{2} \hookrightarrow L_{x}^{\infty}$  and (\ref{2.1}),
one has
 \begin{equation}\label{2.6}
	\|Z_{i,q}\|_{\Lo^{p}_{\om}C_{t}L_{x}^{\infty}}
	\leq
	\|Z_{i,q}\|_{\Lo^{p}_{\om}C_{t}H_{x}^{2}}
	\lesssim  (p-1)^{\frac 12}L,\quad i=1,2.
\end{equation}

Using the random shift
\begin{align}  \label{uB-z-shift}
	u = \Bu + z_{1},\ \ B = \Bb+ z_{2},
\end{align} 
we reformulate the original stochastic MHD system
\eqref{1.1} as follows 
\begin{eqnarray}   \label{1.3}
	\left\{\begin{array}{lc}
		\p_{t}\Bu+\nu(-\Delta)^{\a}\Bu-z_{1}+\big((\Bu+z_{1})\cdot \nabla \big)(\Bu+z_{1})-\big((\Bb+z_{2})\cdot \nabla \big)(\Bb+z_{2})+\nabla P  =0,\\
	\p_{t}\Bb+\nu(-\Delta)^{\a}\Bb-z_{2}+\big((\Bu+z_{1})\cdot \nabla \big)(\Bb+z_{2})-\big((\Bb+z_{2})\cdot \nabla \big)(\Bu+z_{1}) =0,\\
		\div \Bu=0,\quad\div \Bb=0 ,\\
		 \Bu(0)=0,\quad \Bb(0)=0,\quad t\in [0, T].
	\end{array}\right.
\end{eqnarray}
It is thus equivalent to construct non-unique solutions to 
the transformed equation \eqref{1.3}.

Our construction is based on the convex integration method.
Let $r>1$, $L\geq 1$ be as in Proposition \ref{Proposition stochastic} below,
$ c $ satisfy $ q80^{q} \leq c 85^{q}$, and $q\in\bbn_{0}$. Let $ a\in \bbn$, $b\in 2\bbn$, $\beta \in (0, 1]$, and let
\begin{equation}\label{2.2}
	\la := a^{b^{q}}, \quad \dq:=\lambda_{1}^{3\beta}\la^{-2\beta}, \quad
	\varsigma_{q}:=\la^{-30}.
\end{equation}
We take $a$, $b$ large enough and $\beta$, $\delta$ small such that
\begin{equation}\label{2.2-2}
	a\geq (85\cdot 8\cdot 80rL^{2})^{c},\quad
	b>\frac{16000}{\varepsilon}, \quad
	 \beta \leq \frac{5}{2b^2},
	 \quad\delta<\frac{1}{60},
\end{equation}
where  $\varepsilon \in \bbq_{+}$ is sufficiently small such that for the given $(s,\gamma,p) \in S$,
\begin{equation}\label{2.3}
	\varepsilon\leq \frac{1}{20}\min \left\{\frac{3}{2}-\a, \frac{2\a}{\gamma}+\frac{2\a-2}{p}-(2\a-1)-s \right\},
\ \  b(2-2\a-10\varepsilon),\ \  b\varepsilon \in \bbn.
\end{equation}

Now we consider the relaxed random MHD-Reynolds system: for each integer $ q\in \bbn $,
\begin{eqnarray}\label{2.5}
	\left\{\begin{array}{lc}
	\p_{t}\Bu_{q}+\nu(-\Delta)^{\a}\Bu_{q}-z_{1,q}+\div\big((\Bu_{q}+z_{1,q})\otimes(\Bu_{q}+z_{1,q})-(\Bb_{q}+z_{2,q})\otimes(\Bb_{q}+z_{2,q})\big)+\nabla P_{q}  =\div \ru^{u},\\
		\p_{t}\Bb_{q}+\nu(-\Delta)^{\a}\Bb_{q}-z_{2,q}+\div\big((\Bu_{q}+z_{1,q})\otimes(\Bb_{q}+z_{2,q})-(\Bb_{q}+z_{2,q})\otimes(\Bu_{q}+z_{1,q})\big) =\div \ru^{\Bb},\\
		\div \Bu_{q}=0,\quad \div \Bb_{q}=0 ,\\
			\Bu_{q}(0)=0,\quad \Bb_{q}(0)=0.
	\end{array}\right.
\end{eqnarray}
For $t<0$ and $t>T$, we let $\Bu_{q}$, $\Bb_{q}$, $\ru^{u}$, $\ru^{\Bb}$ and $z_{iq}$ be zero, $i=1,2$.

Given $(u_0, B_0) \in L^2_\sigma \times L^2_\sigma$,
let $M\geq 1$ be such that
\begin{equation}\label{in}
\|u_{0}\|_{L^{2}_{x}}+\|B_{0}\|_{L^{2}_{x}}\leq M.
\end{equation}

The following main iterative estimates are crucial in the proof of Theorem \ref{Thm-Nonuniq-Hyper}.
As in \citep{HZZ-2},
we control the moment of relaxed solutions in a quantitative way
in the convex integration scheme.

\begin{proposition}(Main iterative estimates) \label{Proposition Main iteration}
Let  $\a\in[1, 3/2) $, $(s,p,\ga)\in \mathcal{S}$, 
$r> 1$, $T\in(0, \infty)$. 
Let $M^*$ be large enough 
and $(\Bu_{q},\Bb_{q},\ru^{u},\ru^{B})$ be a  solution to
the relaxed random MHD-Reynolds system (\ref{2.5}) for some $ q\in\bbn_{0}  $  satisfying that for any $ 0\leq N\leq 4 $ and any $ m\geq 1 $,
\begin{align}
		&\|(\Bu_{q}, \Bb_{q})\|_{\Lo^{m}_{\om}\mathcal{C}_{t,x}^{N}}
		\leq
		\la^{{2\a(N+1)+2}}\(8(5N+22)mL^{2}80^{q-1}\)^{(5N+22)80^{q-1}},\label{2.11}\\
		&\|(\Ru^{u}_{q},\Ru^{B}_{q})\|_{\Lo^{m}_{\om} C_{t}\mathcal{L}_{x}^{1}}\leq \la^{4\a+5}(8 mL^{2}80^{q})^{80^{q}},\label{2.12}\\
		&\|(\Ru^{u}_{q},\Ru^{B}_{q})\|_{\Lo^{r}_{\om} L_{t}^{1}\mathcal{L}_{x}^{1}}\lesssim \dqq, \label{2.13}
\end{align}	
where the implicit constants are independent of $q, m$ and $r$.
Then, there exists an $(\ft)$-adapted process
$(\Bu_{q+1}, \Bb_{q+1},\Ru^{u}_{q+1},\Ru^{B}_{q+1}) $
which solves (\ref{2.5}) and obeys estimates (\ref{2.11})-(\ref{2.13}) at level $ q+1 $.
In addition, it holds that
\begin{align}
	&\|(\Bu_{q+1}-\Bu_{q},\Bb_{q+1}-\Bb_{q})\|_{\Lo^{2r}_{\om} L_{t}^{2}\mathcal{L}_{x}^{2}}
	\leq M^* \delta_{q+1}^{\frac{1}{2}},	\label{2.14}\\
	&\|(\Bu_{q+1}-\Bu_{q},\Bb_{q+1}-\Bb_{q})\|_{\Lo^{r}_{\om} L_{t}^{1}\mathcal{L}_{x}^{2}}
	\leq\dqqq^{1/2},\label{2.15}\\
	&\|(\Bu_{q+1}-\Bu_{q},\Bb_{q+1}-\Bb_{q})\|_{\Lo^{r}_{\om}L_{t}^{\ga}\mathcal{W}_{x}^{s,p}}
	\leq\dqqq^{1/2}.\label{2.16}
\end{align}	
\end{proposition}

\medskip
The proof of Proposition \ref{Proposition Main iteration}
occupies Sections \ref{Sec-Vel-Mag} and \ref{Sec-Reynolds} below.
As a preparation,
let us mollify the velocity and magnetic fields
in order to avoid the loss of derivatives in the convex integration scheme.

Let $\varrho \in C_c^{\infty}(\bbr^3 ; \mathbb{R}_{+})$, supp $\varrho\subset B_1(0)$, and $\vartheta \in C_c^{\infty}(\bbr ; \mathbb{R}_{+})$ with ${\rm supp}\, \vartheta \subset[0,1]$.
Define the mollifiers by
\begin{align}  \label{3.1}
	\varrho_{\ell} (\cdot):= \ell^{-3} \varrho(\cdot / \ell),\ \
	\vartheta_{\ell} (\cdot):= \ell^{-1} \vartheta(\cdot / \ell),
\end{align}
where 
 \begin{equation}\label{3.2}
 	\begin{aligned}
 		\ell:=\la^{-60}.
	\end{aligned}
\end{equation}
By  (\ref{2.2}) and (\ref{2.2-2}), we have
\begin{equation}\label{3.2-1}
	\begin{aligned}
	(8\cdot 80rL^{2}80^{q})^{80^{q}}\leq \la, \quad	\ell^{-260}<\lambda_{q+1}^{\varepsilon}.
	\end{aligned}
\end{equation}

Define the spacial-temporal mollifications of $ \Bu_{q}$, $\Bb_{q}$,  $z_{1,q}$, $z_{2,q}$, $\Ru^{u}_{q}$, $\Ru^{B}_{q}$  as follows:
\begin{align}
       & \Bu_{\ell}:=(\Bu_{q} *_x \varrho_{\ell}) *_t \vartheta_{\ell},
        \quad \Bb_{\ell}:=(\Bb_{q} *_x \varrho_{\ell}) *_t \vartheta_{\ell}, \label{3.4-1}\\ 
         &\Ru^{u}_{\ell}:=(\Ru^{u}_{q} *_x \varrho_{\ell}) *_t \vartheta_{\ell},\quad
		\Ru^{B}_{\ell}:=(\Ru^{B}_{q} *_x \varrho_{\ell}) *_t \vartheta_{\ell},  \label{3.4-2} \\ 
		& z_{1,\ell}:=z^{u}_{1,\ell}+Z_{1,\ell},\quad z_{2,\ell}:=z^{B}_{2,\ell}+Z_{2,\ell}, \label{3.4-00} \\  
           &Z_{i,\ell}:=(Z_{i,q} *_x \varrho_{\ell}) *_t \vartheta_{\ell},\ \  i=1,2, \label{3.4-02}
\end{align} 
with 
\begin{align}  \label{3.4-00-1}
          &z^{u}_{1,\ell}:=(z^{u}_{1} *_x \varrho_{\ell}) *_t \vartheta_{\ell},\quad z^{B}_{2,\ell}:=(z^{B}_{2} *_x \varrho_{\ell}) *_t \vartheta_{\ell}, 
\end{align}
Then by (\ref{2.5}), $ (\Bu_{\ell}, \Bb_{\ell},\Ru^{u}_{\ell},\Ru^{B}_{\ell}) $ satisfies
\begin{align}\label{3.5}
	\left\{\begin{array}{lc}
	\p_{t}\Bu_{\ell}+\nu(-\Delta)^{\a}\Bu_{\ell}-z_{1,\ell}+\div\big((\Bu_{\ell}+z_{1,\ell})\otimes(\Bu_{\ell}+z_{1,\ell})-(\Bb_{\ell}+z_{2,\ell})\otimes(\Bb_{\ell}+z_{2,\ell})\big)+\nabla P_{\ell}\\
	\quad=\div( \Ru^{u}_{\ell}+\Ru^{u}_{com1}),\\
		\p_{t}\Bb_{\ell}+\nu(-\Delta)^{\a}\Bb_{\ell}-z_{2,\ell}+\div\big((\Bu_{\ell}+z_{1,\ell})\otimes(\Bb_{\ell}+z_{2,\ell})-(\Bb_{\ell}+z_{2,\ell})\otimes(\Bu_{\ell}+z_{1,\ell})\big) \\
		\quad=\div (\Ru^{B}_{\ell}+\Ru^{B}_{com1}),\\
		\div \Bu_{\ell}=0,\quad\div \Bb_{\ell}=0 ,\\
		 \Bu_{\ell}(0)=0,\quad\Bb_{\ell}(0)=0.
	\end{array}\right.
\end{align}
where the traceless symmetric commutator stress $ \Ru^{u}_{com1} $ and the skew-symmetric commutator stress $ 	\Ru^{B}_{com1} $ are of form
\begin{equation}\label{3.7}
	\begin{aligned}
	\Ru^{u}_{com1}:=&(\Bu_{\ell}+z_{1,\ell})\mathring{\otimes}(\Bu_{\ell}+z_{1,\ell})-(\Bb_{\ell}+z_{2,\ell})\mathring{\otimes}(\Bb_{\ell}+z_{2,\ell})\\
	&-\big((\Bu_{q}+z_{1,q})\mathring{\otimes}(\Bu_{q}+z_{1,q})-(\Bb_{q}+z_{2,q})\mathring{\otimes}(\Bb_{q}+z_{2,q})\big)*_x \varrho_{\ell} *_t \vartheta_{\ell},
	\end{aligned}
\end{equation}
\begin{equation}\label{3.8}
	\begin{aligned}
		\Ru^{B}_{com1}:=&(\Bb_{\ell}+z_{2,\ell})\otimes(\Bu_{\ell}+z_{1,\ell})-(\Bu_{\ell}+z_{1,\ell})\otimes(\Bb_{\ell}+z_{2,\ell})\\
		&-\big((\Bb_{q}+z_{2,q})\otimes(\Bu_{q}+z_{1,q})-(\Bu_{q}+z_{1,q})\otimes(\Bb_{q}+z_{2,q})\big)*_x \varrho_{\ell} *_t \vartheta_{\ell},
	\end{aligned}
\end{equation}
and the pressure $	P_{\ell} $ is given by
\begin{equation}\label{3.6}
	\begin{aligned}
		P_{\ell}:=\left(P_{q} *_x \varrho_{\ell}\right) *_t \vartheta_{\ell}
   - \frac 13 (|\Bu_{\ell}+z_{1,\ell}|^{2}-|\Bb_{\ell}+z_{2,\ell}|^{2})
   +\frac 13\big((|\Bu_{q}+z_{1,q}|^{2}-|\Bb_{q}+z_{2,q}|^{2})*_x \varrho_{\ell}\big) *_t \vartheta_{\ell}.
	\end{aligned}
\end{equation}
For simplicity, we also use the notations
\begin{align}  \label{Def-Jq-wtJq}
 J_{q}:=\|(\Ru^{u}_{q},\Ru^{B}_{q})\|_{C_{t} \mathcal{L}_{x}^{1}} \quad
  \widetilde{J}_{q}:=\|(\Ru^{u}_{q},\Ru^{B}_{q})\|_{L_{t}^{1}\mathcal{L}_{x}^{1}} .	
\end{align}

\section{Regularity of stochastic convolutions}   \label{Sec-Regul-Noise}

This section mainly contains the estimates of the moments of stochastic convolutions.

\begin{proposition}\label{Proposition stochastic}
	Suppose that $ Tr(G_{i}G_{i}^{*})<\infty $, $ i=1,2$.
Then, for any $\delta\in (0, 1/2)$, $\kappa\in [0, \a)$, $p\geq 2$,
\begin{align}\label{2.1}
			\mathbb{E}\left[ \|Z_{i}\|^{p}_{C_{t}H_{x}^{\kappa+4}}
			+ \|Z_{i}\|^{p}_{C_{t}^{1/2-\delta} L_{x}^{2}}\right]
			\leq (p-1)^{p/2}L^{p},
\end{align}
	where $ L \geq 1 $ depends on $ Tr(G_{i}G_{i}^{*}) $, $ \delta $, $\kappa$, $\a$, and $\nu$ are independent of $ p $.
\end{proposition}

We refer to \citep{HZZ-2} for the classical case where $\alpha =1$. 
The hyper-dissipative case where $\alpha>1$ can be proved in a similar manner.
For completeness,
we present the detailed proof below.

We shall use the following estimates of fractional heat-semigroups.
For the case $\alpha\in (0,1]$
we refer to  \citep{DV}.

\begin{proposition}\label{Proposition convolution}
		Fix $\a\in (0,\infty)$. We have that for all $1\leq m \leq n\leq \infty$, $\beta\geq 0$ and $t\geq 0$, 
	\begin{equation}\label{p1.5}
		\begin{aligned}
			\|(-\Delta)^{\frac{\beta}{2}}e^{-t(-\Delta)^{\a}}\|_{L^{m}(\T^{3})\rightarrow L^{n}(\T^{3})}\lesssim t^{-\frac{3}{2\a}\(\frac{1}{m}-\frac{1}{n}\)-\frac{\beta}{2\a}},
	\end{aligned}
\end{equation}
where the implicit constant only depends on $\a$, $\beta$, $m$, $n $.
	\end{proposition}

\begin{proof}
Let us first note that $(-\Delta)^{\frac{\beta}{2}}e^{t\(-(-\Delta)^{\a}\)}\phi=K_{t}*\phi$
for all $\phi\in L^{m}(\T^{3})$, where
\begin{align*}
	K_{t}(x)=\sum_{k\in \Z^{3}}|k|^{\beta}e^{-t|k|^{2\a}+ik\cdot x}.
\end{align*}
Let $H_{t}(x):=\F^{-1}(|k|^{\beta}e^{-t|k|^{2\a}})(x)$, $x\in \bbr^{3}$, and let $H^{per}_{t}(x)$ denote the periodic function generted by $H_{t}$
\begin{align*}
	H^{per}_{t}(x):=\sum_{k\in \Z^{3}}H_{t}(x+2\pi k).
\end{align*}
For $x\in \bbr^{3}$ fixed, by the Poisson summation formula, 
\begin{align*}
	H^{per}_{t}(x)=(\frac{1}{2\pi})^{3}\sum_{k\in \Z^{d}}\F H_{t}(k)e^{ik\cdot x}.
\end{align*}
It follows that  $H^{per}_{t}(x)=(\frac{1}{2\pi})^{3}	K_{t}(x)$, and so
\begin{align*}
\|K_{t}\|_{L^{r}(\T^{3})}=(2\pi)^{3}\|	H^{per}_{t}\|_{L^{r}(\T^{3})}\leq (2\pi)^{3}\|H_{t}\|_{L^{r}(\bbr^{3})},\quad 1\leq r \leq \infty.
\end{align*}
Note that  by the change of variables, and $\F^{-1}(|\eta|^{\beta}e^{-|\eta|^{2\a}}) \in L^{p}(\bbr^{3})$, $p\in [1, \infty]$ (see \citep[Lemma 2.2]{MYZ} ),
\begin{align*}
\|H_{t}\|_{L^{r}(\bbr^{3})}
	&= \|\F^{-1}(-|\xi|^{\beta}e^{t(-|\xi|^{2\a})})\|_{L^{r}(\bbr^{3})}\\
	&=t^{-\frac{\beta}{2\a}-\frac{3}{2\a}}\|\F^{-1}(-|\eta|^{\beta}e^{-|\eta|^{2\a}})(t^{-\frac{1}{2\a}}\cdot)\|_{L^{r}(\bbr^{3})}\\
	&\lesssim
	t^{-\frac{\beta}{2\a}-\frac{3}{2\a}+\frac{3}{2r\a}}\|\F^{-1}(-|\eta|^{\beta}e^{-|\eta|^{2\a}})\|_{L^{r}(\bbr^{3})}\\
	&\lesssim
	t^{-\frac{\beta}{2\a}-\frac{3}{2\a}+\frac{3}{2r\a}}.
\end{align*}
Thus, using  Young's inequality,  $1+\frac{1}{n}=\frac{1}{r}+\frac{1}{m}$, we come to 
\begin{align*}
	\|(-\Delta)^{\beta}e^{t\(-(-\Delta)^{\a}\)}\|_{L^{m}(\T^{3})\rightarrow L^{n}(\T^{3})}
	\leq \|K_{t}\|_{L^{r}(\T^{3})}
	&\lesssim
	t^{-\frac{\beta}{2\a}-\frac{3}{2\a}+\frac{3}{2r\a}}=t^{-\frac{3}{2\a}\(\frac{1}{m}-\frac{1}{n}\)-\frac{\beta}{2\a}}
	\end{align*}
and finish the proof.
\end{proof}

We are now ready to prove Proposition \ref{Proposition stochastic}.\\

{\bf Proof of Proposition \ref{Proposition stochastic}.}
We mainly prove (\ref{2.1}) for $i=1$, as the arguments in the case where  $i=2$ are similar.
For  $0\leq t_{2}\leq t_{1}\leq T$, $|t_{1}-t_{2}|\leq 1$, we have
\begin{align}\label{1.5.1}
	Z_{1}(t_{1})-Z_{1}(t_{2})
	=\int_{t_{2}}^{t_{1}}e^{(t_{1}-s)\(-\nu(-\Delta)^{\a}-I\)}\mathrm{d} W^{(1)}(s)+\int_{0}^{t_{2}}e^{(t_{1}-s)\(-\nu(-\Delta)^{\a}-I\)}-e^{(t_{2}-s)\(-\nu(-\Delta)^{\a}-I\)}\mathrm{d} W^{(1)}(s).
\end{align}

For  the first term on the R.H.S of (\ref{1.5.1}), by It\^{o}'s isometry and (\ref{p1.5}),  
\begin{flalign*}
	\mathbb{E}\|\int_{t_{2}}^{t_{1}}e^{(t_{1}-s)\(-\nu(-\Delta)^{\a}-I\)}\mathrm{d} W^{(1)}(s)\|_{L_{x}^{2}}^{2}
	&=\sum_{k\in \bbn} (c_{k}^{(1)})^{2}\int_{t_{2}}^{t_{1}}	\|e^{(t_{1}-s)\(-\nu(-\Delta)^{\a}-I\)}e_{k}\|_{L_{x}^{2}}^{2}\mathrm{d} s \\
	& \lesssim Tr(G_{1}G_{1}^{\ast})\(t_{1}-t_{2}\).
\end{flalign*}
Regarding the second  term of (\ref{1.5.1}),
we have
\begin{flalign*}
	&\quad\mathbb{E}\|\int_{0}^{t_{2}}e^{(t_{1}-s)\(-\nu(-\Delta)^{\a}-I\)}-e^{(t_{2}-s)\(-(-\Delta)^{\a}-I\)}\mathrm{d} W^{(1)}(s)\|_{L_{x}^{2}}^{2}\\
	&=\sum_{k\in \bbn} (c_{k}^{(1)})^{2}\int_{0}^{t_{2}}\|\int_{t_{2}-s}^{t_{1}-s}\p_{r}e^{r\(-\nu(-\Delta)^{\a}-I\)}e_{k}\mathrm{d} r\|_{L_{x}^{2}}^{2}\mathrm{d} s\\
	&\lesssim\sum_{k\in \bbn} (c_{k}^{(1)})^{2}\int_{0}^{t_{2}}e^{-2(t_{2}-s)}\left|\int_{t_{2}-s}^{t_{1}-s}\|\(-\nu(-\Delta)^{\a}-I\)e^{r\(-\nu(-\Delta)^{\a}\)}e_{k}\|_{L_{x}^{2}}\mathrm{d} r\right|^{2}\mathrm{d} s,
\end{flalign*}
which, via (\ref{p1.5}), can be bounded by
\begin{flalign*}
	\sum_{k\in \bbn} (c_{k}^{(1)})^{2}\int_{0}^{t_{2}}e^{-2(t_{2}-s)}\left|\int_{t_{2}-s}^{t_{1}-s}r^{-1}+1\mathrm{d} r\right|^{2}\mathrm{d} s.
\end{flalign*}

It follows that for $\gamma\in (0, 1/2)$, 
\begin{flalign*}
	&\quad\mathbb{E}\|\int_{0}^{t_{2}}e^{(t_{1}-s)\(-\nu(-\Delta)^{\a}-I\)}-e^{(t_{2}-s)\(-\nu(-\Delta)^{\a}-I\)}\mathrm{d} W^{(1)}(s)\|_{L_{x}^{2}}^{2}\\
	&\lesssim \sum_{k\in \bbn} (c_{k}^{(1)})^{2}\int_{0}^{t_{2}}e^{-2(t_{2}-s)}	
	\left|\int_{t_{2}-s}^{t_{1}-s}r^{-1}+1 \mathrm{d} r\right|^{2}\mathrm{d} s\\
	&\lesssim \sum_{k\in \bbn} (c_{k}^{(1)})^{2}\(\int_{0}^{t_{2}}e^{-2(t_{2}-s)}	\(t_{2}-s\)^{-2\gamma}
	\left|\int_{t_{2}-s}^{t_{1}-s}r^{-1+\gamma}\mathrm{d} r\right|^{2}\mathrm{d} s+(t_{1}-t_{2})^{2}\)\\	
	&\lesssim  Tr(G_{1}G_{1}^{\ast})(t_{1}-t_{2})^{2\gamma}.
\end{flalign*}

Thus,  we obtain
\begin{flalign*}
	\mathbb{E}\|Z_{1}(t)-Z_{1}(s)\|^{2}_{L_{x}^{2}}\lesssim Tr(G_{1}G_{1}^{\ast})(t_{1}-t_{2})+Tr(G_{1}G_{1}^{\ast})(t_{1}-t_{2})^{2\gamma}.
\end{flalign*}
The implicit constants above depend on $\a$, $\nu$ and $\gamma$,
but are independent of time. 
Using Gaussianity  and $|t_{1}-t_{2}|\leq 1$, we have
\begin{flalign*}
	\mathbb{E}\|Z_{1}(t_{1})-Z_{1}(t_{2})\|_{L_{x}^{2}}^{p}
    \leq& (p-1)^{p/2}\(\mathbb{E}\|Z_{1}(t_{1})-Z_{1}(t_{2})\|^{2}_{L_{x}^{2}}\)^{p/2}  \\
    \lesssim& (p-1)^{p/2}Tr(G_{1}G_{1}^{\ast})^{2/p}(t_{1}-t_{2})^{p\gamma}.
\end{flalign*}
By Kolmogrov's continuity criterion (\citep{Da Prato}), there exists a constant $L$ such that for any $\delta\in (0, 1/2)$, 
\begin{flalign*}
	\mathbb{E}\|Z_{i}\|^{p}_{C_{t}^{1/2-\delta} L_{x}^{2}}
	\leq \frac{1}{2}(p-1)^{p/2}L^{p}.
\end{flalign*}

Next, we consider the $H_{x}^{\kappa+4}$-norm.
By It\^{o}'s isometry and (\ref{p1.5}), we obtain that for $\kappa\in (0, \a)$,
\begin{flalign*}
	\mathbb{E}\|\int_{t_{2}}^{t_{1}}e^{(t_{1}-s)\(-\nu(-\Delta)^{\a}-I\)}\mathrm{d} W^{(1)}(s)\|_{\dot{H}_{x}^{\kappa+4}}^{2}
	&=\sum_{k\in \bbn} (c_{k}^{(1)})^{2}\int_{t_{2}}^{t_{1}} \|e^{(t_{1}-s)\(-\nu(-\Delta)^{\a}-I\)}e_{k}\|_{\dot{H}_{x}^{\kappa+4}}^{2}\mathrm{d} s \\
	&\lesssim Tr(G_{1}G_{1}^{\ast})(t_{1}-t_{2})^{1-\frac{\kappa}{\a}}.
\end{flalign*}
Moreover, by (\ref{p1.5}),  we have  for $\kappa'\in (\frac{\kappa}{2\a}, \frac{1}{2})$,
\begin{flalign*}
	&\quad\mathbb{E}\|\int_{0}^{t_{2}}e^{(t_{1}-s)\(-\nu(-\Delta)^{\a}-I\)}-e^{(t_{2}-s)\(-\nu(-\Delta)^{\a}-I\)}\mathrm{d} W^{(1)}(s)\|_{\dot{H}_{x}^{\kappa+4}}^{2}\\
	&=\sum_{k\in \bbn} (c_{k}^{(1)})^{2}\int_{0}^{t_{2}}\|\int_{t_{2}-s}^{t_{1}-s}\p_{r}e^{r\(-\nu(-\Delta)^{\a}-I\)}e_{k}\mathrm{d} r\|_{\dot{H}_{x}^{\kappa+4}}^{2}\mathrm{d} s\\
	&\lesssim \sum_{k\in \bbn} (c_{k}^{(1)})^{2}\int_{0}^{t_{2}}e^{-2(t_{2}-s)}\left|\int_{t_{2}-\mathrm{d} s}^{t_{1}-s}r^{-1-\frac{\kappa}{2\a}}+r^{-\frac{\kappa}{2\a}}\mathrm{d} r\right|^{2}\mathrm{d} s\\
	&\lesssim \sum_{k\in \bbn} (c_{k}^{(1)})^{2}\int_{0}^{t_{2}}e^{-2(t_{2}-s)}	\big((t_{2}-s)^{-2\kappa'}	\left|\int_{t_{2}-s}^{t_{1}-s}r^{-1-\frac{\kappa}{2\a}+\kappa'}\mathrm{d} r\right|^{2}+(t_{1}-t_{2})^{2(1-\frac{\kappa}{2\a})}\big)
	\mathrm{d} s\\
	&\lesssim  Tr(G_{1}G_{1}^{\ast})(t_{1}-t_{2})^{2(\kappa'-\frac{\kappa}{2\a})}.
\end{flalign*}
Hence, plugging the above estimates into (\ref{1.5.1}) we obtain
\begin{flalign*}
	\mathbb{E}\|Z_{1}(t)-Z_{1}(s)\|^{2}_{H_{x}^{\kappa+4}}\lesssim Tr(G_{1}G_{1}^{\ast})(t_{1}-t_{2})^{1-\frac{\kappa}{\a}}+Tr(G_{1}G_{1}^{\ast})(t_{1}-t_{2})^{2(\kappa'-\frac{\kappa}{2\a})},
\end{flalign*}
where all the constants depends only on $\a$, $\kappa'$, $\nu$ and $\kappa$ but are independent of time. Using Gaussianity and $|t_{1}-t_{2}|\leq 1$, we lead to
\begin{flalign*}
	\mathbb{E}\|Z_{1}(t_{1})-Z_{1}(t_{2})\|_{H_{x}^{\kappa+4}}^{p}
    \leq& (p-1)^{p/2}\big(\mathbb{E}\|Z_{1}(t_{1})-Z_{1}(t_{2})\|^{2}_{H_{x}^{\kappa+4}}\big)^{p/2} \\
    \lesssim& (p-1)^{p/2}Tr(G_{1}G_{1}^{\ast})^{2/p}(t_{1}-t_{2})^{p(\kappa'-\frac{\kappa}{2\a})}.
\end{flalign*}
Thus, applying Kolmogrov's continuity criterion we obtain \eqref{2.1} 
and finish the proof.  
\hfill$\square$

\section{Velocity and magnetic flows} \label{Sec-Vel-Mag}

In this section,
we construct new velocity and magnetic flows
and verify the corresponding iterative estimates in Proposition \ref{Proposition Main iteration}.

\subsection{Intermittent flows}

In the sequel we use  the  intermittent flows  introduced in \citep{ZZL-MHD-sharp} which are 
indexed by four parameters $r_{\perp}, \lambda, \tau$ and $\sigma$ :
\begin{equation}\label{4.1}
	\begin{aligned}
r_{\perp}:=\lambda_{q+1}^{2-2 \alpha-10 \varepsilon},\quad \lambda:=\lambda_{q+1}, \quad
\tau:=\lambda_{q+1}^{2 \alpha},\quad \sigma:=\lambda_{q+1}^{2 \varepsilon},
	\end{aligned}
\end{equation}
where $\varepsilon$ sastidfies (\ref{2.3}).

\medskip
\paragraph{\bf $\bullet $ {Spatial building blocks.}}
Let $\Phi: \bbr \rightarrow \bbr$ be a smooth cut-off function supported on $[-1,1]$ and normalize $\Phi$ such that $\phi:=-\frac{d^2}{d x^2} \Phi$ satisfies
\begin{equation}\label{4.2}
	\begin{aligned}
\frac{1}{2 \pi} \int_{\bbr} \phi^2(x) \mathrm{d} x=1 .
	\end{aligned}
\end{equation}
Define the  rescaled cut-off functions  by
\begin{flalign*}
\phi_{r_{\perp}}(x):=r_{\perp}^{-\frac{1}{2}} \phi\left(\frac{x}{r_{\perp}}\right), \quad \Phi_{r_{\perp}}(x):=r_{\perp}^{-\frac{1}{2}} \Phi\left(\frac{x}{r_{\perp}}\right) .
\end{flalign*}
 To avoid an abuse of notation, we periodize $\phi_{r_{\perp}}$ and $\Phi_{r_{\perp}}$ so that they are treated as periodic functions defined on $\T$.
Define the intermittent velocity shear flows  by
\begin{equation}\label{4.3}
	\begin{aligned}
W_{(k)}:=\phi_{r_{\perp}}(\lambda r_{\perp} N_{\Lambda} k \cdot x) k_1, \quad k \in \Lambda_u \cup \Lambda_B,
	\end{aligned}
\end{equation}
and the intermittent magnetic shear flows by
\begin{equation}\label{4.4}
	\begin{aligned}
D_{(k)}:=\phi_{r_{\perp}}(\lambda r_{\perp} N_{\Lambda} k \cdot x) k_2, \quad k \in \Lambda_B.
\end{aligned}
\end{equation}
Here, $N_{\Lambda}$ is given by (\ref{6.3}) below, $\(k, k_1, k_2\)\subseteq \bbr^{3}$ are the orthonormal bases as in Geometric Lemmas \ref{Lemma First Geometric} and \ref{Lemma Second Geometric}, and  $\Lambda_u, \Lambda_B$ are the wave vector sets. In particular, $\left\{W_{(k)}, D_{(k)}\right\}$ are $\(\mathbb{T} /(\lambda r_{\perp})\)^3$-periodic.

For the sake of simplicity, we set
\begin{equation}\label{4.5}
	\begin{aligned}
\phi_{(k)}(x):=\phi_{r_{\perp}}(\lambda r_{\perp} N_{\Lambda} k \cdot x), \quad \Phi_{(k)}(x):=\Phi_{r_{\perp}}(\lambda r_{\perp} N_{\Lambda} k \cdot x),
\end{aligned}
\end{equation}
and rewrite
\begin{align} \label{4.6}
	  W_{(k)}=\phi_{(k)} k_1,\ k \in \Lambda_u \cup \Lambda_B,  \quad
     {\rm and}  \quad  D_{(k)}=\phi_{(k)} k_2,\  k \in \Lambda_B.
\end{align}
Then, $W_{(k)}$ and $D_{(k)}$ are mean-zero on $\mathbb{T}^3$.

Define the corresponding potentials  by
	\begin{align}
	W_{(k)}^c & :=\frac{1}{\lambda^2 N_{\Lambda}^2} \Phi_{(k)} k_1, \quad k \in \Lambda_u \cup \Lambda_B, \label{4.8} \\
	D_{(k)}^c & :=\frac{1}{\lambda^2 N_{\Lambda}^2} \Phi_{(k)} k_2, \quad k \in \Lambda_B. \label{4.9}
\end{align}

The following lemma contains the analytic estimates of shear flows.
\begin{lemma}\label{Lemma spacial building blocks} (\citep{BBV-IMHD}, Estimates of intermittent shear flows). For any $p \in[1, +\infty]$ and $N \in \bbn$, one has
\begin{equation}\label{4.10}
	\begin{aligned}
\|\nabla^N \phi_{(k)}\|_{L_x^p}+\|\nabla^N \Phi_{(k)}\|_{L_x^p} \lesssim r_{\perp}^{\frac{1}{p}-\frac{1}{2}} \lambda^N .
\end{aligned}
\end{equation}
In particular,
	\begin{align}
	&\|\nabla^N W_{(k)}\|_{ L_x^p}+\lambda^2\|\nabla^N W_{(k)}^c\|_{ L_x^p} \lesssim r_{\perp}^{\frac{1}{p}-\frac{1}{2}} \lambda^N, \quad k \in \Lambda_u \cup \Lambda_B, \label{4.11}\\
   &\|\nabla^N D_{(k)}\|_{ L_x^p}+\lambda^2\|\nabla^N D_{(k)}^c\|_{ L_x^p} \lesssim r_{\perp}^{\frac{1}{p}-\frac{1}{2}} \lambda^N, \quad k \in \Lambda_B.\label{4.12}
\end{align}
The implicit constants above are independent of the parameters $r_{\perp}$ and $\lambda$.
\end{lemma}

\paragraph{\bf $\bullet$ Temporal building blocks.}
In order to treat the hyper-viscous case
of large  $\alpha$, particularly,  beyond the Lions exponent $5 / 4$,
we shall use the temporal intermittency in building blocks.

Let $\{g_{k}\}_{k\in \Lambda_u\cup\Lambda_B}\subset C_c^{\infty}([0, T])$ be any cut-off function such that
\begin{equation}\label{4.14}
	\begin{aligned}
\fint_0^T g^{2}_{k}(t) \mathrm{d} t=1 ,\quad k\in \Lambda_u\cup\Lambda_B,
\end{aligned}
\end{equation}
and $g_{k}$, $g_{k'}$ have disjoint temporal supports if $k\neq k'$.
The existence of $\{g_{k}\}_{k\in \Lambda_u\cup\Lambda_B}$
can be guaranteed by choosing $g_{k}=g(t-\a_{k})$,
where $g\in C_c^{\infty}([0, T])$ with very small support,  
and $\{\a_{k}\}_{k\in \Lambda_u\cup\Lambda_B}$ are the temporal shifts 
such that the supports of  $\{g_{k}\}$  are disjoint.

Then, rescale the cut-off function $g_{k}$ by
\begin{equation}\label{4.15}
	\begin{aligned}
g_{k,\tau}(t):=\tau^{\frac{1}{2}} g_{k}(\tau t),\quad k\in \Lambda_u\cup\Lambda_B.
\end{aligned}
\end{equation}
We then  periodize  $g_{k,\tau}$ such that it is treated as a periodic function defined on $[0, T]$ (still denoted by $ g_{k,\tau}$). Moreover, let
\begin{equation}\label{4.16}
	\begin{aligned}
h_{k,\tau}(t):=\int_0^t\left(g_{k,\tau}^2(s)-1\right) \mathrm{d} s,\quad t\in [0, T],\quad k\in \Lambda_u\cup\Lambda_B,
\end{aligned}
\end{equation}
and
\begin{equation}\label{4.17}
	\begin{aligned}
g_{(k)}:=g_{k,\tau}(\sigma t), \quad h_{(k)}(t):=h_{k,\tau}(\sigma t),\quad   \sigma\in \bbn.
\end{aligned}
\end{equation}
One  has the following estimates of the temporal building blocks.

\begin{lemma} \label{Lemma temporal building blocks}(\citep{ZZL-MHD,CL-transport equation,CL-NSE1}, Estimates of temporal intermittency). For $\gamma \in[1,+\infty]$, $M \in \mathbb{N}$, we have
\begin{equation}\label{4.18}
	\begin{aligned}
\|\partial_t^M g_{(k)}\|_{L_t^\gamma} \lesssim \sigma^M \tau^{M+\frac{1}{2}-\frac{1}{\gamma}},\quad k\in \Lambda_u\cup\Lambda_B,
\end{aligned}
\end{equation}
where the implicit constants are independent of $\tau$ and $\sigma$. Moreover, it holds
\begin{equation}\label{4.19}
	\begin{aligned}
\|h_{(k)}\|_{L_t^{\infty}} \lesssim 1,\quad
\|h_{(k)}\|_{C_t^{N}}\lesssim (\sigma\tau)^{N},\quad  N\geq1,\quad k\in \Lambda_u\cup\Lambda_B.
\end{aligned}
\end{equation}
\end{lemma}

\subsection{Amplitudes}
In order to decrease the effects of the old Reynolds and magnetic stresses,
we further introduce the  amplitudes of magnetic and  velocity perturbations.

\medskip
\paragraph{\bf $\bullet$ Magnetic amplitudes.}
Let $\chi:[0,+\infty) \rightarrow \mathbb{R}$ be a smooth cut-off function such that
\begin{equation}\label{4.20}
	\begin{aligned}
\chi(z)=\left\{\begin{array}{ll}
	1, & 0 \leq z \leq 1, \\
	z, & z \geq 2,
\end{array}\right.
\end{aligned}
\end{equation}
and
\begin{equation}\label{4.21}
	\begin{aligned}
\frac{1}{2} z \leq \chi(z) \leq 2 z\ \  \text { for }\ z \in(1,2) .
\end{aligned}
\end{equation}

Let
\begin{equation}\label{4.22}
	\begin{aligned}
\varrho_{B}(t, x):=2 \varepsilon_{B}^{-1} \dqq \chi\Big(\frac{|\Ru^{B}_{\ell}(t, x)|}{\dqq}\Big),
\end{aligned}
\end{equation}
where $\varepsilon_{B}$ is the small radius in Geometric Lemma \ref{Lemma Second Geometric}. Note that
${\Ru^{B}_{\ell}}/{\varrho_{B}} \in B_{\varepsilon_{B}}(0)$.

Define the amplitudes of the magnetic perturbations by
\begin{equation}\label{4.24}
	\begin{aligned}
a_{(k)}(t,x):=\varrho_{B}^{1/2}(t,x)\gamma_{(k)}\Big(\frac{-\Ru^{B}_{\ell}(t,x)}{\varrho_{B}(t,x)}\Big), \quad k\in \Lambda_B,
\end{aligned}
\end{equation}
where $\gamma_{(k)}$ is the smooth function in Geometric Lemma \ref{Lemma Second Geometric}.

\begin {lemma} [Magnetic amplitudes]   \label{Lemma Magnetic amplitudes}
It holds that
	\begin{align*} 
	\sum_{k \in \Lambda_B} a_{(k)}^2 g_{(k)}^2(D_{(k)} \otimes W_{(k)}-W_{(k)} \otimes D_{(k)})
	&= -\Ru^{B}_{\ell}+\sum_{k \in \Lambda_B} a_{(k)}^2 g_{(k)}^2 \bbp_{\neq 0}(D_{(k)} \otimes W_{(k)}-W_{(k)} \otimes D_{(k)}) \notag\\
	&\quad +\sum_{k \in \Lambda_B} a_{(k)}^2(g_{(k)}^2-1) \fint_{\mathbb{T}^3} D_{(k)} \otimes W_{(k)}-W_{(k)} \otimes D_{(k)} \mathrm{d} x,
\end{align*}
where $\bbp_{\neq 0}$ is the spatial projection onto nonzero Fourier modes.
Moreover, for $0 \leq N \leq 9$ and $k \in \Lambda_B$, 
\begin{equation}\label{4.26}
	\begin{aligned}
\|a_{(k)}\|_{C_{t, x}^N}
 \lesssim \ell^{-6N-7}(1+J_{q}^{N+2}),\quad
 \|a_{(k)}\|_{L_{t}^{2}L_{x}^{2}}
 \lesssim \delta_{q+1}^{\frac{1}{2}}+\|\Ru^{B}_{q}\|^{\frac{1}{2}}_{L_{t}^{1}L_{x}^{1}},
\quad \|a_{(k)}\|_{\Lo^{2r}_{\om}L_{t}^{2}L_{x}^{2}}
 \lesssim \delta_{q+1}^{\frac{1}{2}}, 
\end{aligned}
\end{equation}
where $J_q$ is given by \eqref{Def-Jq-wtJq},
and the implicit constants are deterministic and independent of $q$.
\end{lemma}

\begin{proof} 
The algebraic identity  was proved in \citep{ZZL-MHD-sharp}.
Below we focus on the proof of
(\ref{4.26}). Let us start with the following estimates
\begin{align}
	&\|\varrho_{B}\|_{\dot{C}_{t,x}^{N}}	
	\lesssim
	\left\{\begin{array}{ll}
		\ell^{-4}(1+\J),\quad N= 0, \\
	\ell^{-5N-1}(1+\J^{N}), \quad 1\leq N\leq9,
	\end{array}\right.\label{4.28}\\
	&\|\varrho_{B}^{-1}\|_{\dot{C}_{t,x}^{N}}
\lesssim
\left\{\begin{array}{ll}
	\delta_{q+1}^{-1}, \quad N= 0 ,\\
	 \ell^{-6N-1}(1+\J^{N}), \quad 1\leq N\leq9,
\end{array}\right.\label{4.29}\\
	&\|\varrho_{B}^{\frac{1}{2}}\|_{\dot{C}_{t,x}^{N}}
\lesssim
\left\{\begin{array}{ll}
	\ell^{-2}(1+\J^{\frac{1}{2}}), \quad N= 0, \\
	\ell^{-6N-1}(1+\J^{N}), \quad 1\leq N\leq9,
\end{array}\right.\label{4.30}\\
&\big\|\gamma_{(k)}\big(\frac{-\Ru^{B}_{\ell}}{\varrho_{B}}\big)\big\|_{\dot{C}_{t,x}^{N}} \lesssim  \ell^{-6N-5}(1+	\J^{N+1}),\quad 0\leq N\leq9.\label{4.31}	
\end{align}

To this end, by (\ref{4.22}), we first note that
\begin{align}\label{4.32-1}	\|\varrho_{B}\|_{C_{t,x}}
		& \lesssim \dqq\|1+\dqq^{-1}\Ru^{B}_{\ell}\|_{C_{t,x}}
		\lesssim \dqq+\|\Ru^{B}_{\ell}\|_{C_{t,x}}.
\end{align}
Using the Sobolev embedding $ W^{4,1}(\T^{3}) \hookrightarrow L^{\infty}(\T^{3}) $ and mollification estimates we obtain
\begin{equation}\label{4.27}
	\begin{aligned}
		\|\Ru^{B}_{\ell}\|_{C^{N}_{t,x}}\lesssim\ell^{-N-4}\|\Ru^{B}_{q}\|_{C_{t} L_{x}^{1}}\lesssim \ell^{-N-4}\J.
\end{aligned}\end{equation}
Then, plugging (\ref{4.27}) into (\ref{4.32-1}) with $N=0$ we have
\begin{equation}\label{4.32}
	\begin{aligned}
		\|\varrho_{B}\|_{C_{t,x}}
	   \lesssim \dqq+\ell^{-4}\J
	   \lesssim \ell^{-4}\(1+\J\),
	\end{aligned}\end{equation}
and
\begin{equation}\label{4.33}
	\begin{aligned}
	\|\varrho_{B}^{\frac{1}{2}}\|_{C_{t,x}}
		&\lesssim\big\|\big[\dqq(1+\dqq^{-1}\Ru^{B}_{\ell})\big]^{\frac{1}{2}}\big\|_{C_{t,x}} \lesssim \dqq^{\frac{1}{2}}+\|\Ru^{B}_{\ell}\|^{\frac{1}{2}}_{C_{t,x}}
		\lesssim \ell^{-2}(1+\J^{\frac{1}{2}}).
\end{aligned}\end{equation}
It also follows from (\ref{4.20}) that $\varrho_{B}\geq  \varepsilon_{B}^{-1} \dqq $,  and  so $\|\varrho_{B}^{-1}\|_{C_{t,x}}\lesssim \dqq ^{-1}$. Hence, the first estimates in (\ref{4.28})-(\ref{4.30}) are verified.

Next, we prove the second estimates in (\ref{4.28})-(\ref{4.30}). 
For every $1\leq  N\leq 9$, 
by (\ref{6.9}) below,  
\begin{equation}\label{4.34}
	\begin{aligned}
\|\varrho_{B}\|_{\dot{C}_{t,x}^{N}}
		&\lesssim \dqq\left[\dqq^{-1} \ell^{-N-4}\|\Ru^{B}_{q}\|_{C_{t} L_{x}^{1}}+\dqq^{-N} \ell^{-4(N-1)}\|\Ru^{B}_{q}\|_{C_{t} L_{x}^{1}}^{N-1}\ell^{-N-4}\|\Ru^{B}_{q}\|_{C_{t} L_{x}^{1}} \right]\\
		&\lesssim  \ell^{-5N-1}(1+\J^{N}).
\end{aligned}\end{equation}
Then we apply (\ref{6.10}) in the Appendix to the function $ \zeta_{1}(x)=x^{\frac{1}{2}},x\in [\varepsilon_{B}^{-1} \dqq, \infty)$  and use (\ref{4.34}) to derive 
\begin{align}\label{4.35}
	\|\varrho_{B}^{\frac{1}{2}}\|_{\dot{C}_{t,x}^{N}}
		&\lesssim \|\zeta_{1}\|_{\dot{C}^{1}_{[\varepsilon_{B}^{-1} \dqq, \infty)}}\|\varrho_{B}\|_{\dot{C}_{t,x}^{N}}+\|D\zeta_{1}\|_{C^{N-1}_{[\varepsilon_{B}^{-1} \dqq, \infty)}}\|\varrho_{B}\|^{N}_{\dot{C}_{t,x}^{1}}\notag\\
		&\lesssim \dqq^{-\frac{1}{2}} \ell^{-5N}(1+\J^{N})
		+ \dqq^{-\frac{1}{2}-N}\ell^{-6N}(1+\J^{N})\notag\\
		&\lesssim
		 \ell^{-6N-1}(1+\J^{N}).
	\end{align}
Similarly, we apply (\ref{6.10}) below to the function $ \zeta_{2}(x)=x^{-1}, x\in [\varepsilon_{B}^{-1} \dqq, \infty)$  and use  (\ref{4.27}) to obtain
  \begin{align}\label{4.36}	\|\varrho_{B}^{-1}\|_{\dot{C}_{t,x}^{N}}
  		&\lesssim \|\zeta_{2}\|_{\dot{C}^{1}_{[\varepsilon_{B}^{-1} \dqq, \infty)}}\|\varrho_{B}\|_{\dot{C}_{t,x}^{N}}+\|D\zeta_{2}\|_{C^{N-1}_{[\varepsilon_{B}^{-1} \dqq, \infty)}}\|\varrho_{B}\|^{N}_{\dot{C}_{t,x}^{1}}\notag\\
  	&\lesssim \ell^{-6N-1}(1+\J^{N}).
  \end{align} 
The above estimates yield (\ref{4.28})-(\ref{4.30}), as claimed.

Regarding the  bound \eqref{4.31}, using (\ref{6.9}) again we have
\begin{align*}		\big\|\gamma_{(k)}\big(\frac{-\Ru^{B}_{\ell}}{\varrho_{B}}\big)\big\|_{\dot{C}_{t,x}^{N}}
		&\lesssim \|\gamma_{(k)}(\cdot)\|_{\dot{C}^{1}}\|\varrho_{B}^{-1}\Ru^{B}_{\ell}\|_{\dot{C}_{t,x}^{N}}+\|D\gamma_{(k)}(\cdot)\|_{C^{N-1}}\|\varrho_{B}^{-1}\Ru^{B}_{\ell}\|^{N-1}_{C_{t,x}}\|\varrho_{B}^{-1}\Ru^{B}_{\ell}\|_{\dot{C}_{t,x}^{N}}\notag\\
		&\lesssim \|\varrho_{B}^{-1}\Ru^{B}_{\ell}\|_{\dot{C}_{t,x}^{N}}.
\end{align*}
Then, by (\ref{4.22}), (\ref{4.27}) and  (\ref{4.36}),  
\begin{align}\label{4.38}	 \|\varrho_{B}^{-1}\Ru^{B}_{\ell}\|_{\dot{C}_{t,x}^{N}}
		 &\lesssim \sum_{j=1}^{N} \|\varrho_{B}^{-1}\|_{\dot{C}_{t,x}^{j}} \|\Ru^{B}_{\ell}\|_{\dot{C}_{t,x}^{N-j}}+\|\varrho_{B}^{-1}\|_{C_{t,x}} \|\Ru^{B}_{\ell}\|_{\dot{C}_{t,x}^{N}}\notag\\
		 &\lesssim \sum_{j=1}^{N} \ell^{-6j-1}(1+\J^{j})\ell^{-(N-j)-4}\J+\dqq^{-1}\ell^{-N-4}\J\notag\\
        &\lesssim \ell^{-6N-5}(1+\J^{N+1}),
\end{align}
which yields (\ref{4.31}), as claimed.

We are now ready to verify (\ref{4.26}). Note that,
\begin{align}\label{4.39}
\|a_{(k)}\|_{\dot{C}_{t,x}^{N}}
&\lesssim \sum_{j=1}^{N}\|\varrho_{B}^{1/2}\|_{\dot{C}_{t,x}^{j}}	\big\|\gamma_{(k)}\big(\frac{-\Ru^{B}_{\ell}}{\varrho_{B}}\big)\big\|_{\dot{C}_{t,x}^{N-j}}
+\|\varrho_{B}^{1/2}\|_{C_{t,x}}	\big\|\gamma_{(k)}\big(\frac{-\Ru^{B}_{\ell}}{\varrho_{B}}\big)\big\|_{\dot{C}_{t,x}^{N}}\notag\\
&\lesssim \sum_{j=1}^{N} \ell^{-6j-1}(1+\J^{j}) \ell^{-6(N-j)-5}(1+\J^{(N-j)+1})+\ell^{-2}(1+\J^{\frac{1}{2}}) \ell^{-6N-5}(1+\J^{N+1})\notag\\
&\lesssim \ell^{-6N-7}(1+\J^{N+2}).
\end{align}
This verifies the first inequality of (\ref{4.26}).

Regarding the second inequality of (\ref{4.26}), using (\ref{4.24}) and (\ref{6.4})  we obtain
\begin{equation*}
	\begin{aligned}
		\|a_{(k)}\|_{L_{t}^{2}L_{x}^{2}}
		\lesssim \|\varrho_{B}\|^{1/2}_{L_{t}^{1}L_{x}^{1}}
		\|\gamma_{(k)}\|_{C(B_{\varepsilon_{B}}(0))}
		\lesssim \delta_{q+1}^{\frac{1}{2}}+\|\Ru^{B}_{q}\|^{\frac{1}{2}}_{L_{t}^{1}L_{x}^{1}},
	\end{aligned}
\end{equation*}
which, via (\ref{2.13}), yields
\begin{equation*} 
	\begin{aligned}
		\|a_{(k)}\|_{\Lo^{2r}_{\om}L_{t}^{2}L_{x}^{2}}
		\lesssim \delta_{q+1}^{\frac{1}{2}}+\|\Ru^{B}_{q}\|^{\frac{1}{2}}_{\Lo^{r}_{\om}L_{t}^{1}L_{x}^{1}}
		\lesssim
		\delta_{q+1}^{\frac{1}{2}}.
	\end{aligned}
\end{equation*}
Therefore, the proof of Lemma \ref{Lemma Magnetic amplitudes} is completed.
\end{proof}

\medskip
\paragraph{\bf $\bullet$ Velocity amplitudes.}
Proceeding as in \citep{BBV-IMHD,BV-NS, ZZL-MHD-sharp}, we use an additional matrix
\begin{equation}\label{4.42}
	\begin{aligned}
\G:=\sum_{k \in \Lambda_{B}} a_{(k)}^2 \fint_{\T^3}\left( W_{(k)} \otimes W_{(k)}-D_{(k)} \otimes D_{(k)}\right) \mathrm{d} x,
\end{aligned}
\end{equation}
in order to cancel the old stress.
In view of estimate (\ref{4.26}), we have
\begin{equation}\label{4.49}
	\begin{aligned}
		\|\G\|_{C^{N}_{t,x}}
		\lesssim \sum_{k \in
			\Lambda_{B}}\|a_{(k)}^2\|_{C^{N}_{t,x}}
		\lesssim
	\ell^{-6N-14}(1+J_{q}^{N+4}),
\end{aligned}\end{equation}
and
\begin{equation}\label{4.49-2}
	\begin{aligned}
		\|\G\|_{\Lo^{r}_{\om}L_{t}^{1}L_{x}^{1}}
		\lesssim \sum_{k \in
			\Lambda_{B}}\|a_{(k)}\|^{2}_{\Lo^{2r}_{\om}L_{t}^{2}L_{x}^{2}}
		\lesssim
		\dqq.
\end{aligned}\end{equation}

Define the amplitudes of the velocity perturbations  by
\begin{equation}\label{4.45}
	\begin{aligned}
		a_{(k)}(t,x):=\varrho_{u}^{1/2}\gamma_{(k)}\big(\mathrm{Id} -\frac{\Ru^{u}_{\ell}+\G}{\varrho_{u}}\big), \quad k\in \Lambda_u,
	\end{aligned}
\end{equation}
where
\begin{equation}\label{4.43}
	\begin{aligned}
		\varrho_{u}(t, x):=2 \varepsilon_{u}^{-1} \dqq \chi\big(\frac{|\Ru^{u}_{\ell}(t, x)+\G|}{\dqq}\big),
	\end{aligned}
\end{equation}
with $\varepsilon_{u}$ being the small radius in Geometric Lemma \ref{Lemma First Geometric}. Note that
$\frac{\Ru^{u}_{\ell}+\G}{\varrho_{u}}\in B_{\varepsilon_{u}}(0)$.

\begin {lemma} [Velocity amplitudes]     \label{Lemma Velocity amplitudes}
It holds that
\begin{equation*} 
	\begin{aligned}
		 \sum_{k \in \Lambda_u} a_{(k)}^2 g_{(k)}^2(W_{(k)} \otimes W_{(k)})
		= & \varrho_{u}\mathrm{Id} -(\Ru^{u}_{\ell}+\G)+\sum_{k \in \Lambda_u} a_{(k)}^2 g_{(k)}^2 \bbp_{\neq 0}(W_{(k)} \otimes W_{(k)}) \\
		& +\sum_{k \in \Lambda_u} a_{(k)}^2(g_{(k)}^2-1) \fint_{\T^3} W_{(k)} \otimes D_{(k)}-W_{(k)} \otimes D_{(k)} \mathrm{d} x.
	\end{aligned}
\end{equation*}
 Moreover, for any $0 \leq N \leq 9$, $k \in \Lambda_u$,
\begin{equation}\label{4.47}
	\begin{aligned}
	\quad\|a_{(k)}\|_{C_{t, x}^N}
	\lesssim \ell^{-21N-22}(1+\J^{5N+6}),
	\ 
  \|a_{(k)}\|_{L_{t}^{2}L_{x}^{2}}
	 \lesssim
	\delta_{q+1}^{\frac{1}{2}}+\|\Ru^{u}_{\ell}\|_{L_{t}^{1}L_{x}^{1}}^{\frac{1}{2}}+\|\G\|_{L_{t}^{1}L_{x}^{1}}^{\frac{1}{2}},\ 
 \|a_{(k)}\|_{\Lo^{2r}_{\om}L_{t}^{2}L_{x}^{2}}
	\lesssim \delta_{q+1}^{\frac{1}{2}} ,
	\end{aligned}
\end{equation}
where the implicit constants are deterministic and independent of $q$.
\end{lemma}
\begin{proof}
The first identity was proved previously in \citep{ZZL-MHD-sharp}.
Let us focus on the proof of (\ref{4.47}).
Similaly to (\ref{4.28})-(\ref{4.31}), we first  prove that
\begin{align}
		&\|\varrho_{u}\|_{\dot{C}_{t,x}^{N}}\lesssim\left\{\begin{array}{ll}
		\ell^{-4}(1+\J),  \quad N= 0, \\
	\ell^{-20N-1}(1+J_{q}^{5N}),  1 \leq N \leq 9,
	\end{array}\right.\label{4.50}\\
	&\|\varrho_{u}^{-1}\|_{\dot{C}_{t,x}^{N}}\lesssim\left\{\begin{array}{ll}
		\dqq^{-1},  N= 0, \\
	 \ell^{-21N-1}(1+\J^{5N}), \quad 1 \leq N \leq 9,
	\end{array}\right.\label{4.51}\\
&\|\varrho_{u}^{\frac{1}{2}}\|_{\dot{C}_{t,x}^{N}}\lesssim\left\{\begin{array}{ll}
	\ell^{-7}(1+\J^{2}), \quad N= 0, \\
	 \ell^{-21N-1}(1+\J^{5N}), \quad 1 \leq N \leq 9,
\end{array}\right.\label{4.52}\\
&\big\|\gamma_{(k)}\big(\mathrm{Id} -\frac{\Ru^{u}_{\ell}+\G}{\varrho_{u}}\big)\big\|_{\dot{C}_{t,x}^{N}} \lesssim  \ell^{-15-21N}(1+\J^{4N+3}),\quad 0 \leq N \leq 9.\label{4.53}	
\end{align}
To this end, by (\ref{4.43}), we first note that
	\begin{align}
		\|\varrho_{u}\|_{C_{t,x}}
		&\lesssim \dqq\|1+\dqq^{-1}(\Ru^{u}_{\ell}+\G)\|_{C_{t,x}}\lesssim \dqq+	\|\Ru^{u}_{\ell}\|_{C_{t,x}}+\|\G\|_{C_{t,x}}.
\end{align}
Using the Sobolev embedding $ W^{4,1}(\T^{3}) \hookrightarrow L^{\infty}(\T^{3}) $ again and mollification estimates, we obtain
\begin{equation}\label{4.48}
	\begin{aligned}
		\|\Ru^{u}_{\ell}\|_{C^{N}_{t,x}}\lesssim\ell^{-N-4}\|\Ru^{u}_{q}\|_{\C L_{x}^{1}}
		\lesssim\ell^{-N-4}\J.
\end{aligned}\end{equation}
Then, by (\ref{4.49}) and (\ref{4.48}),  
\begin{equation}\label{4.54}
	\begin{aligned}
		\|\varrho_{u}\|_{C_{t,x}} \lesssim \dqq+\ell^{-4}J_{q}+\ell^{-14}(1+J_{q}^{4})
		\lesssim \ell^{-14}(1+J_{q}^{4}),
\end{aligned}\end{equation}
and
\begin{equation}\label{4.55}
	\begin{aligned}
		\|\varrho_{u}^{\frac{1}{2}}\|_{C_{t,x}}
		&\lesssim
		\big\|\big[\dqq\big(1+\dqq^{-1}(\Ru^{u}_{\ell}+\G)\big)\big]^{\frac{1}{2}}\big\|_{C_{t,x}}\\
		&\lesssim \dqq^{\frac{1}{2}}+	\|\Ru^{u}_{\ell}\big\|^{\frac{1}{2}}_{C_{t,x}}+\|\G\|^{\frac{1}{2}}_{C_{t,x}}\\
		&\lesssim
		\ell^{-7}(1+\J^{2}).
\end{aligned}\end{equation}
Moreover, from (\ref{4.20}), it follows that  $\|\varrho_{u}^{-1}\|_{C_{t,x}}\lesssim \dqq ^{-1}$.
Thus, the first estimates of (\ref{4.50})-(\ref{4.52}) are verified.

Next, for the second estimates of (\ref{4.50})-(\ref{4.52}), 
by (\ref{6.9}),  for every $1 \leq N \leq 9$,
\begin{align}\label{4.56}
		\|\varrho_{u}\|_{\dot{C}_{t,x}^{N}}
		&\lesssim \dqq\left[\|\chi(\dqq^{-1}\cdot)\|_{\dot{C}^{1}}\|\Ru^{u}_{\ell}+\G\|_{\dot{C}_{t,x}^{N}}+\|D\chi(\dqq^{-1}\cdot)\|_{C^{N-1}}\|\Ru^{u}_{\ell}+\G\|^{N-1}_{C_{t,x}}\|\Ru^{u}_{\ell}+\G\|_{\dot{C}_{t,x}^{N}}\right]\notag\\
		 &\lesssim
	 \ell^{-20N-1}(1+J_{q}^{5N}).
\end{align}
Moreover, arguing as in (\ref{4.35}) and (\ref{4.36}) we get 
\begin{align}\label{4.57}
		\|\varrho_{u}^{\frac{1}{2}}\|_{\dot{C}_{t,x}^{N}}
		&\lesssim \|\zeta_{3}\|_{\dot{C}^{1}_{ [\varepsilon_{u}^{-1} \dqq, \infty)}}\|\varrho_{u}\|_{\dot{C}_{t,x}^{N}}+\|D\zeta_{3}\|_{C^{N-1}_{ [\varepsilon_{u}^{-1} \dqq, \infty)}}\|\varrho_{u}\|^{N}_{\dot{C}_{t,x}^{1}}\notag\\
	&\lesssim \ell^{-21N-1}(1+\J^{5N}), 
\end{align}
 and
	\begin{align}\label{4.58}
		\|\varrho_{u}^{-1}\|_{\dot{C}_{t,x}^{N}}
		&\lesssim \|\zeta_{4}\|_{\dot{C}^{1}_{[\varepsilon_{u}^{-1} \dqq, \infty)}}\|\varrho_{u}\|_{\dot{C}_{t,x}^{N}}+\|D\zeta_{4}\|_{C^{N-1}_{[\varepsilon_{u}^{-1} \dqq, \infty)}}\|\varrho_{u}\|^{N}_{\dot{C}_{t,x}^{1}}\notag\\
		&\lesssim \ell^{-21N-1}(1+\J^{5N}).
\end{align}
Regarding estimate \eqref{4.53}, using (\ref{6.9}) we have
\begin{align}\label{4.59}
	\big\|\gamma_{(k)}\big(\mathrm{Id} -\frac{\Ru^{u}_{\ell}+\G}{\varrho_{u}}\big)\big\|_{\dot{C}_{t,x}^{N}}
		&\lesssim        \|\gamma_{(k)}(\mathrm{Id} -\cdot)\|_{\dot{C}^{1}}\big\|\frac{\Ru^{u}_{\ell}
		+\G}{\varrho_{u}}\big\|_{\dot{C}_{t,x}^{N}}\notag\\
	    &\quad+\|D\gamma_{(k)}(\mathrm{Id} -\cdot)\|_{C^{N-1}}\big\|\frac{\Ru^{u}_{\ell}+\G}{\varrho_{u}}\big\|^{N-1}_{C_{t,x}}\big\|\frac{\Ru^{u}_{\ell}+\G}{\varrho_{u}}\big\|_{\dot{C}_{t,x}^{N}}\notag\\
		&\lesssim \big\|\frac{\Ru^{u}_{\ell}+\G}{\varrho_{u}}\big\|_{\dot{C}_{t,x}^{N}}.
\end{align}
Then, by (\ref{4.49}), (\ref{4.48}) and  (\ref{4.58}), we have
\begin{equation}\label{4.60}
	\begin{aligned}
		\big\|\frac{\Ru^{u}_{\ell}+\G}{\varrho_{u}}\big\|_{\dot{C}_{t,x}^{N}}
		&\lesssim
		\sum_{j=1}^{N} \|\varrho_{u}^{-1}\|_{\dot{C}_{t,x}^{j}} \|\Ru^{u}_{\ell}+\G\|_{\dot{C}_{t,x}^{N-j}}+\|\varrho_{u}^{-1}\|_{C_{t,x}} \|\Ru^{u}_{\ell}+\G\|_{\dot{C}_{t,x}^{N}}\\
		&\lesssim
		\sum_{j=1}^{N} \ell^{-21j-1}(1+\J^{5j})\(\ell^{-(N-j)-4}\J+ \ell^{-6(N-j)-14}(1+\J^{(N-j)+4})\)\\
		&\quad+ \dqq^{-1}\(\ell^{-N-4}\J+ \ell^{-6N-14}(1+\J^{N+4})\)\\
		&\lesssim  \ell^{-15-21N}(1+\J^{5N+4}).
\end{aligned}\end{equation}

Therefore, we get
	\begin{align}\label{4.61}
		\|a_{(k)}\|_{\dot{C}_{t,x}^{N}}
		&\lesssim \sum_{j=1}^{N}\|\varrho_{u}^{1/2}\|_{\dot{C}_{t,x}^{j}}\big\|\gamma_{(k)}\big(\mathrm{Id} -\frac{\Ru^{u}_{\ell}+\G}{\varrho_{u}}\big)\big\|_{\dot{C}_{t,x}^{N-j}}
		+\|\varrho_{u}^{1/2}\|_{C_{t,x}}\big\|\gamma_{(k)}\big(\mathrm{Id} -\frac{\Ru^{u}_{\ell}+\G}{\varrho_{u}}\big)\big\|_{\dot{C}_{t,x}^{N}}\notag\\
		&\lesssim \sum_{j=1}^{N} \ell^{-21j-1}(1+\J^{5j}) \ell^{-21(N-j)-15}(1+\J^{5(N-j)+4})
		  +\ell^{-7}(1+\J^{2}) \ell^{-21N-15}(1+\J^{5N+4})\notag\\
		&\lesssim \ell^{-21N-22}(1+\J^{5N+6}),
\end{align}
and, by (\ref{4.45})  and (\ref{6.4}),
	\begin{align}\label{4.62}
		\|a_{(k)}\|_{L_{t}^{2}L_{x}^{2}}
		&\lesssim \|\varrho_{u}\|^{1/2}_{L_{t}^{1}L_{x}^{1}}
		\|\gamma_{(k)}\|_{C(B_{\varepsilon_{u}}(\mathrm{Id} ))} \notag \\
		&\lesssim \big\|\delta_{q+1}(1+\frac{|\Ru^{u}_{\ell}+\G|}{\delta_{q+1}})\big\|^{\frac{1}{2}}_{L_{t}^{1}L_{x}^{1}}
		 \lesssim
	\delta_{q+1}^{\frac{1}{2}}+\|\Ru^{u}_{\ell}\|_{L_{t}^{1}L_{x}^{1}}^{\frac{1}{2}}+\|\G\|_{L_{t}^{1}L_{x}^{1}}^{\frac{1}{2}}.
	\end{align}
which, via (\ref{2.13}) and (\ref{4.49-2}), yields  (\ref{4.47}). 
\end{proof}

\subsection{Velocity and magnetic perturbations}

The velocity and magnetic perturbations consist of
three components:
the principal part,
the incompressible part
and the temporal incompressible.

\medskip
\paragraph{\bf $\bullet$ Principal parts.}
The principal parts  of  velocity and magnetic perturbations 
are defined, respectively, by
\begin{equation}\label{4.64}
	\begin{aligned}
		w_{q+1}^{(p)}:=\sum_{k \in \Lambda_u \cup \Lambda_B} a_{(k)} g_{(k)} W_{(k)}, \quad
		d_{q+1}^{(p)}:=\sum_{k \in \Lambda_B} a_{(k)} g_{(k)} D_{(k)}.
\end{aligned}\end{equation}

\medskip
\paragraph{\bf $\bullet$ Incompressibility correctors.}
Because neither $w_{q+1}^{(p)}$ or $d_{q+1}^{(p)}$ is divergence-free, we need the corresponding incompressibility correctors defined by
\begin{equation}\label{4.65}
	\begin{aligned}
	w_{q+1}^{(c)} & :=\sum_{k \in \Lambda_u \cup \Lambda_B} g_{(k)}\left(\nabla a_{(k)} \times \curl W_{(k)}^c+\curl(\nabla a_{(k)} \times W_{(k)}^c)\right), \\
	d_{q+1}^{(c)} & :=\sum_{k \in \Lambda_B} g_{(k)}\left(\nabla a_{(k)} \times \curl D_{(k)}^c+\curl(\nabla a_{(k)} \times D_{(k)}^c)\right) .
\end{aligned}\end{equation}
It holds that
\begin{equation}\label{4.66}
	\begin{aligned}
	w_{q+1}^{(p)}+w_{q+1}^{(c)}&=\sum_{k \in \Lambda_u \cup \Lambda_B} \operatorname{curlcurl}(a_{(k)} g_{(k)} W_{(k)}^c), \quad
	d_{q+1}^{(p)}+d_{q+1}^{(c)}&=\sum_{k \in \Lambda_B} \operatorname{curlcurl}(a_{(k)} g_{(k)} D_{(k)}^c),
\end{aligned}\end{equation}
and so one has the divergence free results:
\begin{equation}\label{4.67}
	\begin{aligned}
\operatorname{div}(w_{q+1}^{(p)}+w_{q+1}^{(c)})=\operatorname{div}(d_{q+1}^{(p)}+d_{q+1}^{(c)})=0.
\end{aligned}\end{equation}

\medskip
\paragraph{\bf $\bullet$ Temporal correctors.}  Define the temporal correctors  by
\begin{equation}\label{4.68}
	\begin{aligned}
	w_{q+1}^{(o)}:= & -\sigma^{-1} \sum_{k \in \Lambda_u} \mathbb{P}_H \mathbb{P}_{\neq 0}\left(h_{(k)} \fint_{\mathbb{T}^3} W_{(k)} \otimes W_{(k)} \mathrm{d} x \nabla(a_{(k)}^2)\right) \\
	& -\sigma^{-1} \sum_{k \in \Lambda_B} \mathbb{P}_H \mathbb{P}_{\neq 0}\left(h_{(k)} \fint_{\mathbb{T}^3} W_{(k)} \otimes W_{(k)}-D_{(k)} \otimes D_{(k)} \mathrm{d} x \nabla(a_{(k)}^2)\right), \\
	d_{q+1}^{(o)}:= &-\sigma^{-1} \sum_{k \in \Lambda_B} \mathbb{P}_H \mathbb{P}_{\neq 0}\left(h_{(k)} \fint_{\mathbb{T}^3} D_{(k)} \otimes W_{(k)}-W_{(k)} \otimes D_{(k)} \mathrm{d} x \nabla(a_{(k)}^2)\right).
\end{aligned}\end{equation}
Here,  $ \mathbb{P}_H =\mathrm{Id} -\nabla\Delta^{-1}\div $ is the Helmholtz projection.

\medskip
\paragraph{\bf $\bullet$ Velocity and magnetic perturbations} Choose the cut-off function $\Theta_{q+1} \in C^{\infty}([0, T];[0,1])$ such that
\begin{equation}\label{4.69}
	\begin{aligned}
\Theta_{q+1}(t)=\left\{\begin{array}{ll}
	0,\quad t \leq \varsigma_{q}/ 2, \\
	1,\quad \varsigma_{q}\leq t \leq T,
\end{array} \quad \text { and } \quad\|\Theta_{\mathrm{q+1}}\|_{C^{n}_{t}} \lesssim \varsigma_{q}^{-n},\right.
\end{aligned}\end{equation}
where $\varsigma_{q}=\lambda_{q}^{-30}$ is as in (\ref{2.2}).
Then, set
\begin{equation}\label{4.70}
	\begin{aligned}
&\widetilde{w}_{q+1}^{(p)}:=\Theta_{q+1} w_{q+1}^{(p)}, \quad \widetilde{w}_{q+1}^{(c)}:=\Theta_{q+1} w_{q+1}^{(c)}, \quad \widetilde{w}_{q+1}^{(o)}:=\Theta_{q+1}^2 w_{q+1}^{(o)}, \\
&\widetilde{d}_{q+1}^{(p)}:=\Theta_{q+1} d_{q+1}^{(p)}, \quad \widetilde{d}_{q+1}^{(c)}:=\Theta_{q+1} d_{q+1}^{(c)}, \quad \widetilde{d}_{q+1}^{(o)}:=\Theta_{q+1}^2 d_{q+1}^{(o)}.
\end{aligned}\end{equation}
For simplicity, we write
\begin{align}
	&w^{(*_{1})+(*_{2})}_{q+1}:=w^{(*_{1})}_{q+1}+w^{(*_{2})}_{q+1},\quad 	d^{(*_{1})+(*_{2})}_{q+1}:=d^{(*_{1})}_{q+1}+d^{(*_{2})}_{q+1},\quad\text{where} \quad *_{1},~*_{2}\in \{p, c, o\}. \label{pco-1}
\end{align}
Similar notations also apply to $\omw^{(*_{1})+(*_{2})}_{q+1}$ and $\dw^{(*_{1})+(*_{2})}_{q+1}$.

Now, we define the velocity and magnetic  perturbations $w_{q+1}$ and  $d_{q+1}$, respectively, at level $q+1$  by
\begin{equation}\label{4.71}
	\begin{aligned}
w_{q+1}=\widetilde{w}_{q+1}^{(p)}+\widetilde{w}_{q+1}^{(c)}+\widetilde{w}_{q+1}^{(o)},\quad
d_{q+1}=\widetilde{d}_{q+1}^{(p)}+\widetilde{d}_{q+1}^{(c)}+\widetilde{d}_{q+1}^{(o)}.
\end{aligned}\end{equation}
Note that, $w_{q+1}$ and  $d_{q+1}$ are  mean-zero and divergence-free.

Then, we define the new velocity and magnetic fields, respectively, at level ${q+1}$  by
\begin{equation}\label{4.72}
	\begin{aligned}
\Bu_{q+1}:=\Bu_{\ell}+w_{q+1},\quad \Bb_{q+1}:=\Bb_{\ell}+d_{q+1}.
\end{aligned}\end{equation}

The estimates of the velocity and magnetic perturbations are summarized below.
\begin{lemma}(Estimates of perturbations)\label{Lemma Estimates of perturbations}.
	For any $\rho \in(1, +\infty), \gamma \in[1, +\infty]$ and every integer $0 \leq N \leq 7$, the following estimates hold :
	\begin{align}
			&\|(\nabla^N w_{q+1}^{(p)},\nabla^N d_{q+1}^{(p)})\|_{L_t^\gamma \mathcal{L}_x^\rho} \lesssim  (1+\J^{5N+6})\ell^{-22}\laq^{N} \rw^{\frac{1}{\rho}-\frac{1}{2}}
			\tau^{\frac{1}{2}-\frac{1}{\gamma}},\label{4.73-1} \\
			&\|(\nabla^N w_{q+1}^{(c)},\nabla^N d_{q+1}^{(c)})\|_{L_t^\gamma \mathcal{L}_x^\rho}
			\lesssim 	(1+\J^{5N+16})\ell^{-43}\laq^{N-1} \rw^{\frac{1}{\rho}-\frac{1}{2}}\tau^{\frac{1}{2}-\frac{1}{\gamma}},\label{4.73-2}\\
				&\|(\nabla^N w_{q+1}^{(o)},\nabla^N d_{q+1}^{(o)})\|_{L_t^\gamma \mathcal{L}_x^\rho} \lesssim (1+\J^{5N+17})\ell^{-21N-65}\sigma^{-1},\label{4.73-3}
			\end{align}
	where the implicit constants are deterministic and  depend only on $\gamma$ and $\rho$. Moreover,
		\begin{align}
			&\|(w_{q+1}^{(p)},d_{q+1}^{(p)})\|_{\mathcal{C}_{t,x}^{N}}
			\lesssim
		(1+\J^{5N+6})\lambda_{q+1}^{2\a(N+1)},\label{4.74-1}\\	&\|(w_{q+1}^{(c)},d_{q+1}^{(c)})\|_{\mathcal{C}_{t,x}^{N}}
		\lesssim
		(1+\J^{5N+16})\lambda_{q+1}^{2\a(N+1)-1},\label{4.74-2}\\		&\|(w_{q+1}^{(o)},d_{q+1}^{(o)})\|_{\mathcal{C}_{t,x}^{N}}
		\lesssim
		(1+\J^{5N+22})\lambda_{q+1}^{2\a N+1},\label{4.74-3}
		\end{align}
where the implicit constants are deterministic and  independent of $q$.
\end{lemma}
\begin{proof} First, using Lemmas \ref{Lemma spacial building blocks}, \ref{Lemma temporal building blocks}, \ref{Lemma Velocity amplitudes} and (\ref{4.64}), we get
	\begin{align*}
\|\nabla^N w_{q+1}^{(p)}\|_{L_t^\gamma L_x^\rho}
&\lesssim \sum_{k \in \Lambda_u \cup \Lambda_B}   \sum_{N_{1}+N_{2}=N}\|a_{(k)}\|_{C_{t,x}^{N_{1}}}\|\nabla^{N_{2}} W_{(k)}\|_{L_x^\rho}\|g_{(k)}\|_{L_t^\gamma}\notag\\
&\lesssim \sum_{N_{1}+N_{2}=N}\ell^{-21N_{1}-22}(1+\J^{5N_{1}+6}) \rw^{\frac{1}{\rho}-\frac{1}{2}}\laq^{N_{2}}\tau^{\frac{1}{2}-\frac{1}{\gamma}}\notag\\
&\lesssim (1+\J^{5N+6})\ell^{-22}\laq^{N} \rw^{\frac{1}{\rho}-\frac{1}{2}}
\tau^{\frac{1}{2}-\frac{1}{\gamma}}. 
\end{align*}
Similarly, we obtain
\begin{align*}
		 \|\nabla^N w_{q+1}^{(c)}\|_{L_t^\gamma L_x^\rho}
		&\lesssim
		\sum_{k \in \Lambda_u \cup \Lambda_B} \|g_{(k)}\|_{L_t^\gamma}  \big(\sum_{N_{1}+N_{2}=N}\| a_{(k)}\|_{C_{t,x}^{N_{1}+1}}\| W_{(k)}^{c}\|_{W_x^{N_{2}+1,\rho}}
		+\| a_{(k)}\|_{C_{t,x}^{N_{1}+2}}
		\| W_{(k)}^{c}\|_{W_x^{N_{2},\rho}}\big)\notag\\
	&\lesssim (1+\J^{5N+16})\tau^{\frac{1}{2}-\frac{1}{\gamma}}(\ell^{-43} \rw^{\frac{1}{\rho}-\frac{1}{2}}\laq^{N-1}
		+\ell^{-64} \rw^{\frac{1}{\rho}-\frac{1}{2}}\laq^{N-2})\notag\\
		&\lesssim
		(1+\J^{5N+16})\ell^{-43}\laq^{N-1} \rw^{\frac{1}{\rho}-\frac{1}{2}}\tau^{\frac{1}{2}-\frac{1}{\gamma}}.
\end{align*}
For the temporal correctors, using Lemmas \ref{Lemma spacial building blocks}-\ref{Lemma Velocity amplitudes}  and (\ref{4.68}), we have
\begin{align*}
		\|\nabla^N w_{q+1}^{(o)}\|_{L_t^\gamma L_x^\rho}
		&\lesssim
		 \sigma^{-1} \sum_{k \in \Lambda_u} \|\nabla^N\mathbb{P}_H \mathbb{P}_{\neq 0}
\big(h_{(k)} \fint_{\mathbb{T}^3} W_{(k)} \otimes W_{(k)} \mathrm{d} x \nabla(a_{(k)}^2) \big) \|_{L_t^\gamma L_x^\rho}\notag\\
         & \quad+\sigma^{-1} \sum_{k \in \Lambda_B} \|\nabla^N\mathbb{P}_H \mathbb{P}_{\neq 0}
         \big(h_{(k)} \fint_{\mathbb{T}^3} W_{(k)} \otimes W_{(k)}-D_{(k)} \otimes D_{(k)} \mathrm{d} x \nabla(a_{(k)}^2) \big)\|_{L_t^\gamma L_x^\rho}\notag\\
        &\lesssim
        \sigma^{-1} \sum_{k \in \Lambda_u\cup \Lambda_B} \|h_{(k)}\|_{C_{t}}\|\nabla ^{N+1}(a_{(k)}^2)\|_{C_{t,x}}\\\notag
        &\lesssim
         (1+\J^{5N+17})\ell^{-21N-65}\sigma^{-1}.
\end{align*}
The $L_t^\gamma L_x^\rho$-estimate of  magnetic pertubations can be derived in a similar manner. 
This verifies (\ref{4.73-1})-(\ref{4.73-3}).

It remains to prove $ C_{t, x}^{N}$-estimates. We apply (\ref{4.26}) and (\ref{4.47}) to get	
\begin{align*}
\|w_{q+1}^{(p)}\|_{C_{t, x}^{N}}
&\lesssim
	\sum_{k \in \Lambda_u \cup \Lambda_B}	\sum_{N_{1}+N_{2}\leq N}\|g_{(k)}\|_{C_t^{N_{1}}}   \sum_{N_{21}+N_{22}\leq N_{2}}\|a_{(k)}\|_{C_{t, x}^{N_{21}}}\|W_{(k)}\|_{C_x^{N_{22}}}\notag\\
	&\lesssim (1+\J^{5N+6})\ell^{-22} \rw^{-\frac{1}{2}}\sigma^{N}\tau^{N+\frac{1}{2}}\notag\\
		&\lesssim	(1+\J^{5N+6})\lambda_{q+1}^{2\a(N+1)},
\end{align*}
where we also used $\sigma\tau=\lambda_{q+1}^{2\varepsilon+2\a}>\lambda_{q+1}$  and (\ref{4.1}) in the last two estimates.

Similarly, we obtain
	\begin{align*}
		\| w_{q+1}^{(c)}\|_{C_{t, x}^{N}}
		 &\lesssim
		\sum_{k \in \Lambda_u \cup \Lambda_B}
		\sum_{N_{1}+N_{2}\leq N}\|g_{(k)}\|_{C_t^{N_{1}}}\bigg(
		\sum_{N_{21}+N_{22}=N_{2}}
		\| a_{(k)}\|_{C_{t, x}^{N_{21}+1}}\| W_{(k)}^{c}\|_{C_{x}^{N_{22}+1}}\notag\\	&\quad\quad\quad\quad\quad\quad\quad\quad\quad\quad\quad\quad\quad\quad\quad\quad\quad
	+\| a_{(k)}\|_{C_{t, x}^{N_{21}+2}}\|W_{(k)}^{c}\|_{C_{x}^{N_{22}}}
    \bigg)\notag\\
		&\lesssim (1+\J^{5N+16})\ell^{-43} \rw^{-\frac{1}{2}}\laq^{-1}\sigma^{N}\tau^{N+\frac{1}{2}}\notag \\
		&\lesssim
		(1+\J^{5N+16})\lambda_{q+1}^{2\a(N+1)-1}.
\end{align*}

For the temporal corrector, using the Sobolev embedding $ W^{1,4}(\T^{3}) \hookrightarrow L^{\infty}(\T^{3})$ and the fact that $\mathbb{P}_H \mathbb{P}_{\neq 0}$ is bounded in $ W^{1,4}(\T^{3})$, we obtain
	\begin{align*}
		\|w_{q+1}^{(o)}\|_{C_{t, x}^{N}}
		&\lesssim
		\sigma^{-1} \sum_{k \in \Lambda_u\cup \Lambda_B}  \|\mathbb{P}_H\mathbb{P}_{\neq 0}\big(h_{(k)} \nabla (a_{(k)}^2)\big) \|_{C_{t,x}^{N}}\notag\\
		&\lesssim
		\sigma^{-1} \sum_{k \in \Lambda_u\cup \Lambda_B}  \sum_{N_{1}+N_{2}\leq N}\|\mathbb{P}_H\mathbb{P}_{\neq 0}\big(h_{(k)} \nabla (a_{(k)}^2)\big) \|_{C_{t}^{N_{1}}W_{x}^{N_{2}+1,4}}\notag\\
		&\lesssim
		\sigma^{-1} \sum_{k \in \Lambda_u\cup \Lambda_B} \sum_{N_{1}+N_{2}\leq N} \|h_{(k)}\|_{C_{t}^{N_{1}}}\|\nabla (a_{(k)}^2)\|_{C_{t,x}^{N_{2}+1}}.
	\end{align*}
Since $\sigma\tau>\ell^{-21}$, $\varepsilon\leq \frac{1}{40}$, by (\ref{4.1}), we come to
	\begin{align*}
		\|w_{q+1}^{(o)}\|_{C_{t, x}^{N}}
		&\lesssim
		(1+\J^{5(N+2)+12})\sigma^{-1}\sum_{N_{1}+N_{2}\leq N}\sigma^{N_{1}}\tau^{N_{1}}\ell^{-21(N_{2}+2)-44}\notag\\
		&\lesssim
		(1+\J^{5N+22})\ell^{-86}\sigma^{N-1}\tau^{N}\notag\\
		&\lesssim
			(1+\J^{5N+22})\lambda_{q+1}^{2\a N+1}.
\end{align*}
The  magnetic parts can be estimated analogously. 
Thus, the proof is completed.
\end{proof}

\subsection{Verification of inductive estimates}
We are now ready to verify the inductive estimates  (\ref {2.11}), (\ref {2.14})-(\ref {2.16}) at 
level $q+1$ for the velocity and magnetic fields.

By (\ref{4.70}), (\ref{4.74-1})-(\ref{4.74-3}), for $0\leq N\leq 4$, we obtain
\begin{equation}\label{4.81}
	\begin{aligned}
	\|(w_{q+1},d_{q+1})\|_{\mathcal{C}_{t,x}^{N}}	
		&\lesssim
		(1+\J^{5N+6})\lambda_{q+1}^{2\a(N+1)}\varsigma_{q}^{-N}+(1+\J^{5N+16})\lambda_{q+1}^{2\a(N+1)-1}\varsigma_{q}^{-N}+(1+\J^{5N+22})\lambda_{q+1}^{2\a N+1}\varsigma_{q}^{-N}
	\\
		&\lesssim
		(1+\J^{5N+22})\lambda_{q+1}^{2\a(N+1)+1}.
\end{aligned}\end{equation}
Thus,  by (\ref{2.11}), (\ref{2.12}) and (\ref{4.81}), for $0\leq N\leq 4$,   
	\begin{align}\label{4.82}
	\|(\Bu_{q+1},\Bb_{q+1})\|_{\Lo^{m}_{\om}\mathcal{C}_{t,x}^{N}}
		&\lesssim
		\|(\Bu_{\ell},\Bb_{\ell})\|_{\Lo^{m}_{\om}\mathcal{C}_{t,x}^{N}}+\|(w_{q+1},d_{q+1})\|_{\Lo^{m}_{\om}\mathcal{C}_{t,x}^{N}}	\notag\\
		&\lesssim
			\|(\Bu_{q},\Bb_{q})\|_{\Lo^{m}_{\om}\mathcal{C}_{t,x}^{N}}+(1+\|\J\|^{5N+22}_{\Lo^{(5N+22)m}_{\om}})\lambda_{q+1}^{2\a(N+1)+1}	\notag\\
	 	&\leq
\(8(5N+22)mL^{2}80^{q}\)^{(5N+22)80^{q}}\laq^{2\a(N+1)+2},
\end{align}
where in the last step we also used the inequality $\lambda_{q}^{(4\a+5)(5N+22)}<\lambda_{q+1}$.

Next, we shall verify the $\Lo^{2r}_{\om}L_{t}^{2}\mathcal{L}_{x}^{2}$-decay estimate (\ref{2.14}). Note that, the previous estimate (\ref{4.73-1}) cannot yield the decay  estimate (\ref{2.14}). It is important here to apply the $L^{p}-$decorrelation in Lemma \ref{Lemma Decorrelation1} below:
	\begin{align}\label{4.83}
		\|(w_{q+1}^{(p)},d_{q+1}^{(p)})\|_{L_{t}^{2}\mathcal{L}_{x}^{2}}
		&\lesssim
		\sum_{k \in \Lambda_u \cup \Lambda_B}\(\|a_{(k)}\|_{L_{t}^2L_{x}^2}\|g_{(k)}\|_{L_t^2}\|\phi_{(k)}\|_{L_x^2}+\sigma^{-\frac{1}{2}}\|a_{(k)}\|_{C_{t, x}^1}\|g_{(k)}\|_{L_t^2}\|\phi_{(k)}\|_{L_x^2}\).
\end{align}
By Lemmas \ref{Lemma Magnetic amplitudes} and \ref{Lemma Velocity amplitudes}, 
	\begin{align*}
	\|(w_{q+1}^{(p)},d_{q+1}^{(p)})\|_{L_{t}^{2}\mathcal{L}_{x}^{2}}
	\lesssim
	\dqq^{\frac{1}{2}}+\JJ^{\frac{1}{2}}+\|\G\|_{L_{t}^{1}L_{x}^{1}}^{\frac{1}{2}}+\sigma^{-\frac{1}{2}}\ell^{-43}(1+\J^{11}).
\end{align*} 
Taking  into account (\ref{4.73-2})-(\ref{4.73-3}) we have
	\begin{align*}
		\|(w_{q+1},d_{q+1})\|_{L_{t}^{2}\mathcal{L}_{x}^{2}}
		&\lesssim
		\dqq^{\frac{1}{2}}+\JJ^{\frac{1}{2}}+\|\G\|_{L_{t}^{1}L_{x}^{1}}^{\frac{1}{2}}+(1+\J^{11})\sigma^{-\frac{1}{2}}\ell^{-43}+(1+\J^{16})\ell^{-43}\laq^{-1}+(1+\J^{17})\ell^{-65}\sigma^{-1}	\notag\\
		&\lesssim
		\dqq^{\frac{1}{2}}+\JJ^{\frac{1}{2}}+\|\G\|_{L_{t}^{1}L_{x}^{1}}^{\frac{1}{2}}+(1+\J^{17})\ell^{-43}\sigma^{-\frac{1}{2}}.
\end{align*}
Then, by  (\ref{2.2-2}), (\ref{2.12}), (\ref{3.2-1}), (\ref{4.49-2}) and  (\ref{4.69}), we obtain
\begin{align*}
		\|(w_{q+1},d_{q+1})\|_{\Lo^{2r}_{\om}L_{t}^{2}\mathcal{L}_{x}^{2}}
		&\lesssim
		\dqq^{\frac{1}{2}}+\|\JJ\|_{\Lo^{r}_{\om}}^{\frac{1}{2}}+\|\G\|_{\Lo^{r}_{\om}L_{t}^{1}L_{x}^{1}}^{\frac{1}{2}}+(1+\|\J\|^{17}_{\Lo^{34r}_{\om}})\ell^{-43}\sigma^{-\frac{1}{2}}	\notag\\
		&\lesssim
		\dqq^{\frac{1}{2}}+\(\lambda_{q}^{(4\a+5)}(8\cdot34rL^{2}80^{q})^{80^{q}}\)^{17}\la^{43\times60}\laq^{-\varepsilon}\notag\\
		&\lesssim
		\dqq^{\frac{1}{2}}+\lambda_{q+1}^{-\frac{1}{2}\varepsilon}
		\lesssim
		\dqq^{\frac{1}{2}}.
\end{align*}
Thus, by (\ref{2.11}),
	\begin{align}\label{4.86}
	\|(\Bu_{q+1}-v_q,\Bb_{q+1}-\Bb_q)\|_{\Lo^{2r}_{\om}L_{t}^{2}\mathcal{L}_{x}^{2}}
	&\lesssim
	 \|(v_\ell-v_q,\Bb_\ell-\Bb_q)\|_{\Lo^{2r}_{\om}L_{t}^{2}\mathcal{L}_{x}^{2}}+	\|(w_{q+1},d_{q+1})\|_{\Lo^{2r}_{\om}L_{t}^{2}\mathcal{L}_{x}^{2}} \notag\\
	& \lesssim
	\ell\|(v_q,\Bb_q)\|_{\Lo^{2r}_{\om}\mathcal{C}_{t,x}^{1}}
    +\delta_{q+1}^{\frac{1}{2}} \notag\\
    &\lesssim
    \lambda_{q}^{-60}\la^{4\a+2}(8\cdot54rL^{2}80^{q-1})^{27\cdot80^{q-1}}+\delta_{q+1}^{\frac{1}{2}}
    \leq M^*\delta_{q+1}^{\frac{1}{2}},
\end{align}
which yields (\ref{2.14}) at level $q+1$ for $M^*$ sufficiently large.

Concerning the $\Lo^{r}_{\om}L_{t}^{1}\mathcal{L}_{x}^{2}$-estimate (\ref{2.15}) at $q+1$,  by  (\ref{4.69}) and (\ref{4.73-1})-(\ref{4.73-3}) in Lemma \ref{Lemma Estimates of perturbations}, 
	\begin{align*}
		\|(w_{q+1},d_{q+1})\|_{L_{t}^{1}\mathcal{L}_{x}^{2}}
	&\lesssim
	(1+\J^{6})\ell^{-22}\tau^{-\frac{1}{2}}
	+(1+\J^{16})\ell^{-43}\tau^{-\frac{1}{2}}\laq^{-1}+(1+\J^{17})\ell^{-65}\sigma^{-1}\notag\\
	&\lesssim
	(1+\J^{17})\ell^{-65}\sigma^{-1},
\end{align*}
which, along with (\ref{2.2-2}), (\ref{2.11}), (\ref{2.12})  and (\ref{3.2-1}), implies
	\begin{align}\label{4.88}
		 \|(\Bu_{q+1}-v_q,\Bb_{q+1}-\Bb_q)\|_{\Lo^{r}_{\om}L_{t}^{1}\mathcal{L}_{x}^{2}} 
		& \lesssim
		\|(v_\ell-v_q,\Bb_\ell-\Bb_q)\|_{\Lo^{r}_{\om}L_{t}^{1}\mathcal{L}_{x}^{2}}+\|(w_{q+1},d_{q+1})\|_{\Lo^{r}_{\om}L_{t}^{1}\mathcal{L}_{x}^{2}}\notag \\
		& \lesssim
		\ell\|(v_q,\Bb_q)\|_{\Lo^{r}_{\om}\mathcal{C}_{t,x}^{1}}+	(1+\|\J\|_{\Lo^{17r}_{\om}}^{17})\ell^{-65}\sigma^{-1} \notag\\
		 &\leq
		\dqqq^{\frac{1}{2}}.
\end{align}
 This verifies the $\Lo^{r}_{\om}L_{t}^{1}L_{x}^{2}$-decay estimate (\ref{2.15}) at level $q+1$.

At last, regarding the inductive estimate \eqref{2.16},  by Lemma \ref{Lemma Estimates of perturbations}, we get
	\begin{align*}
	\|(w_{q+1},d_{q+1})\|_{L_{t}^{\ga}\mathcal{W}_{x}^{s,p}}
	&\lesssim
	(1+\J^{5s+6})\ell^{-22}\laq^{s}\tau^{\frac{1}{2}-\frac{1}{\ga}}r_{\perp}^{\frac{1}{p}-\frac{1}{2}}
	+(1+\J^{5s+16})\ell^{-43}\laq^{s-1}\tau^{\frac{1}{2}-\frac{1}{\ga}}r_{\perp}^{\frac{1}{p}-\frac{1}{2}}\notag\\
	&\quad+(1+\J^{5s+17})\ell^{-21s-65}\sigma^{-1}\notag\\
&\lesssim
	(1+\J^{5s+17})\ell^{-21s-65}\sigma^{-1},
\end{align*}
where the last step was due to  (\ref{2.3}),
\begin{align*}
	\ell^{-22}\laq^{s}\tau^{\frac{1}{2}-\frac{1}{\ga}}r_{\perp}^{\frac{1}{p}-\frac{1}{2}}
	&=\lambda_{q}^{22\times60}\lambda_{q+1}^{s-\frac{2\a-2}{p}-\frac{2\a}{\gamma}+(6-\frac{10}{p})\varepsilon+2\a-1}
	\leq \lambda_{q+1}^{-10\varepsilon}\leq \ell^{-21s-65}\sigma^{-1}.
\end{align*}
Then, using the embedding $C_{x}^{2}(\T^{3}) \hookrightarrow W_{x}^{s,p}(\T^{3})$ for  $(s, \gamma, p) \in \mathcal{S}$  and (\ref{3.2-1}), we get
	\begin{align}\label{4.89}
		 \|(\Bu_{q+1}-v_q,\Bb_{q+1}-\Bb_q)\|_{\Lo^{r}_{\om}L_{t}^{\ga}\mathcal{W}_{x}^{s,p}} 
		& \lesssim
		\|(v_\ell-v_q,\Bb_\ell-\Bb_q)\|_{\Lo^{r}_{\om}L_{t}^{\ga}\mathcal{W}_{x}^{s,p}}+\|(w_{q+1},d_{q+1})\|_{\Lo^{r}_{\om}L_{t}^{\ga}\mathcal{W}_{x}^{s,p}} \notag\\
		&\lesssim
		\ell\|(v_q,\Bb_q)\|_{\Lo^{r}_{\om}\mathcal{C}_{t,x}^{3}}+	 (1+	\|\J\|_{\Lo^{(5s+17)r}_{\om}}^{5s+17})\ell^{-21s-65}\sigma^{-1}\notag\\
		 &\lesssim
		\lambda_{q}^{-45}+\lambda_{q+1}^{-\varepsilon}
	\lesssim
		\dqqq^{\frac{1}{2}}.
\end{align}
 Therefore,  estimate (\ref{2.16}) is verified at level $q+1$.

\section{Reynolds and magnetic stresses}\label{Sec-Reynolds}

This section aims to verify the inductive estimates (\ref{2.12}) and (\ref{2.13})
for the delicate Reynolds and magnetic stresses at level $q+1$.

\subsection{Magnetic  stress}
Let us first derive
the new magnetic stress $\Ru^{B}_{q+1}$ at level $q+1$. 
Subtracting system (\ref{3.5})  from system (\ref{2.5}) at level $q+1$, we compute
	\begin{align}\label{5.1}
	&\quad\div\Ru^{B}_{q+1}\notag\\
	&=\p_{t}(\dw^{(p)}_{q+1}+\dw^{(c)}_{q+1})+\nu(-\Delta)^{\a}d_{q+1}-( z_{2,q+1}-z_{2,\ell})\notag\\
	&\quad\underbrace{+\div\((\widetilde{B}_{\ell}+z_{2,\ell})\otimes w_{q+1}-w_{q+1}\otimes (\widetilde{B}_{\ell}+z_{2,\ell})+d_{q+1}\otimes( \Bu_{\ell}+z_{1,\ell})-( \Bu_{\ell}+z_{1,\ell})\otimes d_{q+1}\)}_{\div \Ru^{B}_{lin}}\notag\\ 
	&\quad+\underbrace{\div \(\dw^{(p)}_{q+1}\otimes                                                          (\omw^{(c)}_{q+1}+\omw^{(o)}_{q+1})+(\dw^{(c)}_{q+1}+\dw^{(o)}_{q+1})\otimes w_{q+1}
	-\omw^{(p)}_{q+1}\otimes
	(\dw^{(c)}_{q+1}+\dw^{(o)}_{q+1})-(\omw^{(c)}_{q+1}+\dw^{(o)}_{q+1})\otimes d_{q+1}\)}_{\div \Ru^{B}_{corr}}\notag\\ 
	&\quad+\div\( \Bb_{q+1}\otimes z_{1,q+1}-z_{1,q+1}\otimes \Bb_{q+1}+z_{1,\ell}\otimes \Bb_{q+1}- \Bb_{q+1}\otimes z_{1,\ell}+z_{2,q+1}\otimes \Bu_{q+1}
	-\Bu_{q+1}\otimes  z_{2,q+1}\)\notag\\
	&\quad+\underbrace{\div\(\Bu_{q+1}\otimes z_{2,\ell}-z_{2,\ell}\otimes \Bu_{q+1}
	+z_{1,\ell}\otimes z_{2,\ell}
	-z_{2,\ell}\otimes z_{1,\ell}
	+z_{2,q+1}\otimes z_{1,q+1}
	-z_{1,q+1}\otimes  z_{2,q+1}\)}_{\div \Ru^{B}_{com2}}\notag\\
	&\quad+\div(\Ru^{B}_{com1})+\underbrace{\div\(\dw^{(p)}_{q+1}\otimes\omw^{(p)}_{q+1}-\omw^{(p)}_{q+1}\otimes\dw^{(p)}_{q+1}+\Ru^{B}_{\ell}\)+\p_{t}\dw^{(o)}_{q+1}}_{\div \Ru^{B}_{osc}}, 
\end{align} 
where $\Ru^{B}_{com1}$ is given by \eqref{3.8}. 
Using the inverse divergence operator $\mathcal{R}^{B}$ (see  Section  \ref{Standard tools} below), we set 
\begin{align}
		&\Ru^{B}_{lin}
		:=\mathcal{R}^{B}\p_{t}\dw_{q+1}^{(p)+(c)}+\mathcal{R}^{B}\(\nu(-\Delta)^{\a}d_{q+1}\)+\mathcal{R}^{B}(z_{2,\ell}- z_{2,q+1})
		 +(\widetilde{B}_{\ell}+z_{2,\ell})\otimes w_{q+1}\notag\\
		&\quad\quad\quad
		-w_{q+1} \otimes  (\widetilde{B}_{\ell}+z_{2,\ell})+d_{q+1} \otimes ( \Bu_{\ell}+z_{1,\ell})
		-( \Bu_{\ell}+z_{1,\ell})\otimes d_{q+1},\label{5.4}\\
		&\Ru^{B}_{corr}
		:=\dw^{(p)}_{q+1}\otimes \omw^{(c)+(o)}_{q+1}+\dw^{(c)+(o)}\otimes w_{q+1}-\omw^{(p)}_{q+1}\otimes  \dw^{(c)+(o)} -\omw^{(c)+(o)}\otimes d_{q+1}, \label{5.5}\\
	&\Ru^{B}_{com2}
	:= \Bb_{q+1}\otimes z_{1,q+1}	
	   -z_{1,q+1}\otimes \Bb_{q+1}
	   +z_{1,\ell}\otimes \Bb_{q+1}
	   -\Bb_{q+1}\otimes z_{1,\ell}
		+z_{2,q+1}\otimes \Bu_{q+1}
		-\Bu_{q+1}\otimes z_{2,q+1}
		\notag\\
		&\quad\quad\quad\quad
		+ \Bu_{q+1}\otimes z_{2,\ell}
		-z_{2,\ell}\otimes \Bu_{q+1}
	    +z_{1,\ell}\otimes z_{2,\ell}
		-z_{2,\ell}\otimes z_{1,\ell}+ z_{2,q+1}\otimes z_{1,q+1}-z_{1,q+1}\otimes  z_{2,q+1}.\label{5.6}
\end{align}

In order to treat the delicate oscillation error from the last line in (\ref{5.1}), using Lemma \ref{Lemma Magnetic amplitudes} we obtain
	\begin{align*}
		&\quad\div(\dw^{(p)}_{q+1}\otimes\omw^{(p)}_{q+1}-\omw^{(p)}_{q+1}\otimes\dw^{(p)}_{q+1}+\Ru^{B}_{\ell})\notag\\
			&=(1-\Theta_{q+1}^{2})\div \Ru^{B}_{\ell}
		+\Theta_{q+1}^{2}\div\(\sum_{k\in\Lambda_B} a_{(k)}^{2}(g_{(k)}^{2}-1) \fint_{\T^3}D_{(k)}\otimes W_{(k)}-W_{(k)}\otimes D_{(k)} \mathrm{d} x\)\notag\\
		&\quad+\Theta_{q+1}^{2}\div\(\sum_{k\in\Lambda_B}a_{(k)}^{2}g_{(k)}^{2}\mathbb{P}_{\neq 0}\(D_{(k)}\otimes W_{(k)}-W_{(k)}\otimes D_{(k)}\)\),
\end{align*}
which, along with \eqref{4.68}, yields that
\begin{align*}
	&\quad\div(\dw^{(p)}_{q+1}\otimes\omw^{(p)}_{q+1}-\dw^{(p)}_{q+1}\otimes\omw^{(p)}_{q+1}+\Ru^{B}_{\ell})+\p_{t}\dw^{(o)}_{q+1}\notag\\
	&=\mathbb{P}_H\div(\dw^{(p)}_{q+1}\otimes\omw^{(p)}_{q+1}-\dw^{(p)}_{q+1}\otimes\omw^{(p)}_{q+1}+\Ru^{B}_{\ell})+\p_{t}\dw^{(o)}_{q+1}\notag\\
	&=(1-\Theta_{q+1}^{2})\div \Ru^{B}_{\ell}+\p_{t}\Theta_{q+1}^{2}d^{(o)}_{q+1}
	+\Theta_{q+1}^{2}\mathbb{P}_H\div\(\sum_{k\in\Lambda_B}a_{(k)}^{2}g_{(k)}^{2}\mathbb{P}_{\neq 0}\(D_{(k)}\otimes W_{(k)}-W_{(k)}\otimes D_{(k)}\)\)\notag\\
	&\quad-\Theta_{q+1}^{2}\sigma^{-1}\sum_{k\in\Lambda_B}\mathbb{P}_H\mathbb{P}_{\neq 0}\(h_{(k)}\fint_{\mathbb{T}^3}D_{(k)} \otimes W_{(k)}-W_{(k)} \otimes D_{(k)} \mathrm{d} x \p_{t}\nabla (a_{(k)}^{2})\).
\end{align*}
The above identities suggest the following candidate for the oscillation error 
\begin{align}  
\Ru^{B}_{osc}:=\Ru^{B}_{osc.1}+\Ru^{B}_{osc.2}+\Ru^{B}_{osc.3},  \label{5.7-5}
\end{align}
where
	\begin{align}
	&\Ru^{B}_{osc.1}
	:=(1-\Theta_{q+1}^{2}) \Ru^{B}_{\ell}+\mathcal{R}^{B}(\p_{t}\Theta_{q+1}^{2}d^{(o)}_{q+1}), \label{5.7-1}\\
		&\Ru^{B}_{osc.2}:=\Theta_{q+1}^{2}\sum_{k\in\Lambda_B}\mathcal{R}^{B}\mathbb{P}_H\mathbb{P}_{\neq 0}\(g_{(k)}^{2}\mathbb{P}_{\neq 0}\(D_{(k)}\otimes W_{(k)}-W_{(k)}\otimes D_{(k)}\)\nabla(a_{(k)}^{2})\),\label{5.7-2}\\
	&\Ru^{B}_{osc.3}:=-\Theta_{q+1}^{2}\sigma^{-1}\sum_{k\in\Lambda_B}\mathcal{R}^{B}\mathbb{P}_H\mathbb{P}_{\neq 0}\(h_{(k)}\fint_{\mathbb{T}^3}D_{(k)} \otimes W_{(k)}- W_{(k)}\otimes D_{(k)} \mathrm{d} x \p_{t}\nabla (a_{(k)}^{2})\).    \label{5.7-3}
\end{align}
Thus, we have the new magnetic stress at level $ q + 1 $: 
$$\Ru^{B}_{q+1}:=\Ru^{B}_{lin}+\Ru^{B}_{corr}+\Ru^{B}_{com1}+\Ru^{B}_{com2}+\Ru^{B}_{osc}.$$

\subsection{Velocity Reynolds stress}
We also compute for the Reynolds stress as follows.
\begin{align}\label{5.8}
		&\quad\div\Ru^{u}_{q+1}-\nabla P_{q+1}\notag\\
		&=\p_{t}\omw^{(p)+(c)}_{q+1}+\nu(-\Delta)^{\a}w_{q+1}-(z_{1,q+1}-z_{1,\ell})\notag\\
		&\quad\underbrace{+\div\((\Bu_{\ell}+z_{1,\ell})\otimes w_{q+1}+w_{q+1}\otimes ( \Bu_{\ell}+z_{1,\ell})-(\widetilde{B}_{\ell}+z_{2,\ell})\otimes d_{q+1}-d_{q+1}\otimes(\widetilde{B}_{\ell}+z_{2,\ell})\)}_{\div \Ru^{u}_{lin}+\nabla P_{lin}}\notag\\  
		&\quad+\underbrace{\div \(\omw^{(p)}_{q+1}\otimes                                                          \omw^{(c)+(o)}_{q+1}+\omw^{(c)+(o)}_{q+1}\otimes w_{q+1}
		-\dw^{(p)}_{q+1}\otimes \dw^{(c)+(o)}_{q+1}
		-\dw^{(c)+(o)}_{q+1}\otimes d_{q+1}\)}_{\div \Ru^{u}_{corr}+\nabla P_{corr}}\notag\\ 
		&\quad+\div\( \Bu_{q+1}\otimes z_{1,q+1}+z_{1,q+1}\otimes \Bu_{q+1}
		-z_{1,\ell}\otimes \Bu_{q+1}
		-\Bu_{q+1}\otimes z_{1,\ell}
		-\Bb_{q+1}\otimes  z_{2,q+1}
		- z_{2,q+1}\otimes \Bb_{q+1}
		\)\notag\\
		&\quad+\underbrace{\div\(z_{2,\ell}\otimes \Bb_{q+1}
		+\Bb_{q+1}\otimes z_{2,\ell}
		+z_{1,q+1}\otimes z_{1,q+1}- z_{2,q+1}\otimes  z_{2,q+1}-z_{1,\ell}\otimes z_{1,\ell}+z_{2,\ell}\otimes z_{2,\ell}\)}_{\div \Ru^{u}_{com2}+\nabla P_{com2}}\notag\\
		&\quad+\div(\Ru^{u}_{com1})+\underbrace{\div\(\omw^{(p)}_{q+1}\otimes\omw^{(p)}_{q+1}-\dw^{(p)}_{q+1}\otimes\dw^{(p)}_{q+1}+\Ru^{u}_{\ell}\)+\p_{t}\omw^{(o)}_{q+1}}_{\div \Ru^{u}_{osc}+\nabla P_{osc}}, 
\end{align} 
where $\Ru^{u}_{com1}$ is given by \eqref{3.7}. 
Using the inverse divergence operator  $\mathcal{R}^{u}$ in Section  \ref{Standard tools} below we define
	\begin{align}
	&\Ru^{u}_{lin}
	:=\mathcal{R}^{u}\p_{t}\omw_{q+1}^{(p)+(c)}+\mathcal{R}^{u}\(\nu(-\Delta)^{\a}w_{q+1}\)+\mathcal{R}^{u}(z_{1,\ell}-z_{1,q+1})+( \Bu_{\ell}+z_{1,\ell})\mathring{\otimes}w_{q+1}\notag\\
	&\quad\quad\quad\quad+w_{q+1}\mathring{\otimes}( \Bu_{\ell}+z_{1,\ell})
	-(\widetilde{B}_{\ell}+z_{2,\ell})\mathring{\otimes}d_{q+1}-d_{q+1}\mathring{\otimes}(\widetilde{B}_{\ell}+z_{2,\ell}),\label{5.11}\\
	&\Ru^{u}_{corr}:=
	\omw_{q+1}^{(p)}\mathring{\otimes} \omw_{q+1}^{(c)+(o)}	+\omw_{q+1}^{(c)+(o)}\mathring{\otimes}w_{q+1}
	-\dw_{q+1}^{(p)}\mathring{\otimes} \dw^{(c)+(o)}    -\dw^{(c)+(o)}\mathring{\otimes} d_{q+1},\label{5.12}\\
	&\Ru^{u}_{com2}
	:= \Bu_{q+1}\mathring{\otimes}z_{1,q+1}+z_{1,q+1}\mathring{\otimes} \Bu_{q+1}
	-z_{1,\ell}\mathring{\otimes} \Bu_{q+1}
	-\Bu_{q+1}\mathring{\otimes}z_{1,\ell}
	- \Bb_{q+1}\mathring{\otimes} z_{2,q+1}- z_{2,q+1}\mathring{\otimes} \Bb_{q+1}
	\notag\\
	&\quad\quad\quad\quad
	+z_{2,\ell}\mathring{\otimes} \Bb_{q+1}+\Bb_{q+1}\mathring{\otimes}z_{2,\ell}
	+z_{1,q+1}\mathring{\otimes} z_{1,q+1}- z_{2,q+1}\mathring{\otimes} z_{2,q+1}-z_{1,\ell}\mathring{\otimes}z_{1,\ell}+z_{2,\ell}\mathring{\otimes}z_{2,\ell}.\label{5.13}
\end{align}
While for the remaining oscillation error, we use Lemma \ref{Lemma Velocity amplitudes} to derive 
\begin{align}\label{5.9}
	&\quad\div(\omw^{(p)}_{q+1}\otimes\omw^{(p)}_{q+1}-\dw^{(p)}_{q+1}\otimes\dw^{(p)}_{q+1}+\Ru^{u}_{\ell})\notag\\
		&=(1-\Theta_{q+1}^{2})\div \Ru^{u}_{\ell}+\Theta_{q+1}^{2}\div\(\sum_{k\in\Lambda_u} a_{(k)}^{2}(g_{(k)}^{2}-1) \fint_{\T^3}W_{(k)}\otimes W_{(k)} \mathrm{d} x\)\notag\\
	&\quad+\Theta_{q+1}^{2}\div(\varrho_{u}\mathrm{Id} )+\Theta_{q+1}^{2}\div\(\sum_{k\in\Lambda_u}a_{(k)}^{2}g_{(k)}^{2}\mathbb{P}_{\neq 0}(W_{(k)}\otimes W_{(k)})\)\notag\\
	&\quad+\Theta_{q+1}^{2}\div\(\sum_{k\in\Lambda_B}a_{(k)}^{2}g_{(k)}^{2}\mathbb{P}_{\neq 0}(W_{(k)}\otimes W_{(k)}-D_{(k)}\otimes D_{(k)})\)\notag\\
	&\quad+\Theta_{q+1}^{2}\div\(\sum_{k\in\Lambda_B}a_{(k)}^{2}(g_{(k)}^{2}-1)\fint_{\mathbb{T}^3}(W_{(k)}\otimes W_{(k)}-D_{(k)}\otimes D_{(k)})\mathrm{d} x\).
\end{align}
By (\ref{4.68}) and (\ref{5.9}), 
\begin{align*}
	&\quad\div(\omw^{(p)}_{q+1}\otimes\omw^{(p)}_{q+1}-\dw^{(p)}_{q+1}\otimes\dw^{(p)}_{q+1}+\Ru^{u}_{\ell})+\p_{t}\omw^{(o)}_{q+1}\notag\\
	&=(1-\Theta_{q+1}^{2})\div \Ru^{u}_{\ell}+\p_{t}\Theta_{q+1}^{2}w^{(o)}_{q+1}+\Theta_{q+1}^{2}\div(\varrho_{u}\mathrm{Id} )+\Theta_{q+1}^{2}\div\(\sum_{k\in\Lambda_u}a_{(k)}^{2}g_{(k)}^{2}\mathbb{P}_{\neq 0}(W_{(k)}\otimes W_{(k)})\)\notag\\
		&\quad+\Theta_{q+1}^{2}\div\(\sum_{k\in\Lambda_B}a_{(k)}^{2}g_{(k)}^{2}\mathbb{P}_{\neq 0}(W_{(k)}\otimes W_{(k)}-D_{(k)}\otimes D_{(k)})\)\notag\\
		&\quad-\Theta_{q+1}^{2}\sigma^{-1}\sum_{k\in\Lambda_u}\mathbb{P}_{\neq 0}\(h_{(k)}\fint_{\mathbb{T}^3}W_{(k)} \otimes W_{(k)} \mathrm{d} x \p_{t}\nabla (a_{(k)}^{2})\)\notag\\
		&\quad-\Theta_{q+1}^{2}\sigma^{-1}\sum_{k\in\Lambda_B}\mathbb{P}_{\neq 0}\(h_{(k)}\fint_{\mathbb{T}^3}W_{(k)} \otimes W_{(k)}-D_{(k)} \otimes D_{(k)} \mathrm{d} x \p_{t}\nabla (a_{(k)}^{2})\)\notag\\
		&\quad+(\nabla\Delta^{-1}\div)\Theta_{q+1}^{2}\sigma^{-1}\sum_{k\in\Lambda_u}\mathbb{P}_{\neq 0}\p_{t}\(h_{(k)} \fint_{\mathbb{T}^3} W_{(k)} \otimes W_{(k)} \mathrm{d}x \nabla(a_{(k)}^2)\)\notag\\
		&\quad+(\nabla\Delta^{-1}\div)\Theta_{q+1}^{2}\sigma^{-1}\sum_{k\in\Lambda_B}\mathbb{P}_{\neq 0}\p_{t}\(h_{(k)} \fint_{\mathbb{T}^3} W_{(k)} \otimes W_{(k)}-D_{(k)} \otimes D_{(k)} \mathrm{d}x \nabla(a_{(k)}^2)\).
\end{align*}
Hence the oscillation error can be defined by
\begin{align}
   \Ru^{u}_{osc}:=\Ru^{u}_{osc.1}+\Ru^{u}_{osc.2}+\Ru^{u}_{osc.3},   \label{5.14-5}
\end{align}
where
	\begin{align}
		&\Ru^{u}_{osc.1}
		:=(1-\Theta_{q+1}^{2}) \Ru^{u}_{\ell}+\mathcal{R}^{u}(\p_{t}\Theta_{q+1}^{2}w^{(o)}_{q+1}),\label{5.14-1}\\
		&\Ru^{u}_{osc.2}
		:=\Theta_{q+1}^{2}\sum_{k\in\Lambda_u}\mathcal{R}^{u}\mathbb{P}_{\neq 0}\(g_{(k)}^{2}\mathbb{P}_{\neq 0}(W_{(k)}\otimes W_{(k)})\nabla(a_{(k)}^{2})\)\notag\\
		&\quad\quad\quad\quad+\Theta_{q+1}^{2}\sum_{k\in\Lambda_B}\mathcal{R}^{u}\mathbb{P}_{\neq 0}\(g_{(k)}^{2}\mathbb{P}_{\neq 0}(W_{(k)}\otimes W_{(k)}-D_{(k)} \otimes D_{(k)})\nabla(a_{(k)}^{2})\),\label{5.14-2}\\
	&\Ru^{u}_{osc.3}
	:=-\Theta_{q+1}^{2}\sigma^{-1}\sum_{k\in\Lambda_u}\mathcal{R}^{u}\mathbb{P}_{\neq 0}\(h_{(k)}\fint_{\mathbb{T}^3}W_{(k)} \otimes W_{(k)} \mathrm{d} x \p_{t}\nabla (a_{(k)}^{2})\)\notag\\
		&\quad\quad\quad\quad-\Theta_{q+1}^{2}\sigma^{-1}\sum_{k\in\Lambda_B}\mathcal{R}^{u}\mathbb{P}_{\neq 0}\(h_{(k)}\fint_{\mathbb{T}^3}W_{(k)} \otimes W_{(k)}-D_{(k)} \otimes D_{(k)} \mathrm{d} x \p_{t}\nabla (a_{(k)}^{2})\).    \label{5.14-3}
\end{align}

Finally, we obtain the decomposition of the Reynolds stress at level $ q + 1 $:
$$\Ru^{u}_{q+1}:=\Ru^{u}_{lin}+\Ru^{u}_{corr}+\Ru^{u}_{com1}+\Ru^{u}_{com2}+\Ru^{u}_{osc}.$$

\subsection{Verification of growth estimate}
We aim to verify the inductive growth estimate in (\ref{2.12}).
To this end, let us estimate each term in the above choice of $ \Ru^{B}_{q+1}$ and $\Ru^{u}_{q+1}$.

\medskip
\paragraph{\bf $\bullet$ Linear errors.}
First, for the estimates of linear errors $\Ru^{u}_{lin}$ and $\Ru^{B}_{lin}$,
by Lemma \ref{Lemma Estimates of perturbations}, (\ref{4.69})  and  (\ref{4.70}),
	\begin{align}\label{5.18-1}
			 \|(\mathcal{R}^{u}\p_{t}\omw_{q+1}^{(p)+(c)},\mathcal{R}^{B}\p_{t}\dw_{q+1}^{(p)+(c)})\|_{C_{t}\mathcal{L}_x^{1}} 
		&\lesssim
		\| \Theta_{q+1}\|_{C_{t}^{1}} \|(w_{q+1}^{(p)+(c)}, d_{q+1}^{(p)+(c)})\|_{\mathcal{C}_{t,x}^{1}}
		\notag\\
		&\lesssim
		\varsigma_{q}^{-1}\((1+\J^{11})\lambda_{q+1}^{4\a}+(1+\J^{21})\lambda_{q+1}^{4\a-1}\)\notag\\
		&\lesssim
		(1+\J^{21})\lambda_{q+1}^{4\a}\varsigma_{q}^{-1}.
\end{align}
By  (\ref{4.69}), (\ref{4.71}), Lemma \ref{Lemma Estimates of perturbations}
and the Sobolev embedding $H^{2}_{x}(\T^{3})\hookrightarrow H^{2\a-1}_{x}(\T^{3})$
with $\a\in[1, \frac{3}{2})$, we obtain
	\begin{align}\label{5.18-2}
	&\quad\|\(\mathcal{R}^{u}(\nu(-\Delta)^{\a}w_{q+1}),\mathcal{R}^{B}(\nu(-\Delta)^{\a}d_{q+1})\)\|_{C_{t}\mathcal{L}_x^{1}}\notag\\
	&\lesssim
\|(w_{q+1},d_{q+1})\|_{C_{t}\mathcal{H}_x^{2}}\notag\\
	&\lesssim
 (1+\J^{16})\ell^{-22}\laq^{2}\tau^{\frac{1}{2}}+(1+\J^{26})\ell^{-43}\laq\tau^{\frac{1}{2}}+(1+\J^{27})\ell^{-107}\sigma^{-1}\notag\\
 &\lesssim
 (1+\J^{27})\ell^{-22}\laq^{2}\tau^{\frac{1}{2}},
\end{align}
where we also used the boundeness of Calderon-Zygmund operators in spaces $L^{p}$ for  $p>1$ in the first inequality. 
Moreover, by Lemma \ref{Lemma Estimates of perturbations} and H\"{o}lder's inequality, the remaining linear errors in (\ref{5.4}) and \eqref{5.11}  can be bounded by
	\begin{align}\label{5.18-3}
	&\quad\|(z_{1}, z_{2})\|_{C_{t}\mathcal{L}_x^{2}}
	+(\|(\Bu_{q},\Bb_{q})\|_{C_{t}\mathcal{L}_x^{2}}+\|(z_{1},z_{2})\|_{C_{t} L_{x}^{2}})\|(w_{q+1}, d_{q+1})\|_{C_{t} \mathcal{L}_{x}^{2}}\notag\\
	&\lesssim
    \|(z_{1}, z_{2})\|_{C_{t}\mathcal{L}_x^{2}}
    +(\|(\Bu_{q},\Bb_{q})\|_{\mathcal{C}_{t,x} }+\|(z_{1},z_{2})\|_{C_{t} \mathcal{L}_{x}^{2}})(1+\J^{17})\ell^{-22}\tau^{\frac{1}{2}}.
\end{align}
Thus, we conclude from the above estimates \eqref{5.18-1}, \eqref{5.18-2}  and \eqref{5.18-3} that
	\begin{align}\label{5.18-4}
	\|(\Ru^{u}_{lin},\Ru^{B}_{lin})\|_{C_{t}\mathcal{L}_x^{1}}
	&\lesssim
	(1+\J^{27})\lambda_{q+1}^{4\a+1}+\|(z_{1}, z_{2})\|_{C_{t}\mathcal{L}_x^{2}}+(\|(\Bu_{q},\Bb_{q})\|_{\mathcal{C}_{t,x} }+\|(z_{1},z_{2})\|_{C_{t} \mathcal{L}_{x}^{2}})(1+\J^{17})\lambda_{q+1}^{\a+1}. 
\end{align}

\medskip
\paragraph{\bf $\bullet$  Oscillation errors.}
By (\ref{4.69}), (\ref{5.7-1}), (\ref{5.14-1}) and Lemma \ref{Lemma Estimates of perturbations}, we  get
\begin{align*}
	\| (\Ru^{u}_{osc.1}, \Ru^{B}_{osc.1})\|_{C_{t}\mathcal{L}_{x}^{1}}
&\lesssim
\|1-\Theta_{q+1}^{2}\|_{C_{t}}\J+\|\Theta_{q+1}^{2}\|_{C_{t}^{1}}\|(w^{(o)}_{q+1},d^{(o)}_{q+1})\|_{C_{t} \mathcal{L}_{x}^{2}}
\notag\\
&\lesssim
 \J+\varsigma_{q}^{-1}(1+\J^{17})\ell^{-65}\sigma^{-1}\notag\\
 &\lesssim
  (1+\J^{17})\lambda_{q+1}.
\end{align*}

Moreover, by (\ref{4.69}), (\ref{5.7-2}), (\ref{5.14-2}) and Lemmas \ref{Lemma spacial building blocks}-\ref{Lemma Velocity amplitudes},
\begin{align*}
	\| (\Ru^{u}_{osc.2}, \Ru^{B}_{osc.2})\|_{C_{t}\mathcal{L}_{x}^{1}}
	&\lesssim
	\|\Theta_{q+1}^{2}\|_{C_{t}}
	\sum_{k\in\Lambda_B\cup\Lambda_u}\|g_{(k)}^{2}\|_{C_{t}}\|W_{(k)}\|^{2}_{L_{x}^{\infty}}\|a_{(k)}^{2}\|_{C_{t,x}^{1}}\notag\\
	&\quad+\|\Theta_{q+1}^{2}\|_{C_{t}}\sum_{k\in\Lambda_B}\|g_{(k)}^{2}\|_{C_{t}}\|D_{(k)}\|^{2}_{L_{x}^{\infty}}\|a_{(k)}^{2}\|_{C_{t,x}^{1}}\notag\\
	&\quad+	
	\|\Theta_{q+1}^{2}\|_{C_{t}}
	\sum_{k\in\Lambda_B}\|g_{(k)}^{2}\|_{C_{t}}\|W_{(k)}\|_{L_{x}^{\infty}}\|D_{(k)}\|_{L_{x}^{\infty}}\|a_{(k)}^{2}\|_{C_{t,x}^{1}}\notag\\
	&\lesssim
	\tau r_{\perp}^{-1}(1+\J^{17})\ell^{-65}
	\lesssim
	(1+\J^{17})\lambda_{q+1}^{4\a-1},
\end{align*}
where   the last  step was due to (\ref{2.3}) and (\ref{3.2-1}).

Regarding the third component, by (\ref{4.69}), (\ref{4.19}),  (\ref{5.7-3}), (\ref{5.14-3})  and Lemmas \ref{Lemma Magnetic amplitudes} and \ref{Lemma Velocity amplitudes},
\begin{align*}  
	\| (\Ru^{u}_{osc.3}, \Ru^{B}_{osc.3})\|_{C_{t}\mathcal{L}_{x}^{1}}
	&\lesssim \| (\Ru^{u}_{osc.3}, \Ru^{B}_{osc.3})\|_{C_{t}\mathcal{L}_{x}^{2}}\notag\\
	 &\lesssim
		\sigma^{-1}\|\Theta_{q+1}^{2}\|_{C_{t}}\sum_{k\in\Lambda_u\cup\Lambda_B}\|h_{(k)}\|_{C_{t}}
	            \|a_{(k)}^{2}\|_{C_{t,x}^{2}}   \notag\\
	 &\lesssim
		(1+\J^{22})\ell^{-86}\sigma^{-1}
		 \lesssim
		(1+\J^{22})\lambda_{q+1}.
\end{align*}

Thus, taking into account (\ref{5.7-5}) and (\ref{5.14-5}),  we conclude that
\begin{align}  \label{5.35}
		\|(\Ru^{u}_{osc},\Ru^{B}_{osc})\|_{C_{t}\mathcal{L}_x^{1}}
		&\lesssim
		(1+\J^{17})\lambda_{q+1}+  (1+\J^{17})\lambda_{q+1}^{4\a-1}+(1+\J^{22})\lambda_{q+1}\notag\\
		&\lesssim
		(1+\J^{22})\lambda_{q+1}^{4\a-1}.
\end{align}

\medskip
\paragraph{\bf Correctors.}
By Lemma \ref{Lemma Estimates of perturbations}, (\ref{4.69}), (\ref{5.5}) and (\ref{5.13}),  
\begin{align}\label{5.39}
	\|(\Ru^{u}_{corr}, \Ru^{B}_{corr})\|_{C_{t}\mathcal{L}_{x}^{1}}
		&\lesssim
	\|(\omw_{q+1}^{(c)+(o)},  \dw_{q+1}^{(c)+(o)})\|_{\mathcal{C}_{t,x}}(\|(\omw_{q+1}^{(p)}, \dw_{q+1}^{(p)})\|_{\mathcal{C}_{t,x}}+\|(w_{q+1}, d_{q+1})\|_{\mathcal{C}_{t,x}})
		\notag\\
		&\lesssim
	\((1+\J^{16})\lambda_{q+1}^{2\a-1}+(1+\J^{22})\lambda_{q+1}\)\notag\\
		&\quad\times
		\((1+\J^{6})\lambda_{q+1}^{2\a}+(1+\J^{16})\lambda_{q+1}^{2\a-1}+(1+\J^{22})\lambda_{q+1}\)\notag\\
		&\lesssim
		(1+\J^{44})\lambda_{q+1}^{4\a-1}.
\end{align}

\medskip
\paragraph{\bf  $\bullet$ Commutators.} For the commutator terms, by
 (\ref{3.7}) and (\ref{3.8}), we have
	\begin{align}\label{5.44}
		\|(\Ru^{u}_{com1},\Ru^{B}_{com1})\|_{C_{t}\mathcal{L}_x^{1}}
	&\lesssim \|(\Bu_{q},\Bb_{q})\|_{\mathcal{C}_{t,x}}^{2}+\|(z_{1},z_{2})\|_{C_{t}\mathcal{L}_{x}^{2}}^{2}.
\end{align}
Moreover, by (\ref{5.6}) and  (\ref{5.13}),
\begin{align}\label{5.47}
	\|(\Ru^{u}_{com2},\Ru^{B}_{com2})\|_{C_{t}\mathcal{L}_x^{1}}
	&\lesssim
	\|(\Bu_{q+1},\Bb_{q+1})\|_{\mathcal{C}_{t,x}}^{2}+\|(z_{1},z_{2})\|_{C_{t} \mathcal{L}_{x}^{2}}^{2}.
\end{align}

\medskip
\paragraph{\bf  $\bullet$  Verification of inductive estimate  (\ref{2.12}).}
Summing up all the above estimates (\ref{5.18-4}), (\ref{5.35}), (\ref{5.39}), (\ref{5.44}) and  (\ref{5.47}), we conclude that
\begin{align}\label{5.53-0}
&\quad\|(\Ru^{u}_{q+1},\Ru^{B}_{q+1})\|_{C_{t}\mathcal{L}_x^{1}}\notag\\
&\lesssim
	(1+\J^{44})\lambda_{q+1}^{4\a+1}
	+\|(z_{1},z_{2})\|_{C_{t}\mathcal{L}_x^{2}}
	+(\|(\Bu_{q},\Bb_{q})\|_{\mathcal{C}_{t,x} }+\|(z_{1},z_{2})\|_{C_{t} \mathcal{L}_{x}^{2}})(1+\J^{17})\lambda_{q+1}^{\a+1}\notag\\           
	&\quad+	\|(\Bu_{q},\Bb_{q})\|_{\mathcal{C}_{t,x}}^{2}+\|(z_{1},z_{2})\|_{C_{t}\mathcal{L}_{x}^{2}}^{2}  
	+\|(\Bu_{q+1},\Bb_{q+1})\|_{\mathcal{C}_{t,x}}^{2}. 
\end{align}
Taking the $m$-th moment in (\ref{5.53-0}) and using (\ref{2.1}), (\ref{2.11}), (\ref{3.2-1}), (\ref{4.82}) and H\"{o}lder's inequality, we obtain
\begin{align*}\label{5.53}
	\|(\Ru^{u}_{q+1},\Ru^{B}_{q+1})\|_{\Lo^{m}_{\om}C_{t}\mathcal{L}_x^{1}}
		&\lesssim
		(1+\|\J\|^{44}_{\Lo^{44m}_{\om}})\lambda_{q+1}^{4\a+1}
		+\|(z_{1},z_{2})\|_{\Lo^{2m}_{\om}C_{t}\mathcal{L}_{x}^{2}}\notag\\
		&\quad
		+(\|(\Bu_{q},\Bb_{q})\|_{\Lo^{2m}_{\om}\mathcal{C}_{t,x}}+\|(z_{1},z_{2})\|_{\Lo^{2m}_{\om}C_{t} \mathcal{L}_{x}^{2}})(1+\|\J\|^{17}_{\Lo^{34m}_{\om}})\lambda_{q+1}^{\a+1}\notag\\
		&\quad+\|(\Bu_{q},\Bb_{q})\|_{\Lo^{2m}_{\om}\mathcal{C}_{t,x}}^{2}
			+\|(z_{1},z_{2})\|_{\Lo^{2m}(\om;C_{t}\mathcal{L}_{x}^{2})}^{2}  
		+\|(\Bu_{q+1},\Bb_{q+1})\|_{\Lo^{2m}(\om;\mathcal{C}_{t,x})}^{2} 
		\notag\\
		&\leq
		\lambda_{q+1}^{(4\a+5)}(8mL^{2}80^{q+1})^{80^{q+1}},
\end{align*}
which verifies the inductive estimate (\ref{2.12}) at level $q+1$.

\subsection{Verification of decay estimate: away from the initial time}\label{away from zero}
In order to prove the decay estimate \eqref{2.13} at level $q+1$,
let us first consider the difficult regime away from the initial time,
that is, $t\in ({\varsigma_{q}}/{2}, T]$.

 Choose
\begin{align*}
		\rho:=\frac{2 \alpha-2+10 \varepsilon}{2 \alpha-2+9 \varepsilon} \in(1,2),
\end{align*}
where $\varepsilon$ is given by (\ref{2.3}). 
Note that 
\begin{align*}
		(2-2 \alpha-10 \varepsilon)(\frac{1}{\rho}-\frac{1}{2})=1-\alpha-4 \varepsilon,
\end{align*}
and
\begin{equation}\label{5.17}
	\begin{aligned}
		r_{\perp}^{\frac{1}{\rho}-\frac{1}{2}}=\lambda_{q+1}^{1-\alpha-4 \varepsilon}.
\end{aligned}\end{equation}

\paragraph{\bf $\bullet$ Linear errors.}
Let us first consider the linear errors and estimate 
\begin{align*}
	&\quad\|(\mathcal{R}^{u}\p_{t}\omw_{q+1}^{(p)+(c)},\mathcal{R}^{B}\p_{t}\dw_{q+1}^{(p)+(c)})\|_{L_{(\frac{\varsigma_{q}}{2}, T]}^{1}\mathcal{L}_x^{\rho}} \notag\\
		&\leq
	\big\|\big(\mathcal{R}^{u}((\p_t \Theta_{q+1}) w_{q+1}^{(p)+(c)}),\mathcal{R}^{B}((\p_t \Theta_{q+1}) d_{q+1}^{(p)+(c)}) \big)\big\|_{L_{t}^{1}\mathcal{L}_x^{\rho}}
	+\big\|\big(\mathcal{R}^{u}( \Theta_{q+1} \partial_t w_{q+1}^{(p)+(c)}), \mathcal{R}^{B}( \Theta_{q+1} \partial_t d_{q+1}^{(p)+(c)})\big)\big\|_{L_{t}^{1}\mathcal{L}_x^{\rho}} \notag\\
	&=:K_{1}+K_{2}.
\end{align*}
Note that, by ({\ref{4.69}}), (\ref{4.73-1}) and (\ref{4.73-2}),
\begin{align*}
		\quad K_{1}&\lesssim
		\|\Theta_{q+1}\|_{C_{t}^{1}} \|(w_{q+1}^{(p)+(c)},d_{q+1}^{(p)+(c)})\|_{L_{t}^{1}\mathcal{L}_x^{\rho}}\notag\\
		&\lesssim
		\varsigma_{q}^{-1}(1+\J^{6})\ell^{-22}r_{\perp}^{\frac{1}{\rho}-\frac{1}{2}}\tau^{-\frac{1}{2}}
		+	\varsigma_{q}^{-1}(1+\J^{16})\ell^{-43}\laq^{-1}r_{\perp}^{\frac{1}{\rho}-\frac{1}{2}}\tau^{-\frac{1}{2}}\notag\\
		&\lesssim
		(1+\J^{16})\ell^{-22}\varsigma_{q}^{-1}r_{\perp}^{\frac{1}{\rho}-\frac{1}{2}}\tau^{-\frac{1}{2}}.
	\end{align*}
Moreover, by Lemmas \ref{Lemma spacial building blocks}-\ref{Lemma Velocity amplitudes}, (\ref{4.66}) and ({\ref{4.69}}),
\begin{align*}
    \quad K_{2}
	&\lesssim
	\sum_{k \in \Lambda_u \cup \Lambda_B}\|\mathcal{R}^{u}\curl\curl\p_{t}(a_{(k)}g_{(k)}W_{(k)}^{c})\|_{L_{t}^{1}L_x^{\rho}}
	+\sum_{k \in  \Lambda_B}\|\mathcal{R}^{B}\curl\curl\p_{t}(a_{(k)}g_{(k)}D_{(k)}^{c})\|_{L_{t}^{1}L_x^{\rho}}\notag\\
	&\lesssim
	\sum_{k \in \Lambda_u \cup \Lambda_B}(\|\p_{t}g_{(k)}\|_{L_{t}^{1}}\|a_{(k)}\|_{C_{t,x}^{1}}\|W_{(k)}^{c}\|_{W_{x}^{1,\rho}}+\|g_{(k)}\|_{L_{t}^{1}}\|a_{(k)}\|_{C_{t,x}^{2}}\|W_{(k)}^{c}\|_{W_{x}^{1,\rho}})\notag\\
	&\quad+
	\sum_{k \in   \Lambda_B}(\|\p_{t}g_{(k)}\|_{L_{t}^{1}}\|a_{(k)}\|_{C_{t,x}^{1}}\|D_{(k)}^{c}\|_{W_{x}^{1,\rho}}
	+\|g_{(k)}\|_{L_{t}^{1}}\|a_{(k)}\|_{C_{t,x}^{2}}\|D_{(k)}^{c}\|_{W_{x}^{1,\rho}})\notag\\
	&\lesssim (1+\J^{22})(\sigma\tau^{\frac{1}{2}}\ell^{-43}\laq^{-1}r_{\perp}^{\frac{1}{\rho}-\frac{1}{2}}+\tau^{-\frac{1}{2}}\ell^{-64}\laq^{-1}r_{\perp}^{\frac{1}{\rho}-\frac{1}{2}})\notag\\
	&\lesssim
	(1+\J^{22})\ell^{-43}\laq^{-1}r_{\perp}^{\frac{1}{\rho}-\frac{1}{2}}\sigma\tau^{\frac{1}{2}}.
\end{align*}
Thus, we obtain
\begin{align}\label{5.54}
	 \|(\mathcal{R}^{u}\p_{t}\omw_{q+1}^{(p)+(c)},\mathcal{R}^{B}\p_{t}\dw_{q+1}^{(p)+(c)})\|_{L_{(\frac{\varsigma_{q}}{2}, T]}^{1}\mathcal{L}_x^{\rho}}  
		&\lesssim
			(1+\J^{16})\ell^{-22}\varsigma_{q}^{-1}r_{\perp}^{\frac{1}{\rho}-\frac{1}{2}}\tau^{-\frac{1}{2}}
			+(1+\J^{22})\ell^{-43}\laq^{-1}r_{\perp}^{\frac{1}{\rho}-\frac{1}{2}}\sigma\tau^{\frac{1}{2}}\notag\\
			&\lesssim
			(1+\J^{22})\ell^{-43}\laq^{-1}r_{\perp}^{\frac{1}{\rho}-\frac{1}{2}}\sigma\tau^{\frac{1}{2}}.
\end{align}

\medskip
\paragraph{\it Control of hyper-viscosity: }
For the hyper-viscosity $(-\Delta)^{\a}$,
using interpolation and Lemma \ref{Lemma Estimates of perturbations}, we get
\begin{align*}
	 \|(\mathcal{R}^{u}\nu(-\Delta)^{\a}\omw_{q+1}^{(p)},\mathcal{R}^{B}\nu(-\Delta)^{\a}\dw_{q+1}^{(p)})\|_{L^{1}_{(\frac{\varsigma_{q}}{2}, T]}\mathcal{L}_x^{\rho}}  
	&\lesssim \|\Theta_{q+1}\|_{C_{t}}\|(w_{q+1}^{(p)},d_{q+1}^{(p)})\|^{\frac{4-2\a}{3}}_{L_{t}^{1}\mathcal{L}_x^{\rho}}\|(w_{q+1}^{(p)},d_{q+1}^{(p)})\|^{\frac{2\a-1}{3}}_{L_{t}^{1}\mathcal{W}_x^{3,\rho}}\notag\\
	&\lesssim (1+\J^{5(2\a-1)+6})\ell^{-22}\laq^{2\a-1}r_{\perp}^{\frac{1}{\rho}-\frac{1}{2}}\tau^{-\frac{1}{2}}.
\end{align*}
Similarly, we have
\begin{align*}
	\|(\mathcal{R}^{u}\nu(-\Delta)^{\a}\omw_{q+1}^{(c)},\mathcal{R}^{B}\nu(-\Delta)^{\a}\dw_{q+1}^{(c)})\|_{L^{1}_{(\frac{\varsigma_{q}}{2}, T]}\mathcal{L}_x^{\rho}}
	\lesssim (1+\J^{5(2\a-1)+16})\ell^{-43}\laq^{2\a-2}r_{\perp}^{\frac{1}{\rho}-\frac{1}{2}}\tau^{-\frac{1}{2}}, 
	\end{align*}
and
	\begin{align*}
	\|(\mathcal{R}^{u}\nu(-\Delta)^{\a}\omw_{q+1}^{(o)},\mathcal{R}^{B}\nu(-\Delta)^{\a}\dw_{q+1}^{(o)})\|_{L^{1}_{(\frac{\varsigma_{q}}{2}, T]}\mathcal{L}_x^{\rho}}
	\lesssim
	(1+\J^{5(2\a-1)+17})\ell^{-21(2\a-1)-65}\sigma^{-1}.
\end{align*}
Since $\ell^{-22}\laq^{2\a-1}r_{\perp}^{\frac{1}{\rho}-\frac{1}{2}}\tau^{-\frac{1}{2}}=\ell^{-22}\lambda_{q+1}^{-4\varepsilon}\leq \ell^{-107}\sigma^{-1}$, we thus get
\begin{align}\label{5.56}
	&\quad\|(\mathcal{R}^{u}\nu(-\Delta)^{\a}w_{q+1},\mathcal{R}^{B}\nu(-\Delta)^{\a}d_{q+1})\|_{L^{1}_{(\frac{\varsigma_{q}}{2}, T]}\mathcal{L}_x^{\rho}}  \notag\\
	&\lesssim (1+\J^{5(2\a-1)+6})\ell^{-22}\laq^{2\a-1}r_{\perp}^{\frac{1}{\rho}-\frac{1}{2}}\tau^{-\frac{1}{2}}+(1+\J^{5(2\a-1)+16})\ell^{-43}\laq^{2\a-2}r_{\perp}^{\frac{1}{\rho}-\frac{1}{2}}\tau^{-\frac{1}{2}}\notag\\
	&\quad+(1+\J^{5(2\a-1)+17})\ell^{-21(2\a-1)-65}\sigma^{-1}\notag\\
		&\lesssim
		(1+\J^{27})\ell^{-107}\sigma^{-1}.
\end{align}

\medskip
\paragraph{\it Control of noise term:}
In order to control the noise term, using  the standard mollification estimates and (\ref{3.4-00}), we derive
 \begin{align*}
 	&\quad\|\mathcal{R}^{u}(z_{1,\ell}-z_{1,q+1})\|_{L_{(\frac{\varsigma_{q}}{2}, T]}^{1}\mathcal{L}_x^{\rho}} +	\|\mathcal{R}^{B}(z_{2,\ell}-z_{2,q+1})\|_{L_{(\frac{\varsigma_{q}}{2}, T]}^{1}\mathcal{L}_x^{\rho}} \notag\\
 		&\lesssim
 \|(z_{1,\ell}^{u}-z_{1}^{u}, z_{2,\ell}^{B}-z_{2}^{B})\|_{L_{(\frac{\varsigma_{q}}{2}, T]}^{1}\mathcal{L}_x^{\rho}}+\|(Z_{1,\ell}-Z_{1,q}, Z_{2,\ell}-Z_{2,q})\|_{L_{t}^{1}\mathcal{L}_x^{\rho}} +	\|(Z_{1,q}-Z_{1,q+1},Z_{2,q}-Z_{2,q+1})\|_{L_{t}^{1}\mathcal{L}_x^{\rho}} \notag\\
 	&\lesssim
 \|(z_{1,\ell}^{u}-z_{1}^{u}*_{x}\varrho_{\ell}, z_{2,\ell}^{B}-z_{2}^{B}*_{x}\varrho_{\ell})\|_{L_{(\frac{\varsigma_{q}}{2}, T]}^{1}\mathcal{L}_x^{\rho}}+\|(z_{1}^{u}*_{x}\varrho_{\ell}-z_{1}^{u}, z_{2}^{B}*_{x}\varrho_{\ell}-z_{2}^{B})\|_{L_{(\frac{\varsigma_{q}}{2}, T]}^{1}\mathcal{L}_x^{\rho}}\notag\\
 &\quad+\|(Z_{1,\ell}-Z_{1,q}*_{x}\varrho_{\ell}, Z_{2,\ell}-Z_{2,q}*_{x}\varrho_{\ell})\|_{L_{t}^{1}\mathcal{L}_x^{\rho}}  +\|(Z_{1,q}*_{x}\varrho_{\ell}-Z_{1,q}, Z_{2,q}*_{x}\varrho_{\ell}-Z_{2,q})\|_{L_{t}^{1}\mathcal{L}_x^{\rho}}\notag\\
 &\quad+	\|(Z_{1,q}-Z_{1,q+1},Z_{2,q}-Z_{2,q+1})\|_{L_{t}^{1}\mathcal{L}_x^{\rho}} \notag\\
 	&\lesssim
 	\ell^{\frac{1}{2}}(\|(z_{1}^{u},z_{2}^{B})\|_{C^{\frac{1}{2}}_{(\frac{\varsigma_{q}}{2}-\ell, T]} \mathcal{L}_{x}^{2}}+\|(z_{1}^{u},z_{2}^{B})\|_{C_{(\frac{\varsigma_{q}}{2}, T]} \mathcal{H}_{x}^{\frac{1}{2}}})+
 	\ell^{\frac{1}{2}-\delta}\|(Z_{1},Z_{2})\|_{C^{\frac{1}{2}-\delta}_{t}\mathcal{L}_{x}^{2}}+\la^{-15(1-\delta)}\|(Z_{1},Z_{2})\|_{C_{t} \mathcal{H}_{x}^{1-\delta}}.
 \end{align*} 
 
 Next, let us first derive the estimates for $\|(z_{1}^{u},z_{2}^{B})\|_{C_{(\frac{\varsigma_{q}}{2}, T]} \mathcal{H}_{x}^{s_{1}}}$ and $	\|(z_{1}^{u},z_{2}^{B})\|_{C_{(\frac{\varsigma_{q}}{2}-\ell, T]}^{s_{2}}\mathcal{L}_{x}^{2}}$ when  $s_{1}, s_{2}\in (0,1)$. 
For this purpose, by the heat-semigroup estimate in Proposition \ref{Proposition convolution}, \eqref{1.4.2} and \eqref{in},   for  $s_{1}\in (0,1)$,
\begin{align}\label{5.22.01}
	\|(z_{1}^{u},z_{2}^{B})\|_{C_{(\frac{\varsigma_{q}}{2}, T]} \mathcal{H}_{x}^{s_{1}}}\lesssim \(1+(\frac{\varsigma_{q}}{2})^{-\frac{s_{1}}{2\a}}\)	\|(u_{0},B_{0})\|_{ \mathcal{L}_{x}^{2}}\lesssim (1+\varsigma_{q}^{-\frac{s_{1}}{2\a}})M.
\end{align}

Concerning $	\|(z_{1}^{u},z_{2}^{B})\|_{C_{(\frac{\varsigma_{q}}{2}-\ell, T]}^{s_{2}}\mathcal{L}_{x}^{2}}$,
for $t_{1}, t_{2}\in (\frac{\varsigma_{q}}{2}-\ell, T]$ and $s_{2}\in (0,1)$,  we have
\begin{align*}
	\|z_{1}^{u}(t_{1})-z_{1}^{u}(t_{2})\|_{L_{x}^{2}}
	&=\|e^{t_{2}(-\nu(-\Delta)^{\a}-I)}\int_{0}^{t_{1}-t_{2}}(-\nu(-\Delta)^{\a}-I)e^{r(-(-\Delta)^{\a}-I)}u_{0}\mathrm{d} r\|_{L_{x}^{2}}\\
	&\lesssim
	\|\int_{0}^{t_{1}-t_{2}}(-\Delta)^{\a}e^{r(-\nu(-\Delta)^{\a}-I)}e^{t_{2}(-\nu(-\Delta)^{\a}-I)}u_{0}\mathrm{d} r\|_{L_{x}^{2}}\\
	&\quad+
	\|\int_{0}^{t_{1}-t_{2}}e^{r(-\nu(-\Delta)^{\a}-I)}e^{t_{2}(-\nu(-\Delta)^{\a}-I)}u_{0}\mathrm{d} r\|_{L_{x}^{2}},
\end{align*}
which, via Proposition \ref{Proposition convolution},  can be bouned by, if
$\left|t_{1}-t_{2}\right|\leq1$,
\begin{align*}
	\|z_{1}^{u}(t_{1})-z_{1}^{u}(t_{2})\|_{L_{x}^{2}}
	&\lesssim
	\int_{0}^{t_{1}-t_{2}}r^{-(1-s_{2})}\|(-\Delta)^{ s_{2}\a}e^{t_{2}(-\nu(-\Delta)^{\a}-I)}u_{0}\|_{L_{x}^{2}}\mathrm{d} r+\left|t_{1}-t_{2}\right|\|u_{0}\|_{L_{x}^{2}} \\
	&\lesssim
	\left|t_{1}-t_{2}\right|^{s_{2}}t_{2}^{-s_{2}}\|u_{0}\|_{L_{x}^{2}}+\left|t_{1}-t_{2}\right|\|u_{0}\|_{L_{x}^{2}} \\
	&\lesssim
	\left|t_{1}-t_{2}\right|^{s_{2}}(1+t_{2}^{-s_{2}})M, 
\end{align*}
and if  $\left|t_{1}-t_{2}\right|\geq1$ 
\begin{align*}
	\|z_{1}^{u}(t_{1})-z_{1}^{u}(t_{2})\|_{L_{x}^{2}} \leq 2\|u_{0}\|_{L_{x}^{2}}
	\leq	2\left|t_{1}-t_{2}\right|^{s_{2}}M.
\end{align*}
These estimates hold  for $z_{2}^{B}$ as well.
Hence, we obtain that for $s_{2}\in (0,1)$, 
\begin{align}\label{5.22.04}
	\|(z_{1}^{u},z_{2}^{B})\|_{C_{(\frac{\varsigma_{q}}{2}-\ell, T]}^{s_{2}}\mathcal{L}_{x}^{2}}
	\lesssim \(1+(\frac{\varsigma_{q}}{2}-\ell)^{-s_{2}}\) M
	\lesssim (1+\varsigma_{q}^{-s_{2}}) M.
\end{align}
Thus, it follows from (\ref{5.22.01}) and (\ref{5.22.04}) that
\begin{align}\label{5.58}
		&\quad\|\mathcal{R}^{u}(z_{1,\ell}-z_{1,q+1})\|_{L_{(\frac{\varsigma_{q}}{2}, T]}^{1}\mathcal{L}_x^{\rho}} +	\|\mathcal{R}^{B}(z_{2,\ell}-z_{2,q+1})\|_{L_{(\frac{\varsigma_{q}}{2}, T]}^{1}\mathcal{L}_x^{\rho}} \notag\\
		&\lesssim
		\ell^{\frac{1}{2}}(1+\varsigma_{q}^{-\frac{1}{2}})
		 M+\la^{-15(1-\delta)}(\|(Z_{1},Z_{2})\|_{C^{\frac{1}{2}-\delta}_{t} \mathcal{L}_{x}^{2}}+\|(Z_{1},Z_{2})\|_{C_{t} \mathcal{H}_{x}^{1-\delta}}).
\end{align}

\medskip
\paragraph{\it Control of  remaining terms:}
for the remaining terms in (\ref{5.4}) and  (\ref{5.11}), estimating as in (\ref{5.56}), we have
\begin{align}\label{5.59-0}
		&\quad\|(\Bb_{\ell}+z_{2,\ell})\mathring{\otimes}w_{q+1}-w_{q+1}\mathring{\otimes}(\Bb_{\ell}+z_{2,\ell})
		-(\Bu_{\ell}+z_{1,\ell})\mathring{\otimes}d_{q+1}+d_{q+1}\mathring{\otimes}(\Bu_{\ell}+z_{1,\ell})\|_{L_{(\frac{\varsigma_{q}}{2}, T]}^{1}\mathcal{L}_x^{\rho}}  \notag\\
		&\quad+\|( \Bu_{\ell}+z_{1,\ell})\mathring{\otimes}w_{q+1}+w_{q+1}\mathring{\otimes}( \Bu_{\ell}+z_{1,\ell})
		-(\widetilde{B}_{\ell}+z_{2,\ell})\mathring{\otimes}d_{q+1}-d_{q+1}\mathring{\otimes}(\widetilde{B}_{\ell}+z_{2,\ell})\|_{L_{(\frac{\varsigma_{q}}{2}, T]}^{1}\mathcal{L}_x^{\rho}}  \notag\\
		&\lesssim
		\big(\|(\Bu_{q},\Bb_{q})\|_{C_{(\frac{\varsigma_{q}}{2}-\ell, T]} \mathcal{L}_{x}^{\infty}}+\|(z_{1,q},z_{2,q})\|_{C_{(\frac{\varsigma_{q}}{2}-\ell, T]} \mathcal{L}_{x}^{\infty}}\big)\|(w_{q+1},d_{q+1})\|_{L_{(\frac{\varsigma_{q}}{2}, T]}^{1}\mathcal{L}_x^{\rho}} \notag\\
		&\lesssim
		\big(\|(\Bu_{q},\Bb_{q})\|_{C_{t,x}}+\|(z_{1}^{u},z_{2}^{B})\|_{C_{(\frac{\varsigma_{q}}{2}-\ell, T]} \mathcal{L}_{x}^{\infty}}+\|(Z_{1,q},Z_{2,q})\|_{C_{t}\mathcal{L}_{x}^{\infty} }\big)(1+\J^{17})\ell^{-65}\sigma^{-1}.
\end{align}
Note that, by Proposition \ref{Proposition convolution}, (\ref{1.4.2}) and (\ref{in}),
\begin{align}\label{5.23.0}
	\|(z_{1}^{u},z_{2}^{B})\|_{C_{(\frac{\varsigma_{q}}{2}-\ell, T]} \mathcal{L}_{x}^{\infty}}\lesssim \varsigma_{q}^{-\frac{3}{4\a}}\|(u_{0},B_{0})\|_{ \mathcal{L}_{x}^{2}}
	\lesssim \varsigma_{q}^{-\frac{3}{4\a}}M,
\end{align}	
which yields that
\begin{align}\label{5.59}
\text{R.H.S.\ of\ (\ref{5.59-0})}
	\lesssim
	\big(\|(\Bu_{q},\Bb_{q})\|_{\mathcal{C}_{t,x}}+\varsigma_{q}^{-\frac{3}{4\a}}M+\|(Z_{1,q},Z_{2,q})\|_{C_{t}\mathcal{L}_{x}^{\infty} })(1+\J^{17}\big)\ell^{-65}\sigma^{-1}.
\end{align}

Now,  we conclude from  (\ref{5.54}), (\ref{5.56}), (\ref{5.58}) and  (\ref{5.59}) that
\begin{align*}
		\|(\Ru^{u}_{lin},\Ru^{B}_{lin})\|_{L_{(\frac{\varsigma_{q}}{2}, T]}^{1}\mathcal{L}_x^{1}}
		&\lesssim
	   	(1+\J^{22})\ell^{-43}r_{\perp}^{\frac{1}{\rho}-\frac{1}{2}}\laq^{-1}\sigma\tau^{\frac{1}{2}}
		+(1+\J^{27})\ell^{-107}\sigma^{-1} \\
		&\quad
		+\ell^{\frac{1}{2}}(1+\varsigma_{q}^{-\frac{1}{2}})
		 M+\la^{-15(1-\delta)}\big(\|(Z_{1},Z_{2})\|_{C^{\frac{1}{2}-\delta}_{t} \mathcal{L}_{x}^{2}}+\|(Z_{1},Z_{2})\|_{C_{t} \mathcal{H}_{x}^{1-\delta}}\big) \\
		&\quad+\big(\|(\Bu_{q},\Bb_{q})\|_{\mathcal{C}_{t,x}}+\varsigma_{q}^{-\frac{3}{4\a}}M+\|(Z_{1,q},Z_{2,q})\|_{C_{t} \mathcal{L}_{x}^{\infty}}\big)(1+\J^{17})\ell^{-65}\sigma^{-1},
\end{align*}
which, along with (\ref{2.2-2}), (\ref{2.6}), (\ref{2.11}), (\ref{2.12}) and  (\ref{3.2-1}), yields
\begin{align}\label{5.63}
		&\quad\|(\Ru^{u}_{lin},\Ru^{B}_{lin})\|_{\Lo^{r}_{\om} L_{(\frac{\varsigma_{q}}{2}, T]}^{1}\mathcal{L}_{x}^{1}}\notag\\
		&\lesssim
		 (1+\|\J\|^{22}_{\Lo^{22r}_{\om}})\ell^{-43}\laq^{-1}r_{\perp}^{\frac{1}{\rho}-\frac{1}{2}}\sigma\tau^{\frac{1}{2}}
	    +(1+\|\J\|^{27}_{\Lo^{27r}_{\om}})\ell^{-107}\sigma^{-1}+\ell^{\frac{1}{2}}(1+\varsigma_{q}^{-\frac{1}{2}}) M\notag\\
		&\quad+\la^{-15(1-\delta)}\big(\|(Z_{1},Z_{2})\|_{\Lo^{r}_{\om}C_{t}^{1/2-\delta}\mathcal{L}_{x}^{2} }+\|(Z_{1},Z_{2})\|_{\Lo^{r}_{\om}C_{t}\mathcal{H}_{x}^{1-\delta}}\big)\notag\\
		&\quad
		+\big(\|(\Bu_{q},\Bb_{q})\|_{\Lo^{2r}_{\om}\mathcal{C}_{t,x}}+\varsigma_{q}^{-\frac{3}{4\a}}M+\|\(Z_{1,q},Z_{2,q}\)\|_{\Lo^{2r}_{\om}C_{t}\mathcal{L}_{x}^{\infty}}\big)
	(1+\|\J\|^{17}_{\Lo^{34r}_{\om}})\ell^{-65}\sigma^{-1}\notag\\
		&\lesssim
		\lambda_{q}^{-14}+\lambda_{q+1}^{-\varepsilon}
		 \lesssim
		\dqqq.
\end{align}

\paragraph{\bf $\bullet$ Oscillation errors}
By (\ref{4.69}), (\ref{4.70}), (\ref{4.73-3}), (\ref{5.7-1}) and (\ref{5.14-1}),  we get
\begin{align*}
	\|(\Ru^{u}_{osc.1}, \Ru^{B}_{osc.1})\|_{L_{(\frac{\varsigma_{q}}{2}, T]}^{1}\mathcal{L}_{x}^{1}}
	&\lesssim
\|1-\Theta_{q+1}^{2}\|_{L_{t}^{1}}\J+\|\Theta_{q+1}^{2}\|_{C_{t}^{1}}\|(w^{(o)}_{q+1}, d^{(o)}_{q+1})\|_{L_{t}^{1} \mathcal{L}_{x}^{\rho}}\notag\\
	&\lesssim
	\varsigma_{q}\J+(1+\J^{17})\ell^{-65}\sigma^{-1}\varsigma_{q}^{-1}.
\end{align*}

Regarding the high-low spatial oscillation  error $(\Ru^{u}_{osc.2}, \Ru^{B}_{osc.2})$ in (\ref{5.7-2}) and (\ref{5.14-2}),
we apply the decoupling Lemma \ref{Lemma Decorrelation1} and Lemmas \ref{Lemma spacial building blocks}-\ref{Lemma Velocity amplitudes} to estimate
\begin{align*}
		&\quad\|(\Ru^{u}_{osc.2}, \Ru^{B}_{osc.2})\|_{L_{(\frac{\varsigma_{q}}{2}, T]}^{1} \mathcal{L}_{x}^{\rho}}\notag\\
		&\lesssim
		\sum_{k\in\Lambda_u}\|\Theta_{q+1}^{2}\|_{C_{t}}\|g_{(k)}^{2}\|_{L_t^{1}}
		\||\nabla|^{-1}\mathbb{P}_{\neq 0}\(\mathbb{P}_{\geq \frac{\laq r_{\perp}}{2} }(W_{(k)}\otimes W_{(k)})\)\nabla(a_{(k)}^{2})\|_{C_{t} L_{x}^{\rho}}\notag\\
		&\quad+\sum_{k\in\Lambda_B}\|\Theta_{q+1}^{2}\|_{C_{t}}\|g_{(k)}^{2}\|_{L_t^{1}}\|
|\nabla|^{-1}\mathbb{P}_{\neq 0}\(\mathbb{P}_{\geq \frac{\laq r_{\perp}}{2} }(W_{(k)}\otimes W_{(k)}-D_{(k)} \otimes D_{(k)})\)\nabla(a_{(k)}^{2})\|_{C_{t} L_{x}^{\rho}}\notag\\
		&\quad+\sum_{k\in\Lambda_B}\|\Theta_{q+1}^{2}\|_{C_{t}}\|g_{(k)}^{2}\|_{L_t^{1}}\||\nabla|^{-1}\mathbb{P}_{\neq 0}\(\mathbb{P}_{\geq \frac{\laq r_{\perp}}{2} }(D_{(k)}\otimes W_{(k)}-W_{(k)}\otimes D_{(k)})\nabla(a_{(k)}^{2})\)\|_{C_{t} L_{x}^{\rho}}\notag\\
		&\lesssim
		\sum_{k\in\Lambda_u\cup\Lambda_B}\|a_{(k)}^{2}\|_{C_{t,x}^{3}}(\laq r_{\perp})^{-1}\|\phi_{(k)}^{2}\|_{L_{x}^{\rho}}\notag\\
		&\lesssim
		(1+\J^{27})\ell^{-107}r_{\perp}^{\frac{1}{\rho}-2}\laq^{-1}.
\end{align*}
Moreover, by Lemmas \ref{Lemma temporal building blocks}-\ref{Lemma Velocity amplitudes}, (\ref{5.7-3}) and (\ref{5.14-3}),
\begin{align*}
		\|(\Ru^{u}_{osc.3}, \Ru^{B}_{osc.3})\|_{L_{(\frac{\varsigma_{q}}{2}, T]}^{1} \mathcal{L}_{x}^{\rho}}
		&\lesssim
		\sigma^{-1}\|\Theta_{q+1}^{2}\|_{C_{t}}\sum_{k\in\Lambda_u\cup\Lambda_B}\|h_{(k)}\|_{L_t^{1}}
		\|a_{(k)}^{2}\|_{C_{t,x}^{2}}\notag\\
		&\lesssim
		\sigma^{-1}\ell^{-86}(1+\J^{22}).
\end{align*}

Thus, it follows from (\ref{5.7-5}) and (\ref{5.14-5}) that
\begin{align*}
	\|(\Ru^{u}_{osc},\Ru^{B}_{osc})\|_{L_{(\frac{\varsigma_{q}}{2}, T]}^{1}\mathcal{L}_x^{1}}
	&\lesssim
	\varsigma_{q}\J+(1+\J^{17})\ell^{-65}\sigma^{-1}\varsigma_{q}^{-1}+(1+\J^{27})\ell^{-107}r_{\perp}^{\frac{1}{\rho}-2}\laq^{-1}+\sigma^{-1}\ell^{-86}(1+\J^{22}),
\end{align*}
which, via (\ref{2.2-2}), (\ref{2.12}) and (\ref{3.2-1}), yields that
\begin{align}\label{5.73}
	\|(\Ru^{u}_{osc},\Ru^{B}_{osc})\|_{\Lo^{r}_{\om}L_{(\frac{\varsigma_{q}}{2}, T]}^{1}\mathcal{L}_{x}^{1}}
	&\lesssim
\varsigma_{q}\|\J\|_{\Lo^{r}_{\om}}+(1+\|\J\|^{17}_{\Lo^{17r}_{\om}})\ell^{-65}\sigma^{-1}\varsigma_{q}^{-1}\notag\\
&\quad+(1+\|\J\|^{27}_{\Lo^{27r}_{\om}})\ell^{-107}r_{\perp}^{\frac{1}{\rho}-2}\laq^{-1}+\sigma^{-1}\ell^{-86}(1+\|\J\|^{22}_{\Lo^{22r}_{\om}})\notag\\
	&\lesssim
	\lambda_{q}^{-15}+\lambda_{q+1}^{-\varepsilon}
	\lesssim
	\dqqq.
\end{align}

\paragraph{\bf Corrector errors}
By (\ref{4.69}), (\ref{4.70}), (\ref{5.5}), (\ref{5.12}) and Lemma \ref{Lemma Estimates of perturbations},
\begin{align*}
		\|(\Ru^{u}_{corr},\Ru^{B}_{corr})\|_{L_{(\frac{\varsigma_{q}}{2}, T]}^{1}\mathcal{L}_x^{1}}
		&\lesssim
	\|(\omw_{q+1}^{(c)+(o)}, \dw_{q+1}^{(c)+(o)})\|_{L_t^{2}\mathcal{L}_x^{2}}	\big(\|(w_{q+1},d_{q+1})\|_{L_t^{2}\mathcal{L}_x^{2}}+\|(\omw_{q+1}^{(p)},\dw_{q+1}^{(p)})\|_{L_t^{2}\mathcal{L}_x^{2}}\big)
		\notag\\
		&\lesssim
		\((1+\J^{16})\ell^{-43}\laq^{-1}+(1+\J^{17})\ell^{-65}\sigma^{-1}\)\notag\\
		&\quad\quad\times
		\((1+\J^{6})\ell^{-22}+(1+\J^{16})\ell^{-43}\laq^{-1}+(1+\J^{17})\ell^{-65}\sigma^{-1}\)\notag\\
		&\lesssim
		(1+\J^{34})\ell^{-87}\sigma^{-1}.
\end{align*}
Taking into account  (\ref{2.2-2}), (\ref{2.12}) and (\ref{3.2-1}), we get
\begin{align}\label{5.77}
		\|(\Ru^{u}_{corr},\Ru^{B}_{corr})\|_{\Lo^{r}_{\om}L_{(\frac{\varsigma_{q}}{2}, T]}^{1}\mathcal{L}_{x}^{1}}
		&\lesssim
		(1+\|\J\|^{34}_{\Lo^{34r}_{\om}})\ell^{-87}\sigma^{-1}\notag\\
		&\lesssim
		\[\lambda_{q}^{(4\a+5)}(8\cdot 34rL^{2}80^{q})^{80^{q}}\]^{34}\la^{87\times 60}\laq^{-2\varepsilon}\notag\\
		&\lesssim
		\laq^{-\varepsilon}
		\lesssim
		\dqqq.
\end{align}

\paragraph{\bf Commutator errors}   
By the standard mollification estimates, (\ref{3.7}),  (\ref{3.8}), (\ref{5.22.01}) and (\ref{5.22.04}),
it holds that
\begin{align*}
		&\quad\|(\Ru^{u}_{com1},\Ru^{B}_{com1})\|_{L_{(\frac{\varsigma_{q}}{2}, T]}^{1}\mathcal{L}_x^{1}}\notag\\
		&\lesssim
			\big(\ell^{\frac{1}{2}-\delta}\|(\Bu_{q},\Bb_{q})\|_{\mathcal{C}_{t,x}^{1}}
			+\ell^{\frac{1}{2}-\delta}\|(z_{1},z_{2})\|_{C_{(\frac{\varsigma_{q}}{2}-\ell, T]}^{\frac{1}{2}-\delta}\mathcal{L}_{x}^{2}}
			+\ell^{1-\delta}\|(z_{1},z_{2})\|_{C_{(\frac{\varsigma_{q}}{2}, T]} \mathcal{H}_{x}^{1-\delta}}\big)\notag\\
		&\quad
		\times\big(\|(\Bu_{q},\Bb_{q})\|_{\mathcal{C}_{t,x}}+\|(z_{1},z_{2})\|_{C_{t} \mathcal{L}_{x}^{2}}\big)\notag\\
		&\lesssim
		\ell^{\frac{1}{2}-\delta}	\big(\|(\Bu_{q},\Bb_{q})\|_{\mathcal{C}_{t,x}^{1}}+	(1+\varsigma_{q}^{-(\frac{1}{2}-\delta)})M
		+\|(Z_{1},Z_{2})\|_{C_{t}^{\frac{1}{2}-\delta} \mathcal{L}_{x}^{2}}+\|(Z_{1},Z_{2})\|_{C_{t} \mathcal{H}_{x}^{1-\delta}}\big)\notag\\
		&\quad
		\times\big(\|(\Bu_{q},\Bb_{q})\|_{\mathcal{C}_{t,x}}+\|(z_{1},z_{2})\|_{C_{t} \mathcal{L}_{x}^{2}}\big).
\end{align*}
Thus,  by (\ref{2.2-2}),  (\ref{2.11}), (\ref{2.12}) and (\ref{3.2-1}), we get
\begin{align}\label{5.83}
 \|(\Ru^{u}_{com1},\Ru^{B}_{com1})\|_{\Lo^{r}_{\om}L_{(\frac{\varsigma_{q}}{2}, T]}^{1}\mathcal{L}_{x}^{1}}
	 \lesssim \lambda_{q}^{-29}(\lambda_{q}^{4\a+3}+\lambda_{q}^{15})\lambda_{q}^{2\a+3}
	\lesssim
	\lambda_{q}^{-8}
	\lesssim
	\dqqq.		
\end{align}

Concerning the $L_{(\frac{\varsigma_{q}}{2}, T]}^{1}L_x^{1}$-norm of $ \Ru^{u}_{com2} $ and $ \Ru^{B}_{com2} $,
by (\ref{4.72}), (\ref{5.6}), (\ref{5.13}), Lemma \ref{Lemma Estimates of perturbations} and mollification estimate,
\begin{align*}
		&\quad\|(\Ru^{u}_{com2},\Ru^{B}_{com2})\|_{L_{(\frac{\varsigma_{q}}{2}, T]}^{1}\mathcal{L}_x^{1}}\notag\\
		&\lesssim
		\big(\|(\Bu_{q+1},\Bb_{q+1})\|_{L_{(\frac{\varsigma_{q}}{2}, T]}^{1}\mathcal{L}_{x}^{2}}+\|(z_{1,\ell},z_{2,\ell})\|_{L_{(\frac{\varsigma_{q}}{2}, T]}^{1}\mathcal{L}_{x}^{2}}+\|(z_{1,q+1},z_{2,q+1})\|_{L_{(\frac{\varsigma_{q}}{2}, T]}^{1}\mathcal{L}_{x}^{2}}\big)\notag\\
		&\quad\times
		\big(\|(z_{1,\ell}^{u}- z_{1}^{u}, z_{2,\ell}^{B}-z_{2}^{B})\|_{C_{(\frac{\varsigma_{q}}{2}, T]} \mathcal{L}_{x}^{2}}
		+\|(Z_{1,\ell}- Z_{1,q+1}, Z_{2,\ell}- Z_{2,q+1})\|_{C_{(\frac{\varsigma_{q}}{2}, T]} \mathcal{L}_{x}^{2}}\big)\notag\\
		&\lesssim
		\big(\|(\Bu_{q},\Bb_{q})\|_{\mathcal{C}_{t,x}}+\|(w_{q+1},d_{q+1})\|_{L_{t}^{1}\mathcal{L}_{x}^{2}}+\|(z_{1},z_{2})\|_{C_{(\frac{\varsigma_{q}}{2}-\ell, T]}\mathcal{L}_{x}^{2}}\big)\notag\\
	&\quad\times \big(\ell^{\frac{1}{2}}(\|(z_{1}^{u},z_{2}^{B})\|_{C_{(\frac{\varsigma_{q}}{2}-\ell, T]}^{\frac{1}{2}} \mathcal{L}_{x}^{2}}
	+\|(z_{1}^{u},z_{2}^{B})\|_{C_{t} \mathcal{H}_{x}^{\frac{1}{2}}})+\ell^{\frac{1}{2}-\delta}\|(Z_{1},Z_{2})\|_{C_{t}^{\frac{1}{2}-\delta} L_{x}^{2}}+\la^{-15(1-\delta)}\|(Z_{1},Z_{2})\|_{C_{t} H_{x}^{1-\delta}}\big)\notag\\
		&\lesssim
		\big(\|(\Bu_{q},\Bb_{q})\|_{\mathcal{C}_{t,x}}+(1+\J^{17})\ell^{-65}\sigma^{-1}+\|(z_{1},z_{2})\|_{C_{(\frac{\varsigma_{q}}{2}-\ell, T]}L_{x}^{2}}\big)\notag\\
		&\quad\times \big(\ell^{\frac{1}{2}}(1+\varsigma_{q}^{-\frac{1}{2}})M+\la^{-15(1-\delta)}(\|(Z_{1},Z_{2})\|_{C_{t}^{\frac{1}{2}-\delta} L_{x}^{2}}+\|(Z_{1},Z_{2})\|_{C_{t} H_{x}^{1-\delta}})\big),
\end{align*}
which, via (\ref{2.2-2}),  (\ref{2.11}), (\ref{2.12}) and (\ref{3.2-1}), yields that
\begin{align}\label{5.85}
		&\quad\|(\Ru^{u}_{com2},\Ru^{B}_{com2})\|_{\Lo^{r}_{\om}L_{(\frac{\varsigma_{q}}{2}, T]}^{1}\mathcal{L}_{x}^{1}}\notag\\
		&\lesssim
		\big(\|(\Bu_{q},\Bb_{q})\|_{\Lo^{2r}_{\om}\mathcal{C}_{t,x}}+(1+\|\J\|_{\Lo^{34r}_{\om}}^{17})\ell^{-65}\sigma^{-1}+\|(z_{1},z_{2})\|_{\Lo^{2r}_{\om}C_{(\frac{\varsigma_{q}}{2}-\ell, T]} \mathcal{L}_{x}^{2}}\big)\notag\\
		&\quad\times
		\big(\ell^{\frac{1}{2}}(1+\varsigma_{q}^{-\frac{1}{2}})M+\la^{-15(1-\delta)}(\|(Z_{1},Z_{2})\|_{\Lo^{2r}_{\om}C_{t}^{\frac{1}{2}-\delta}\mathcal{L}_{x}^{2}}+\|(Z_{1},Z_{2})\|_{\Lo^{2r}_{\om}C_{t}\mathcal{H}_{x}^{1-\delta}})\big)\notag\\
		 &\lesssim
		\lambda_{q}^{-8}
		\lesssim
		\dqqq.
\end{align}

Therefore,  combining (\ref{5.63})-(\ref{5.85}),  we conclude that
\begin{equation}\label{5.87}
	\begin{aligned}
		\|(\Ru^{u}_{q+1},\Ru^{B}_{q+1})\|_{\Lo^{r}_{\om}L_{(\frac{\varsigma_{q}}{2}, T]}^{1}\mathcal{L}_{x}^{1}}\lesssim \dqqq.
\end{aligned}\end{equation}

\subsection{Verification of decay estimate: near the initial time}\label{near zero}

Let us continue to verify the decay estimate (\ref{2.13})
in the regime near the initial time where $t\in [0, {\varsigma_{q}}/{2}]$.

In this case,  $\Theta_{q+1}(t)=0$ and so $w_{q+1}=d_{q=1}=0$. Then, by (\ref{5.1}) and (\ref{5.8}), we obtain 
\begin{align}\label{5.53.01}
	\Ru^{u}_{q+1}&=\mathcal{R}^{u}\(z_{1,\ell}- z_{1,q+1}\)+\Ru^{u}_{\ell}+	(\Bu_{\ell}+z_{1,\ell})\mathring{\otimes}(\Bu_{\ell}+z_{1,\ell})-(\Bb_{\ell}+z_{2,\ell})\mathring{\otimes}(\Bb_{\ell}+z_{2,\ell})\notag\\
	&\quad-\left((\Bu_{q}+z_{1,q})\mathring{\otimes}(\Bu_{q}+z_{1,q})-(\Bb_{q}+z_{2,q})\mathring{\otimes}(\Bb_{q}+z_{2,q})\right)*_x \varrho_{\ell} *_t \vartheta_{\ell}\notag\\
	&\quad+\Bu_{\ell}\mathring{\otimes}z_{1,q+1}+z_{1,q+1}\mathring{\otimes} \Bu_{\ell}-z_{1,\ell}\mathring{\otimes} \Bu_{\ell}-\Bu_{\ell}\mathring{\otimes}z_{1,\ell}- \Bb_{\ell}\mathring{\otimes} z_{2,q+1}- z_{2,q+1}\mathring{\otimes} \Bb_{\ell}\notag
	\\
	&\quad+z_{2,\ell}\mathring{\otimes} \Bb_{\ell}+\Bb_{\ell}\mathring{\otimes}z_{2,\ell}+z_{1,q+1}\mathring{\otimes} z_{1,q+1}- z_{2,q+1}\mathring{\otimes} z_{2,q+1}-z_{1,\ell}\mathring{\otimes}z_{1,\ell}+z_{2,\ell}\mathring{\otimes}z_{2,\ell},
\end{align}
and
\begin{align}\label{5.53.02}
	\Ru^{B}_{q+1}&=\mathcal{R}^{B}\(z_{2,\ell}- z_{2,q+1}\)+\Ru^{B}_{\ell}
	+	(\Bb_{\ell}+z_{2,\ell})\otimes(\Bu_{\ell}+z_{1,\ell})-(\Bu_{\ell}+z_{1,\ell})\otimes(\Bb_{\ell}+z_{2,\ell})\notag\\
	&\quad-\left((\Bb_{q}+z_{2,q})\otimes(\Bu_{q}+z_{1,q})-(\Bu_{q}+z_{1,q})\otimes(\Bb_{q}+z_{2,q})\right)*_x \varrho_{\ell} *_t \vartheta_{\ell}\notag\\
	&\quad+\Bb_{\ell}\otimes z_{1,q+1}
	-z_{1,q+1}\otimes \Bb_{\ell}
	+z_{1,\ell}\otimes \Bb_{\ell}
	- \Bb_{\ell}\otimes z_{1,\ell}
	+ z_{2,q+1}\otimes \Bu_{\ell}-\Bu_{\ell}\otimes z_{2,q+1}+\Bu_{\ell}\otimes z_{2,\ell}
	\notag\\
	&\quad
	-z_{2,\ell}\otimes \Bu_{\ell}
	+z_{1,\ell}\otimes z_{2,\ell}
	-z_{2,\ell}\otimes z_{1,\ell}
	+ z_{2,q+1}\otimes z_{1,q+1}
	-z_{1,q+1}\otimes  z_{2,q+1}.
\end{align}

 Now, let us estimate the  $L_{[0,\frac{\varsigma_{q}}{2}]}^{1}L_{x}^{1}-$norm of $ \Ru^{u}_{q+1} $ and $ \Ru^{B}_{q+1} $. By (\ref{5.53.01}), (\ref{5.53.02}) and H\"{o}lder's inequality, 
\begin{align*}
&\quad\|(\Ru^{u}_{q+1},\Ru^{B}_{q+1})\|_{ L_{[0, \frac{\varsigma_{q}}{2}]}^{1}\mathcal{L}_{x}^{1}}\notag\\
&\lesssim \|(z_{1},z_{2})\|_{L_{[0, \frac{\varsigma_{q}}{2}]}^{1}\mathcal{L}_{x}^{2}}+\|(z_{1,\ell},z_{2,\ell})\|_{L_{[0, \frac{\varsigma_{q}}{2}]}^{1}\mathcal{L}_{x}^{2}}+\|(\Ru^{u}_{\ell},\Ru^{B}_{\ell})\|_{ L_{[0, \frac{\varsigma_{q}}{2}]}^{1}\mathcal{L}_{x}^{1}}\notag\\
&\quad+\big(\|(\Bu_{q}, \Bb_{q})\|_{C_{[0, \frac{\varsigma_{q}}{2}]}\mathcal{L}_{x}^{2}}+\|(z_{1}, z_{2})\|_{C_{[0, \frac{\varsigma_{q}}{2}]}\mathcal{L}_{x}^{2}}\big)
\big(\|(\Bu_{q}, \Bb_{q})\|_{L_{[0, \frac{\varsigma_{q}}{2}]}^{1}\mathcal{L}_{x}^{2}}+\|(z_{1}, z_{2})\|_{L_{[0, \frac{\varsigma_{q}}{2}]}^{1}\mathcal{L}_{x}^{2}}\big)\notag\\
&\quad+
\big(\|(\Bu_{q}, \Bb_{q})\|_{L_{[0, \frac{\varsigma_{q}}{2}]}^{1}\mathcal{L}_{x}^{2}}+\|(z_{1}, z_{2})\|_{L_{[0, \frac{\varsigma_{q}}{2}]}^{1}\mathcal{L}_{x}^{2}}\big)\|(z_{1}, z_{2})\|_{C_{[0, \frac{\varsigma_{q}}{2}]}\mathcal{L}_{x}^{2}}\notag\\
	&\lesssim
	\varsigma_{q}\big(\|(z_{1},z_{2})\|_{C_{[0, \frac{\varsigma_{q}}{2}]}\mathcal{L}_{x}^{2}}+\|(\Ru^{u}_{q},\Ru^{u}_{q})\|_{ \mathcal{C}_{[0, \frac{\varsigma_{q}}{2}]}\mathcal{L}_{x}^{1}}
	+\|(\Bu_{q}, \Bb_{q})\|_{\mathcal{C}_{[0, \frac{\varsigma_{q}}{2}],x}}^{2}+\|(z_{1},z_{2})\|_{C_{[0, \frac{\varsigma_{q}}{2}]}\mathcal{L}_{x}^{2}}^{2}\big).
\end{align*}
Thus, by (\ref{2.2-2}),  (\ref{2.11}), (\ref{2.12}) and (\ref{3.2-1}), we have
\begin{align*}
	&\quad\|(\Ru^{u}_{q+1},\Ru^{B}_{q+1})\|_{\Lo^{r}_{\om} L_{[0, \frac{\varsigma_{q}}{2}]}^{1}\mathcal{L}_{x}^{1}}\notag\\
	&\lesssim
	\varsigma_{q}\big(\|(z_{1}, z_{2})\|_{\Lo^{r}_{\om}C_{[0, \frac{\varsigma_{q}}{2}]}\mathcal{L}_{x}^{2}}+\|(\Ru^{u}_{q},\Ru^{B}_{q})\|_{ \Lo^{r}_{\om}C_{[0, \frac{\varsigma_{q}}{2}]}\mathcal{L}_{x}^{1}}
	+\|(\Bu_{q}, \Bb_{q})\|_{\Lo^{2r}_{\om}\mathcal{C}_{[0, \frac{\varsigma_{q}}{2}],x}}^{2}+\|(z_{1}, z_{2})\|_{\Lo^{2r}_{\om}C_{[0, \frac{\varsigma_{q}}{2}]}\mathcal{L}_{x}^{2}}^{2}\big)\notag\\
	&\lesssim
	\la^{-30}\big(M+(r-1)^{\frac{1}{2}}L+\lambda_{q}^{(4\a+5)}(8rL^{2}80^{q})^{80^{q}}+\la^{4\a+4}(8\cdot44rL^{2}80^{q-1})^{44\cdot80^{q-1}}+M^{2}+(2r-1)L^{2}\big)\notag\\
	&\lesssim
	\lambda_{q}^{-18}
	\lesssim
	\dqqq,
\end{align*}
which verifies  (\ref{2.13}) at level $q+1$. Therefore, the proof of Proposition \ref{Proposition Main iteration} is completed.\hfill$\square$

\section{Proof of main results}   \label{Sec-Proof-Main}

This section is devoted to the proof of the main results in
Theorems \ref{Thm-Nonuniq-Hyper}, 
\ref{uniq-sharp}  and \ref{Theorem Vanishing noise}.
Let us start with the main non-uniqueness result in Theorem \ref{Thm-Nonuniq-Hyper}.

\medskip
\paragraph{\bf Proof of Theorem \ref{Thm-Nonuniq-Hyper}.}
At the initial step $q=0$, we take 
\begin{align*}
	&\Bu_{0}=\ui,\quad \mathring{R}_{0}^{u}=\mathcal{R}^{u}\(\p_{t}\ui+\nu(-\Delta)^{\a}\ui-z_{1,0}\)+(\ui+z_{1,0})
\mathring{\otimes}(\ui+z_{1,0})-(\bi+z_{2,0})\mathring{\otimes}(\bi+z_{2,0}),     \\	
     &\Bb_{0}=\bi,\quad \mathring{R}_{0}^{B}=\mathcal{R}^{B}\(\p_{t}\bi+\nu(-\Delta)^{\a}\bi-z_{2,0}\)
     +(\bi+z_{2,0})\otimes(\ui+z_{1,0})-(\ui+z_{1,0})\otimes(\bi+z_{2,0}),  \\
     &P_{0}=-\frac{1}{3}(|\ui_{0}+z_{1,0}|^{2}-|\bi_{0}+z_{2,0}|^{2}), 
\end{align*}
where $v,H\in C_0^\infty$ are the 
given divergence-free and mean-free 
fields. 

Note that 
(\ref{2.11}) is satisfied at level $q=0$ 
for sufficiently large $a$.   
Next, let us verify the decay estimate  (\ref{2.13})  at level $q=0$.  
Actually, by (\ref{2.1}), we have
\begin{align*}
	&\quad\|(\mathring{R}_{0}^{u},\mathring{R}_{0}^{B})\|_{\Lo^{r} _{\om}L_{t}^{1}\mathcal{L}_{x}^{1}}\notag\\
	&\lesssim
	\|(\p_{t}\ui,\p_{t}\bi)\|_{C_{t}\mathcal{L}_{x}^{2}}+\|(\ui,\bi)\|_{\mathcal{C}_{t,x}^{2}}+\|(z_{1,0},z_{2,0})\|_{\Lo^{r}_{\om} C_{t}\mathcal{L}_{x}^{2}} +\|(z_{1,0},z_{2,0})\|^{2}_{\Lo^{2r}_{\om} C_{t}\mathcal{L}_{x}^{2}}+ \|(\ui,\bi)\|^{2} _{C_{t}\mathcal{L}_{x}^{2}}\notag\\
	&\lesssim
	L_{*}+M+(r-1)^{\frac{1}{2}}L+(M+(2r-1)^{\frac{1}{2}}L)^{2}
	\lesssim \delta_{1},
\end{align*}
where  $L_{*}:=1+ \|(\ui, \bi)\|_{\mathcal{C}_{t,x}^{8}}^{2}$ is a deterministic constant, 
and we choose $a$ sufficiently large so that $\delta_{1}=\lambda_{1}^{\beta}$ 
is sufficiently large in the last step. Thus,  (\ref{2.13}) is satisfied at level $q=0$.

Similarly,  we have $\|(\mathring{R}_{0}^{u}, \mathring{R}_{0}^{B})\|_{\Lo^{m}_{\om} C_{t}\mathcal{L}_{x}^{1}}\leq\lambda_{0}^{4\a+5}8mL^{2} $ by choosing $a$ possibly larger, 
and thus (\ref{2.12}) is verified at level $q=0$. 

Then, by virtue of Proposition  \ref{Proposition Main iteration}, 
we obtain a sequence of  relaxed solutions $(\Bu_{q}, \Bb_{q}, \ru^{u}, \ru^{B})$ to \eqref{2.5}, which satitsfies \eqref{2.11}-\eqref{2.16} for all $q\geq 0$.

Now we show that the relaxed solutions 
are $(\F_t)$-adapted. 
Actually, at the initial step, 
Since $\ui, \bi$ are deterministic, and $z_{1,0}$, $z_{2,0}$ are $(\F_t)$-adapted process,  
so are $(\Bu_{0}, \Bb_{0}, \Ru_{0}^{u}, \Ru_{0}^{B})$. 
Because $\vartheta_{\ell}$ is a one-sided temporal mollifier, the mollified Reynolds stress $(\Ru^{u}_{\ell_{0}},\Ru^{B}_{\ell_{0}})$ is 
also adapted, 
and so is the amplitude $a_{(k)}$ due to \eqref{4.24} and \eqref{4.45}. 
Then, by \eqref{4.69} and \eqref{4.71}, 
both the velocity and magnetic perturbations $w_{1}$ and $d_{1}$ are adapted. 
In veiw of \eqref{4.72}, \eqref{5.1} and \eqref{5.8},  $(\Bu_{1}, \Bb_{1}, \Ru_{1}^{u}, \Ru_{1}^{B})$ is 
adapted as well. 
Hence, arguing iteratively we infer that 
 $(\Bu_{q}, \Bb_{q}, \Ru_{q}^{u}, \Ru_{q}^{B})$ 
 are adapted for all $q \geq 0$.

In view of (\ref{2.14}), (\ref{2.15}) and (\ref{2.16}),  
we infer that 
there exist $ \Bu, \Bb$  such that
\begin{align}
	 (\Bu_{q}, \Bb_{q})\rightarrow (\Bu, \Bb)  \quad \text{in}\quad \Lo^{2r}_{\om} L_{t}^{2}\mathcal{L}_{x}^{2}\cap\Lo^{r}_{\om} L_{t}^{1}\mathcal{L}_{x}^{2}\cap\Lo^{r}_{\om} L_{t}^{\ga}\mathcal{W}_{x}^{s,p},\label{s1} 
\end{align}
 and
 \begin{align}\label{s2}
 (\ru^{u},  \ru^{B})\rightarrow (0, 0)   \quad\text{in} \quad \Lo^{r}_{\om}L_{t}^{1}\mathcal{L}_{x}^{1}.
\end{align}
In particular, 
there exists a subsequence $\{q_k\}$ 
such that $\textbf{P}$-a.s., 
\begin{align}
	&(\Bu_{q_{k}}, \Bb_{q_{k}})\rightarrow (\Bu, \Bb)  \quad \text{in}\quad  L_{t}^{2}\mathcal{L}_{x}^{2}\cap  L_{t}^{1}\mathcal{L}_{x}^{2}\cap  L_{t}^{\ga}\mathcal{W}_{x}^{s,p},\label{s11}\\ 
	&(\Ru^{u}_{q_{k}},  \Ru^{B}_{q_{k}})\rightarrow (0, 0)   \quad\text{in} \quad L_{t}^{1}\mathcal{L}_{x}^{1},\label{s22}
\end{align}
and
\begin{align}
    (z_{1,q_{k}}, z_{2,q_{k}})\rightarrow (z_{1}, z_{2})  \quad \text{in}\quad  L_{t}^{2}\mathcal{L}_{x}^{2}. \label{s1.1} 
\end{align} 
The above strong convergences suffice to verify that 
the limit $(\wt u, \wt B)$ satisfies  equation (\ref{1.3}) 
in the distributional sense, 
which in turn yields that 
$ (u,B):=(\Bu+z_{1},\Bb+z_{2}) $ is a solution to (\ref{1.1}) with  the initial data $(u_{0},B_{0})$. 
We also note that, by  (\ref{2.15}) and (\ref{2.16}),
	\begin{align}\label{2.19}
		&\quad\|(u-(\ui+z_{1}),B-(\bi+z_{2}))\|_{\Lo^{r}_{\om} L_{t}^{1}\mathcal{L}_{x}^{2}}+\|(u-(\ui+z_{1}),B-(\bi+z_{2}))\|_{\Lo^{r}_{\om}L_{t}^{\ga}\mathcal{W}_{x}^{s,p}}\notag\\
		&\leq\sum_{q=0}^{\infty}\|(\Bu_{q+1}-\Bu_{q},\Bb_{q+1}-\Bb_{q})\|_{\Lo^{r}_{\om} L_{t}^{1}\mathcal{L}_{x}^{2}}+\|(\Bu_{q+1}-\Bu_{q},\Bb_{q+1}-\Bb_{q})\|_{\Lo^{r}_{\om}L_{t}^{\ga}\mathcal{W}_{x}^{s,p}}\notag\\
		&\leq\sum_{q=0}^{\infty}4\dqqq^{1/2}<\epsilon, 
	\end{align} 
and so \eqref{1.5} follows. 

Finally, let us show the non-uniqueness of solutions which follows from the arbitrariness of smooth vector 
fields $(v,H)$. 
Actually, let us choose any divergence-free and mean-free sequences  $ \{(\ui_{i},\bi_{i})\} \subseteq C_{0}^{\infty} $ with $\ui_{i}(0)=\bi_{i}(0)=0$ such that
\begin{equation}\label{2.20}
	\begin{aligned}
	 \|(\ui_{i}-\ui_{j},\bi_{i}-\bi_{j})\|_{L_{t}^{\gamma}\mathcal{W}_{x}^{s,p}}\geq 1,\ \  i\neq j.
	\end{aligned}
\end{equation} 
Then, the above arguments give a sequence of solutions $ \{(u_{i},B_{i})\} $  to (\ref{1.1})  
such that 
	\begin{align*}
		\|(u_{i}-(\ui_{i}+z_{1}),B_{i}-(\bi_{i}+z_{2}))\|_{\Lo^{r}_{\om}L_{t}^{\ga}\mathcal{W}_{x}^{s,p}}<1/3.
	\end{align*}
It follows that 
\begin{align*}
		\|(u_{i}-u_{j}, B_{i}-B_{j})\|_{\Lo^{r}_{\om}L_{t}^{\ga}\mathcal{W}_{x}^{s,p}}
		&\geq \|(\ui_{i}-\ui_{j},\bi_{i}-\bi_{j})\|_{L_{t}^{\gamma}\mathcal{W}_{x}^{s,p}}-\|(u_{i}-(\ui_{i}+z_{1}),B_{i}-(\bi_{i}+z_{2}))\|_{\Lo^{r}_{\om}L_{t}^{\ga}\mathcal{W}_{x}^{s,p}}\notag\\
		&\quad-\|(u_{j}-(\ui_{j}+z_{1}),B_{j}-(\bi_{j}+z_{2}))\|_{\Lo^{r}_{\om}L_{t}^{\ga}\mathcal{W}_{x}^{s,p}}\notag\\
		&\geq 1-2/3>0, 
\end{align*} 
which implies the non-uniqueness of solutions to (\ref{1.1}).  
Therefore,  the proof is completed.
\hfill$\square$

\medskip  
Next, we aim to prove 
the uniqueness result in the (sub)critical LPS regime  
in Theorem \ref{uniq-sharp}. \\

Let us first consider the linearized stochastic  MHD equations:
\begin{eqnarray}\label{8.3}
	\left\{\begin{array}{lc}
		\mathrm{d} v+\big(\nu(-\Delta)^{\a}v+(u\cdot \nabla)v-(B\cdot \nabla )H\big) \mathrm{d} t+\nabla P \mathrm{d} t =\mathrm{d} W^{(1)}(t),\\
		\mathrm{d}H+\big(\nu(-\Delta)^{\a}H+(u\cdot \nabla )H-(B\cdot \nabla )v\big)\mathrm{d} t =\mathrm{d} W^{(2)}(t),\\
		\div v=0, \quad\div H=0,\\
		v(0)=v_{0},\quad H(0)=H_{0}.
	\end{array}\right.
\end{eqnarray}

\begin{lemma}\label{existence}
	Let $(u, B)$ be a probabilistically strong
	and analytically weak solution to (\ref{1.1}) with $\nu_{1}=\nu_{2}=\nu$, 
 $(u, B)\in X_{(0,T)}^{s,\gamma,p}\times X_{(0,T)}^{s,\gamma,p}$, $\mathbf{P}$-a.s., 
 and  $(s,\gamma,p)$ satisfy  (\ref{LPS-2}) with $\a\geq 1$, $2\leq \gamma\leq\infty$, $1\leq p\leq \infty$, $s\geq 0$ and $0\leq \frac{1}{p}-\frac{s}{3}\leq \frac{1}{2}$.  Then, for any $(v_{0}, H_{0})\in \mathcal{L}_{\sigma}^{2}$, there exists a unique probabilistically strong
	and analytically weak solution  $(v, H)$ to  system (\ref{8.3}), 
 and  $(v, H) \in  C_{w}([0, T];\mathcal{L}_{x}^{2})\cap   L^{2}(0, T;\mathcal{H}_{x}^{\a})$, $\mathbf{P}$-a.s.
\end{lemma}

In order to prove Lemma \ref{existence}, we first set 
\begin{equation}\label{8.3.0}
	v:=\widetilde{v}+z_{1}, \quad H:=\widetilde{H}+z_{2},
\end{equation}
where $z_{i}(t)=\int_{0}^{t}e^{(t-s)\(-\nu(-\Delta)^{\a}\)}\mathrm{d}W^{(i)}(s)$, $i=1,2$. 
As in Proposition \ref{Proposition stochastic}, 
it holds that  
 for any  $\kappa\in [0, \a)$, $p\geq 2$,  
	\begin{align}\label{8-3-1}
		\mathbb{E}\|z_{i}\|^{p}_{C_{t}H_{x}^{\kappa+4}}
		\leq (p-1)^{p/2}L^{p},
	\end{align}
	where $ L (\geq 1) $ depends on $ Tr(G_{i}G_{i}^{*}) $,  $T$, $\kappa$, $\a$, $\nu$ and is independent of $ p $.

Then, using (\ref{8.3.0}) we reformulate the linearlized problem (\ref{8.3}) as follows
\begin{eqnarray}\label{8.3.2}
	\left\{\begin{array}{lc}
		\p_{t}\widetilde{v}+\nu(-\Delta)^{\a}\widetilde{v}+(u\cdot \nabla)(\widetilde{v}+z_{1})-(B\cdot \nabla )(\widetilde{H}+z_{2})+\nabla P  =0 ,\\
		\p_{t}\widetilde{H}+\nu(-\Delta)^{\a}\widetilde{H}+(u\cdot \nabla )(\widetilde{H}+z_{2})-(B\cdot \nabla )(\widetilde{v}+z_{1}) =0,\\
		\div\widetilde{v}=0, \quad\div \widetilde{H}=0,\\
		\widetilde{v}(0)=v_{0},\quad \widetilde{H}(0)=H_{0}.
	\end{array}\right.
\end{eqnarray}
Hence, it suffices to prove the existence and uniqueness of solutions to problem (\ref{8.3.2}).

  Take a smooth orthonormal basis $\{ a_{k}\}_{k\geq 1}$ in $L^{2}_\sigma$ with  $a_{k}$ being the  eigenvectors of Stokes operator $-\P\Delta$ with the corresponding eigenvalues $\lambda_{k}$, $k\geq 1$, where $\P$ is the Leray projection (see \citep[Charpter 2]{Robinson}). Let $\mathbb{H}_{n}:=\text{span} \{a_{1}, a_{2}, \cdots, a_{n}\}$ and  define $\mathcal{P}_{n}: L_{x}^{2}\rightarrow \mathbb{H}_{n}$,
\begin{equation}\label{8.2-1}
\mathcal{P}_{n}u:=\sum_{i=1}^{n}u_{i} a_{i},\quad\text{where }\quad  u_{i}=\langle u, a_{i}\rangle.
\end{equation}
Note that,  $\mathcal{P}_{n}$ is also well-defined as an operator from $H_{x}^{-s}$  to $\mathbb{H}_{n}$. 
One has that (\citep[Charpter 4]{Robinson})
\begin{enumerate}[(i)]
	\item For any $u$, $v\in L_{x}^{2}$,
	\begin{align}\label{8.3-0}
		\langle\mathcal{P}_{n}u, v\rangle
		=\langle u, \mathcal{P}_{n}v\rangle.
		\end{align}
	\item For any  $u\in H_{x}^{s}$, $s> 0$,
	\begin{align}\label{8.3-1}
	\mathcal{P}_{n}u\rightarrow u \quad\text{in} \quad  H_{x}^{s},
	\end{align}
and 	\begin{align}\label{8.3-2}
	\|\mathcal{P}_{n}u\|_{H_{x}^{s}}\leq C\|u\|_{H_{x}^{s}}.		
\end{align}
\end{enumerate}

\paragraph{\bf Proof of Lemma \ref{existence}} We fix $\omega\in\Omega$ in the sequel.
Consider the Galerkin approximation of \eqref{8.3.2} 
\begin{eqnarray}\label{8.2-9}
	\left\{\begin{array}{lc}
		\p_{t}\widetilde{v}_{n}+\nu(-\Delta)^{\a}\widetilde{v}_{n}+\mathcal{P}_{n}\big((u_{n}\cdot \nabla)(\widetilde{v}_{n}+z_{1})-(B_{n}\cdot \nabla )(\widetilde{H}_{n}+z_{2})\big)=0,\\
		\p_{t}\widetilde{H}_{n}+\nu( -\Delta)^{\a}\widetilde{H}_{n}+\mathcal{P}_{n}\big((u_{n}\cdot \nabla )(\widetilde{H}_{n}+z_{2})-(B_{n}\cdot \nabla )(\widetilde{v}_{n}+z_{2})\big) =0,\\
		\div \widetilde{v}_{n}=0, \quad\div \widetilde{H}_{n}=0,\\
		\widetilde{v}_{n}(0)=\mathcal{P}_{n}\widetilde{v}_{0},\quad \widetilde{H}_{n}(0)=\mathcal{P}_{n}\widetilde{H}_{0},
	\end{array}\right.
\end{eqnarray}
where 
	\begin{align*}
	&\widetilde{v}_{n}(t,x)=\sum_{i=1}^{n}\widetilde{v}_{n,i}(t)a_{i}(x),\quad
	\widetilde{H}_{n}(t,x)=\sum_{i=1}^{n}\widetilde{H}_{n,i}(t)a_{i}(x),
\end{align*}
and  $(u_{n},B_{n}):=((u*_{x} \varrho_{n^{-1}}) *_{t} \vartheta_{n^{-1}},(B*_{x} \varrho_{n^{-1}}) *_{t} \vartheta_{n^{-1}})$, with $\varrho_{n^{-1}}$ and  $\vartheta_{n^{-1}}$ as in \eqref{3.1}.
Since $\vartheta_{n^{-1}}$ is a one-sided temporal mollifier,  $(u_{n},B_{n})$ is $(\F_{t})$-adapted 
and is continuous in time. 
It is not difficult to see that  system  \eqref{8.2-9} has a unique $(\F_{t})$-adapted solution on $[0,T]$. 

Next, taking the inner product of \eqref{8.2-9}, and   
using the algebraic cancellations 
\begin{align*}
	&\langle\mathcal{P}_{n}\big((u_{n}\cdot \nabla)\widetilde{v}_{n}\big),\widetilde{v}_{n}\rangle=0,\quad \langle\mathcal{P}_{n}\big((u_{n}\cdot \nabla)\widetilde{H}_{n}\big),\widetilde{H}_{n}\rangle=0,\\
	&\langle\mathcal{P}_{n}\big((B_{n}\cdot \nabla)\widetilde{H}_{n}\big),\widetilde{v}_{n}\rangle+\langle\mathcal{P}_{n}\big((B_{n}\cdot \nabla)\widetilde{v}_{n}\big),\widetilde{H}_{n}\rangle=0,
\end{align*}
we derive 
\begin{align}\label{8-24}
	&\quad\frac{1}{2}\|\widetilde{v}_{n}(t)\|_{L_{x}^{2}}^{2}+\frac{1}{2}\|\widetilde{H}_{n}(t)\|_{L_{x}^{2}}^{2}+\nu\int_{0}^{t}\|\widetilde{v}_{n}(s)\|_{\dot{H}_{x}^{\a}}^{2}+\|\widetilde{H}_{n}(s)\|_{\dot{H}_{x}^{\a}}^{2}\mathrm{d}s\notag\\
	&\leq\frac{1}{2}\|\widetilde{v}(0)\|_{L_{x}^{2}}^{2}+\frac{1}{2}\|\widetilde{H}(0)\|_{L_{x}^{2}}^{2}+|\int_{0}^{t}\langle(u_{n}(s)\cdot \nabla)z_{1}(s)-(B_{n}(s)\cdot \nabla )z_{2}(s), \widetilde{v}_{n}(s)\rangle\mathrm{d}s|\notag\\
	&\quad+|\int_{0}^{t}\langle(u_{n}(s)\cdot \nabla)z_{2}(s)-(B_{n}(s)\cdot \nabla )z_{1}(s), \widetilde{H}_{n}(s)\rangle\mathrm{d}s|.
\end{align}
For $p<+\infty$,  since  $ 0\leq \frac{1}{p}-\frac{s}{3}\leq \frac{1}{2}$ and  $2\leq \gamma\leq\infty$, using H\"{o}lder's inequality, Young's inequality and the  embeddings (see \citep{Benyi}) 
\begin{align}  \label{embed} 
   W_{x}^{s,p}\hookrightarrow L_{x}^{2},\ \ H_{x}^{2}\hookrightarrow L_{x}^{\infty},
\end{align}
we obtain
\begin{align*}
|\int_{0}^{t}\langle(u_{n}(s)\cdot \nabla)z_{1}(s), \widetilde{v}_{n}(s)\rangle\mathrm{d}s|
&\lesssim\|\widetilde{v}_{n}\|_{L_{[0,t]}^{2}L_{x}^{2}}\|u_{n}\|_{L_{[0,t]}^{2}L_{x}^{2}}\|\nabla z_{1}\|_{C_{[0,t]}L_{x}^{\infty}}\notag\\
&\lesssim \|\widetilde{v}_{n}\|_{L_{[0,t]}^{2}L_{x}^{2}}\|u\|_{L_{[0,t]}^{2}W_{x}^{s,p}}\|z_{1}\|_{C_{[0,t]}H_{x}^{3}}\notag\\
&\lesssim \|\widetilde{v}_{n}\|_{L_{[0,t]}^{2}L_{x}^{2}}^{2}+\|u\|_{L_{[0,t]}^{\gamma}W_{x}^{s,p}}^{2}\|z_{1}\|_{C_{[0,t]}H_{x}^{3}}^{2},
\end{align*} 
where the implicit constant is deterministic and independent of $n$. 
For $p=+\infty$,   we have  $u\in L_{t}^{2}L_{x}^{\infty}$, 
and so, by H\"{o}lder's inequality and  Young's inequality, 
\begin{align*}
	|\int_{0}^{t}\langle(u_{n}(s)\cdot \nabla)z_{1}(s), \widetilde{v}_{n}(s)\rangle\mathrm{d}s|
	&\lesssim \|\widetilde{v}_{n}\|_{L_{[0,t]}^{2}L_{x}^{2}}^{2}+\|u\|_{L_{[0,t]}^{2}L_{x}^{\infty}}^{2}\|z_{1}\|_{C_{[0,t]}H_{x}^{1}}^{2}.
\end{align*}

Similar estimates also hold for the remaining terms involving $z_{i}$, $i=1,2$. 
Hence, we obtain 
\begin{align}\label{8-24-2}
	&\quad\|(\widetilde{v}_{n}(t),\widetilde{H}_{n}(t))\|_{\mathcal{L}_{x}^{2}}^{2}+\nu\int_{0}^{t}\|(\widetilde{v}_{n}(s),\widetilde{H}_{n}(s))\|_{\dot{\mathcal{H}}_{x}^{\a}}^{2}\mathrm{d}s\notag\\
	&\lesssim\|(\widetilde{v}(0),\widetilde{H}(0))\|_{\mathcal{L}_{x}^{2}}^{2}+\|(\widetilde{v}_{n},\widetilde{H}_{n})\|_{L_{[0,t]}^{2}\mathcal{L}_{x}^{2}}^{2}+\|(u,B)\|_{L_{[0,t]}^{\gamma}\mathcal{W}_{x}^{s,p}}^{2}\|(z_{1},z_{2})\|_{C_{[0,t]}\mathcal{H}_{x}^{3}}^{2}. 
\end{align}
Then, applying Gronwall's inequality, we get 
the uniform energy bound 
\begin{align}\label{8.2-12.0}
	&\quad\sup_{t\in[0.T]}\|(\widetilde{v}_{n}(t),\widetilde{H}_{n}(t))\|_{\mathcal{L}_{x}^{2}}^{2} 
  +\int_{0}^{T}\|(\widetilde{v}_{n}(s), \widetilde{H}_{n}(s))\|_{\dot{\mathcal{H}}_{x}^{\a}}^{2}\mathrm{d}s\notag\\
	&\lesssim \|(\widetilde{v}(0), \widetilde{H}(0))\|_{\mathcal{L}_{x}^{2}}^{2}
	+\|(u,B)\|_{L_{t}^{\gamma} \mathcal{W}_{x}^{s,p}}^{2}\|(z_{1},z_{2})\|_{C_{t}\mathcal{H}_{x}^{3}}^{2}, 
\end{align}
where the implicit constant is deterministic and independent of $n$.

Thus, there exists a subsequence (still denoted by $\{n\}$) which may depend on $\omega$, 
such that
\begin{align}
	&(\widetilde{v}_{n},\widetilde{H}_{n})\rightharpoonup  (\widetilde{v},\widetilde{H}) \quad\text{weakly  in} \quad L^{2}_{t}\mathcal{L}^{2}_{x}\cap L^{2}_{t}\mathcal{H}^{\a}_{x} ,\label{8.2-13}\\
	&(\widetilde{v}_{n}, \widetilde{H}_{n}) 
 \rightharpoonup (\widetilde{v},\widetilde{H}) \quad\text{weakly-$\ast$ in} \quad L^{\infty}_{t}\mathcal{L}^{2}_{x}.\label{8.2-14}
\end{align}
In particular, the following energy bound holds 
\begin{align}\label{8-20-2}
	&\quad\sup_{t\in[0.T]}\|(\widetilde{v}(t),\widetilde{H}(t))\|_{\mathcal{L}_{x}^{2}}^{2} 
 +\int_{0}^{T}\|(\widetilde{v}(s), \widetilde{H}(s))\|_{\dot{\mathcal{H}}_{x}^{\a}}^{2}\mathrm{d}s\notag\\
	&\lesssim \|(\widetilde{v}(0), \widetilde{H}(0))\|_{\mathcal{L}_{x}^{2}}^{2}
	+\|(u,B)\|_{L_{t}^{\gamma} \mathcal{W}_{x}^{s,p}}^{2}\|(z_{1},z_{2})\|_{C_{t}\mathcal{H}_{x}^{3}}^{2}.
\end{align}

Below we verify that the limit $(\widetilde{v},\widetilde{H})$ satisfies system \eqref{8.3.2}. 
To this end, 
take any  divergence-free test function $\Phi(t,x)=\psi(t)\phi(x)$, where $\phi\in C_{0}^{\infty}(\T^{3})$ is divergence-free and $\psi\in C_{0}^{\infty}(0,T)$. 
Using \eqref{8.2-9}
and  Fubini's Theorem, we have 
\begin{align}\label{8-28}
	\int_{0}^{T}\langle \widetilde{v}_{n}(t)-\mathcal{P}_{n}\widetilde{v}_{0}, \Phi(t)\rangle\mathrm{d}t	&=-\int_{0}^{T}	\langle \nu(-\Delta)^{\a}\widetilde{v}_{n}(s), \int_{s}^{T}\Phi(t)\mathrm{d}t\rangle\mathrm{d}s\notag\\
	&\quad-\int_{0}^{T}\langle(u_{n}(s)\cdot \nabla)(\widetilde{v}_{n}(s)+z_{1}(s))-(B_{n}(s)\cdot \nabla )(\widetilde{H}_{n}(s)+z_{2}(s)), \int_{s}^{T}\mathcal{P}_{n}\Phi(t)\mathrm{d}t\rangle\mathrm{d}s.
\end{align}
By  \eqref{8.2-13}, 
\begin{align*}
	&\quad|\int_{0}^{T}\langle \nu(-\Delta)^{\a}\widetilde{v}_{n}(s), \int_{s}^{T}\Phi(t)\mathrm{d}t\rangle\mathrm{d}s-\int_{0}^{T}	\langle \nu(-\Delta)^{\a}\widetilde{v}(s), \int_{s}^{T}\Phi(t)\mathrm{d}t\rangle\mathrm{d}s|\notag\\
	&=	|\int_{0}^{T}\langle \nu(-\Delta)^{\frac{\a}{2}}(\widetilde{v}_{n}(s)-\widetilde{v}(s)), (-\Delta)^{\frac{\a}{2}}\int_{s}^{T}\Phi(t)\mathrm{d}t\rangle\mathrm{d}s|\rightarrow 0.
\end{align*}
Regarding the nonlinearities on the right-hand side of \eqref{8-28}, since  $\div\, u_{n}=0$, we have
\begin{align*}
	&\quad|\int_{0}^{T}\langle(u_{n}(s)\cdot \nabla)\widetilde{v}_{n}(s), \int_{s}^{T}\mathcal{P}_{n}\Phi(t)\mathrm{d}t\rangle\mathrm{d}s-\int_{0}^{T}\langle(u(s)\cdot \nabla)\widetilde{v}(s), \int_{s}^{T}\Phi(t)\mathrm{d}t\rangle\mathrm{d}s|\notag\\
	&\leq |\int_{0}^{T}\langle\widetilde{v}_{n}(s)-\widetilde{v}(s), (u_{n}(s)\cdot \nabla)\int_{s}^{T}\mathcal{P}_{n}\Phi(t)\mathrm{d}t\rangle\mathrm{d}s|\notag\\
	&\quad+|\int_{0}^{T}\langle(u_{n}(s)\cdot \nabla)\widetilde{v}(s), \int_{s}^{T}\mathcal{P}_{n}\Phi(t)\mathrm{d}t\rangle-\langle(u(s)\cdot \nabla)\widetilde{v}(s),\int_{s}^{T}\Phi(t)\mathrm{d}t\rangle\mathrm{d}s|=:I_{1}+I_{2}.
\end{align*} 
We first consider the case where $p<+\infty$, by the embeddings in \eqref{embed}, we have
\begin{align}\label{8-31}
\|(u_{n}\cdot \nabla)\int_{t}^{T}\mathcal{P}_{n}\Phi(s)\mathrm{d}s\|_{L_{t}^{2}L_{x}^{2}}
\lesssim \|u_{n}\|_{L_{t}^{2}L_{x}^{2}}\|\nabla\mathcal{P}_{n}\Phi\|_{C_{t}L_{x}^{\infty}}
\lesssim \|u\|_{L_{t}^{\gamma}W_{x}^{s,p}}\|\Phi\|_{C_{t,x}^{3}}< \infty,
\end{align}
and 
\begin{align}\label{8-31-1}
&\quad\|(u_{n}\cdot \nabla)\int_{t}^{T}\mathcal{P}_{n}\Phi(s)\mathrm{d}s-(u\cdot \nabla)\int_{t}^{T}\Phi(s)\mathrm{d}s\|_{L_{t}^{2}L_{x}^{2}}\notag\\
&\lesssim \|u_{n}-u\|_{L_{t}^{2}L_{x}^{2}}\|\nabla\mathcal{P}_{n}\Phi\|_{C_{t}L_{x}^{\infty}}+\|u\|_{L_{t}^{2}L_{x}^{2}}\|\nabla(\mathcal{P}_{n}\Phi-\Phi)\|_{C_{t}L_{x}^{\infty}}\notag\\
&\lesssim \|u_{n}-u\|_{L_{t}^{2}L_{x}^{2}}\|\mathcal{P}_{n}\Phi\|_{C_{t}H_{x}^{3}}+\|u\|_{L_{t}^{\gamma}W_{x}^{s,p}}\|\mathcal{P}_{n}\Phi-\Phi\|_{C_{t}H_{x}^{3}}
\rightarrow 0. 
\end{align} 
It follows from \eqref{8.2-13}, \eqref{8-31} and \eqref{8-31-1} that for $p<+\infty$
\begin{align}\label{8-32}
 I_{1}=|\int_{0}^{T}\langle\widetilde{v}_{n}(s)-\widetilde{v}(s), (u_{n}(s)\cdot \nabla)\int_{s}^{T}\mathcal{P}_{n}\Phi(t)\mathrm{d}t\rangle\mathrm{d}s|\rightarrow 0.
\end{align}
Moreover, by  H\"{o}lder's inequality 
and \eqref{embed}, 
\begin{align}\label{8-33.0}
 I_{2}&\lesssim
 \|\widetilde{v}\|_{L_{t}^{2}\dot{H}_{x}^{1}}(\|u_{n}\|_{L_{t}^{2}L_{x}^{2}}\|\mathcal{P}_{n}\Phi-\Phi\|_{C_{t}L_{x}^{\infty}}+\|u_{n}-u\|_{L_{t}^{2}L_{x}^{2}}\| \Phi\|_{C_{t,x}})\notag\\
	&\lesssim	\|\widetilde{v}\|_{L_{t}^{2}L_{x}^{2}}^{1-\frac{1}{\a}}\|\widetilde{v}\|_{L_{t}^{2}\dot{H}_{x}^{\a}}^{\frac{1}{\a}}(\|u\|_{L_{t}^{\gamma}W_{x}^{s,p}}\|\mathcal{P}_{n}\Phi-\Phi\|_{C_{t}H_{x}^{2}}+\|u_{n}-u\|_{L_{t}^{2}L_{x}^{2}}\|\Phi\|_{C_{t,x}})\rightarrow 0.
\end{align}
Considering the case where   $p=+\infty$, we get \begin{align}\label{8-31.2}
	\|(u_{n}\cdot \nabla)\int_{t}^{T}\mathcal{P}_{n}\Phi(s)\mathrm{d}s\|_{L_{t}^{2}L_{x}^{2}}
	\lesssim \|u_{n}\|_{L_{t}^{2}L_{x}^{\infty}}\|\nabla\mathcal{P}_{n}\Phi\|_{C_{t}L_{x}^{2}}
	\lesssim \|u\|_{L_{t}^{2}L_{x}^{\infty}}\|\Phi\|_{C_{t,x}^{1}}< \infty,
\end{align}
and 
\begin{align}\label{8-31-2}
	&\quad\|(u_{n}\cdot \nabla)\int_{t}^{T}\mathcal{P}_{n}\Phi(s)\mathrm{d}s-(u\cdot \nabla)\int_{t}^{T}\Phi(s)\mathrm{d}s\|_{L_{t}^{2}L_{x}^{2}}\notag\\
	&\lesssim \|u_{n}-u\|_{L_{t}^{2}L_{x}^{2}}\|\nabla\mathcal{P}_{n}\Phi\|_{C_{t}L_{x}^{\infty}}+\|u\|_{L_{t}^{2}L_{x}^{\infty}}\|\nabla(\mathcal{P}_{n}\Phi-\Phi)\|_{C_{t}L_{x}^{2}}\notag\\
	&\lesssim \|u_{n}-u\|_{L_{t}^{2}L_{x}^{2}}\|\Phi\|_{C_{t}H_{x}^{3}}+\|u\|_{L_{t}^{2}L_{x}^{\infty}}\|\mathcal{P}_{n}\Phi-\Phi\|_{C_{t}H_{x}^{1}}
	\rightarrow 0, 
\end{align} 
which, along with, \eqref{8.2-13} yields that 
\begin{align}\label{8-32-1}
   I_{1}=|\int_{0}^{T}\langle\widetilde{v}_{n}(s)-\widetilde{v}(s), (u_{n}(s)\cdot \nabla)\int_{s}^{T}\mathcal{P}_{n}\Phi(t)\mathrm{d}t\rangle\mathrm{d}s|\rightarrow 0.
\end{align}
Moreover, by H\"{o}lder's inequality 
and \eqref{embed}, we also have 
\begin{align}\label{8-33}
I_{2}&\lesssim
\|\widetilde{v}\|_{L_{t}^{2}\dot{H}_{x}^{1}}(\|u_{n}\|_{L_{t}^{2}L_{x}^{2}}\|\mathcal{P}_{n}\Phi-\Phi\|_{C_{t}L_{x}^{\infty}}+\|u_{n}-u\|_{L_{t}^{2}L_{x}^{2}}\| \Phi\|_{C_{t,x}})\notag\\
&\lesssim
\|\widetilde{v}\|_{L_{t}^{2}L_{x}^{2}}^{1-\frac{1}{\a}}\|\widetilde{v}\|_{L_{t}^{2}\dot{H}_{x}^{\a}}^{\frac{1}{\a}}(\|u\|_{L_{t}^{2}L_{x}^{\infty}}\|\mathcal{P}_{n}\Phi-\Phi\|_{C_{t}H_{x}^{2}}+\|u_{n}-u\|_{L_{t}^{2}L_{x}^{2}}\|\Phi\|_{C_{t,x}})\rightarrow 0.
\end{align}
Thus, we conclude from \eqref{8-32}, \eqref{8-33.0}, \eqref{8-32-1} and \eqref{8-33} that
\begin{align}\label{8-34}
	|\int_{0}^{T}\langle(u_{n}(s)\cdot \nabla)\widetilde{v}_{n}(s), \int_{s}^{T}\mathcal{P}_{n}\Phi(t)\mathrm{d}t\rangle\mathrm{d}s-\int_{0}^{T}\langle(u(s)\cdot \nabla)\widetilde{v}(s), \int_{s}^{T}\Phi(t)\mathrm{d}t\rangle\mathrm{d}s|\rightarrow 0.
\end{align}
One can also take the limit for the remainning  terms in \eqref{8-28} and thus get
\begin{align*}
	\int_{0}^{T}\langle \widetilde{v}(t)-\widetilde{v}_{0}, \Phi(t)\rangle\mathrm{d}t	&=-\int_{0}^{T}	\langle \nu(-\Delta)^{\a}\widetilde{v}(s), \int_{s}^{T}\Phi(t)\mathrm{d}t\rangle\mathrm{d}s\notag\\
	&\quad-\int_{0}^{T}\langle(u(s)\cdot \nabla)(\widetilde{v}(s)+z_{1}(s))-(B(s)\cdot \nabla )(\widetilde{H}(s)+z_{2}(s)), \int_{s}^{T}\Phi(t)\mathrm{d}t\rangle\mathrm{d}s,
\end{align*}
which, via  Fubini's Theorem, yields
\begin{align*}
\int_{0}^{T}\langle \widetilde{v}(t)-\widetilde{v}_{0}, \Phi(t)\rangle\mathrm{d}t	&=-\int_{0}^{T}	\langle \int_{0}^{t}\nu(-\Delta)^{\a}\widetilde{v}(s)\mathrm{d}s, \Phi(t)\rangle\mathrm{d}t\notag\\
&\quad-\int_{0}^{T}\langle\int_{0}^{t}(u(s)\cdot \nabla)(\widetilde{v}(s)+z_{1}(s))-(B(s)\cdot \nabla )(\widetilde{H}(s)+z_{2}(s))\mathrm{d}s, \Phi(t)\rangle\mathrm{d}t.
\end{align*}
Taking into account $\Phi(t,x)=\psi(t)\phi(x)$, 
we come to 
\begin{align*}
\int_{0}^{T}\langle \widetilde{v}(t)-\widetilde{v}_{0}, \phi\rangle\psi(t)\mathrm{d}t
	&=-\int_{0}^{T}	\langle \int_{0}^{t}\nu(-\Delta)^{\a}\widetilde{v}(s)\mathrm{d}s, \phi\rangle\psi(t)\mathrm{d}t\notag\\
	&\quad-\int_{0}^{T}\langle\int_{0}^{t}(u(s)\cdot \nabla)(\widetilde{v}(s)+z_{1}(s))-(B(s)\cdot \nabla )(\widetilde{H}(s)+z_{2}(s))\mathrm{d}s, \phi\rangle\psi(t)\mathrm{d}t.
\end{align*}
Since $\psi$ is arbitrary,  we get for $t\in[0,T]$,
\begin{align*}
	&\quad\langle \widetilde{v}(t)-\widetilde{v}_{0}, \phi\rangle	\notag\\
	&=-\langle \int_{0}^{t}\nu(-\Delta)^{\a}\widetilde{v}(s)\mathrm{d}s, \phi\rangle
	-\langle\int_{0}^{t}(u(s)\cdot \nabla)(\widetilde{v}(s)+z_{1}(s))-(B(s)\cdot \nabla )(\widetilde{H}(s)+z_{2}(s))\mathrm{d}s, \phi\rangle.
\end{align*} 

Similarly, it also holds \begin{align*}
	&\quad\langle \widetilde{H}(t)-\widetilde{H}_{0}, \phi\rangle	\notag\\
	&=-\langle \int_{0}^{t}\nu(-\Delta)^{\a}\widetilde{H}(s)\mathrm{d}s, \phi\rangle
	-\langle\int_{0}^{t}(u(s)\cdot \nabla)(\widetilde{H}(s)+z_{2}(s))-(B(s)\cdot \nabla )(\widetilde{v}(s)+z_{1}(s))\mathrm{d}s, \phi\rangle.
\end{align*}
Thus, $(\widetilde{v},\widetilde{H})$ 
solves \eqref{8.3.2}, and $(v,H)$ 
is a solution to \eqref{8.3}.

Next, we will prove the uniqueness of {solutions} to system \eqref{8.3}.
Let $(v,H)$ and $(v',H')$ be two solutions of system \eqref{8.3} with the same initial data. Let  $w=v-v'$, $D=H-H'$. Then,
\begin{eqnarray}\label{8-45}
	\left\{\begin{array}{lc}
		\p_{t}w+\nu(-\Delta)^{\a}w+(u\cdot \nabla)w-(B\cdot \nabla )D+\nabla\widetilde{P} =0 ,\\
		\p_{t}D+\nu(-\Delta)^{\a}D+(u\cdot \nabla )D-(B\cdot \nabla )w =0,\\
		\div w=0, \quad\div D=0,\\
		w(0)=0,\quad D(0)=0.
	\end{array}\right.
\end{eqnarray}
As in \eqref{8-20-2}, the energy bound holds 
\begin{align*}
	\|(w(t),D(t))\|_{\mathcal{L}_{x}^{2}}^{2} 
+\int_{0}^{t}\|(w(s),D(s))\|_{\dot{\mathcal{H}}_{x}^{\a}}^{2}\mathrm{d}s\lesssim \|(w(0),D(0))\|_{\mathcal{L}_{x}^{2}}^{2} 
=0,
\end{align*}
which yields that $(w,D)=(0,0)$, 
and thus $(v,H) = (v',H')$. 

By virtue of the uniqueness, 
we infer that the limits \eqref{8.2-13} and \eqref{8.2-14} 
actually hold along the whole sequence $\{n\}$. As $\{(\widetilde{v}_{n},\widetilde{H}_{n})\}$ is $(\F_t)$-adapted, so are the limits  $(\widetilde{v},\widetilde{H})$ and $(v, H)$. The proof is completed.
 \hfill$\square$\\

\paragraph{\bf Proof of Theorem \ref{uniq-sharp}} 
Let $(v,H)$ be the solution of (\ref{8.3}) with  the initial data $(u_{0}, B_{0})$. By  Lemma \ref{existence},  $(v, H) \in  C_{w}([0, T];\mathcal{L}_{x}^{2})\cap  L^{2}(0, T;\mathcal{H}_{x}^{\a})$, $\textbf{P}$-a.s.

We have that  $(u, B)=(v,H)$, $\textbf{P}$-a.s. Actually, setting $(w, D):=(u-v, B-H)$, from   (\ref{1.1}) and (\ref{8.3}) we derive 
\begin{eqnarray}\label{8.5}
	\left\{\begin{array}{lc}
		\p_{t} w+\nu(-\Delta)^{\a}w+(u\cdot \nabla)w-(B\cdot \nabla )D +\nabla \widetilde{P}  =0,\\
	\p_{t}D+\nu(-\Delta)^{\a}D+(u\cdot \nabla )D-(B\cdot \nabla )w=0,\\
		\div\, w=0, \quad\div D=0,\\
		w(0)=0,\quad D(0)=0.
	\end{array}\right.
\end{eqnarray}
System \eqref{8.5} can be treated  as a linearized deterministic MHD system. According to  \citep[(B.7)]{ZZL-MHD-sharp}, it holds that $(u,B)=(v,H)$. Thus, $(u, B)$ lies in the Leray-Hopf class $(u, B)\in   C_{w}([0,T];\mathcal{L}_{x}^{2})\cap L^{2}(0,T; \dot{\mathcal{H}}_{x}^{\a})$, $\textbf{P}-a.s.$.

Now, let $(u_{1}, B_{1})$ and $(u_{2}, B_{2})$  be two solutions to (\ref{1.1}) with the same initial data $(u_{0}, B_{0})$, and $(u_{1}, B_{1})$, $(u_{2}, B_{2})\in X_{(0,T)}^{s,\gamma,p}\times X_{(0,T)}^{s,\gamma,p}\cap  C_{w}([0,T];\mathcal{L}_{x}^{2})\cap L^{2}(0,T; \dot{\mathcal{H}}_{x}^{\a})$, $\textbf{P}$-a.s.. 

Let $(w^{(u)},D^{(B)}):=(u_{1}-u_{2},B_{1}-B_{2})$, one has that
\begin{eqnarray*}
	\left\{\begin{array}{lc}
		\p_{t}  w^{(u)}+\nu(-\Delta)^{\a}w^{(u)}+(w^{(u)}\cdot \nabla)u_{1}+\big((u_{1}-w^{(u)})\cdot \nabla\big)w^{(u)}-(D^{(B)}\cdot \nabla )B_{1}+\big((B_{1}-D^{(B)})\cdot \nabla\big)D^{(B)}\\ 
		\quad+\nabla \widetilde{P}=0,\\
		\p_{t}D^{(B)}+\nu(-\Delta)^{\a}D^{(B)}+(u_{1}\cdot \nabla)D^{(B)}+(w^{(u)}\cdot \nabla )(B_{1}-D^{(B)})-(B_{1}\cdot \nabla )w^{(u)}-(D^{(B)}\cdot\nabla)(u_{1}-w^{(u)}) \\
		\quad =0,\\
		\div w^{(u)}=0, \quad\div D^{(B)}=0,\\
		w^{(u)}(0)=0,\quad D^{(B)}(0)=0.
	\end{array}\right.
\end{eqnarray*}
Then, according to Lemma B.4 in \citep{ZZL-MHD-sharp},  it follows that $( w^{(u)},D^{(B)})=(0,0)$, and thus the uniqueness $(u_{1}, B_{1}) \equiv (u_{2}, B_{2})$ holds $\textbf{P}$-a.s. Therefore, the proof of Theorem \ref{uniq-sharp} is completed.\hfill$\square$\\

\medskip
Finally, we prove the noise vanishing limit in Theorem \ref{Theorem Vanishing noise}.
\paragraph{\bf Proof of Theorem \ref{Theorem Vanishing noise}}
Let $z_{1}^{(\epsilon_{n})}$, $z_{2}^{(\epsilon_{n})}$ solve the linear stochastic system \eqref{1.2} 
with  $W^{(i)}$ replaced by $\epsilon_{n} W^{(i)}$,
$i=1,2$. Then, we have 
\begin{align}\label{1.3.5.11}
z_{1}^{(\epsilon_{n})}:=z_{1}^{u}+Z_{1}^{(\epsilon_{n})},\quad z_{2}^{(\epsilon_{n})}:=z_{2}^{B}+Z_{2}^{(\epsilon_{n})},\quad t\geq 0,
\end{align}
where
\begin{align}
 &z_{1}^{u}(t):=e^{t\(-\nu(-\Delta)^{\a}-I\)}u_{0},\quad z_{2}^{B}(t)=:e^{t\(-\nu(-\Delta)^{\a}-I\)}B_{0},\label{1.3.5.12}\\ &Z_{i}^{(\epsilon_{n})}(t):=\epsilon_{n}\int_{0}^{t}e^{(t-\sigma)\(-\nu(-\Delta)^{\a}-I\)}\mathrm{d} W^{(i)}(\sigma), \quad i=1,2.\label{1.3.5.13}
 \end{align}

Set the space-time mollified versions 
\begin{align*}
	&u_{n}:=(u*_{x}\varrho_{\lambda_{n}^{-1}})*_{t} \vartheta_{\lambda_{n}^{-1}},\quad
	B_{n}:=(B*_{x}\varrho_{\lambda_{n}^{-1}})*_{t} \vartheta_{\lambda_{n}^{-1}},\quad z_{i,n}^{(\epsilon_{n})}:=(z_{i}^{(\epsilon_{n})}*_{x}\varrho_{\lambda_{n}^{-1}})*_{t}\vartheta_{\lambda_{n}^{-1}}\\
	&z_{1,n}^{u}:=(z_{1}^{u}*_{x}\varrho_{\lambda_{n}^{-1}})*_{t} \vartheta_{\lambda_{n}^{-1}},\quad
	z_{2,n}^{B}:=(z_{2}^{B}*_{x}\varrho_{\lambda_{n}^{-1}})*_{t} \vartheta_{\lambda_{n}^{-1}},\quad
	Z_{i,n}^{(\epsilon_{n})}:=(Z_{i}^{(\epsilon_{n})}*_{x}\varrho_{\lambda_{n}^{-1}})*_{t} \vartheta_{\lambda_{n}^{-1}},
\end{align*}
and the frequency truncated versions 
	\begin{align*}
	\widehat{z}_{1,n}^{(\epsilon_{n})}:=z_{1,n}^{u}+\bbp_{\leq \lambda_{n}^{15}}Z_{1,n}^{(\epsilon_{n})},\quad	\widehat{z}_{2,n}^{(\epsilon_{n})}:=z_{2,n}^{B}+\bbp_{\leq \lambda_{n}^{15}}Z_{2,n}^{(\epsilon_{n})}.
\end{align*} 
Let
\begin{align}\label{1.3.5.17}
\Bu_{n}:=u_{n}-z_{1,n}^{u}, \quad \Bb_{n}:=B_{n}-z_{2,n}^{B}.
\end{align}

It follows from (\ref{1.3.1}) that
\begin{eqnarray}\label{1.3.55}
	\left\{\begin{array}{lc}
		\p_{t}\Bu_{n}+\nu(-\Delta)^{\a}\Bu_{n}-\widehat{z}_{1,n}^{(\epsilon_{n})}+\div\big((\Bu_{n}+\widehat{z}_{1,n}^{(\epsilon_{n})})\mathring{\otimes} (\Bu_{n}+\widehat{z}_{1,n}^{(\epsilon_{n})})-(\Bb_{n}+z_{2,n,n}^{(\epsilon_{n})})\mathring{\otimes}(\Bb_{n}+\widehat{z}_{2,n}^{(\epsilon_{n})})\big)+\nabla P_{n}
		\\ \quad=\div \run^{u},\\
		\p_{t}\Bb_{n}+\nu(-\Delta)^{\a}\Bb_{n}-\widehat{z}_{2,n}^{(\epsilon_{n})}+\div\big((\Bu_{n}+\widehat{z}_{1,n}^{(\epsilon_{n})})\otimes(\Bb_{n}+\widehat{z}_{2,n}^{(\epsilon_{n})})-(\Bb_{n}+\widehat{z}_{2,n}^{(\epsilon_{n})})\otimes(\Bu_{n}+\widehat{z}_{1,n}^{(\epsilon_{n})})\big)\\ \quad=\div \run^{B},\\
		\div \Bu_{n}=0,\quad \div \Bb_{n}=0, \quad \Bu_{n}(0)=0,\quad\Bb_{n}(0)=0,
	\end{array}\right.
\end{eqnarray}
where the stresses $\run^{u}$ and $\run^{B}$ are given by
\begin{align}
	&\run^{u}:=(\Bu_{n}+\widehat{z}_{1,n}^{(\epsilon_{n})})\mathring{\otimes} (\Bu_{n}+\widehat{z}_{1,n}^{(\epsilon_{n})})-(\Bb_{n}+\widehat{z}_{2,n}^{(\epsilon_{n})})\mathring{\otimes}(\Bb_{n}+\widehat{z}_{2,n}^{(\epsilon_{n})})-\big((u\mathring{\otimes} u- B\mathring{\otimes} B)*_{x}\varrho_{\lambda_{n}^{-1}}\big)*_{t}\vartheta_{\lambda_{n}^{-1}}\notag\\
	&\quad\quad\quad-\mathcal{R}^{u}(\bbp_{\leq \lambda_{n}^{15}}Z_{1,n}^{(\epsilon_{n})}),
	\label{1.3.56}\\
		&\run^{B}:= (\Bb_{n}+\widehat{z}_{2,n}^{(\epsilon_{n})}) \otimes  (\Bu_{n}+\widehat{z}_{1,n}^{(\epsilon_{n})})-(\Bu_{n}+\widehat{z}_{1,n}^{(\epsilon_{n})}) \otimes (\Bb_{n}+\widehat{z}_{2,n}^{(\epsilon_{n})})-\big((B \otimes u- u \otimes B)*_{x}\varrho_{\lambda_{n}^{-1}}\big)*_{t}\vartheta_{\lambda_{n}^{-1}}\notag\\
	&\quad\quad\quad-\mathcal{R}^{B}(\bbp_{\leq \lambda_{n}^{15}}Z_{2,n}^{(\epsilon_{n})}),\label{1.3.57}
\end{align}
and
\begin{equation}\label{1.3.58}
	P_{n}:= P*_{x}\varrho_{\lambda_{n}^{-1}}*_{t}\vartheta_{\lambda_{n}^{-1}}-|\Bu_{n}+\widehat{z}_{1,n}^{(\epsilon_{n})}|^{2}+|\Bb_{n}+\widehat{z}_{2,n}^{(\epsilon_{n})}|^{2}+\big((|u|^{2}-|B|^{2})*_{x}\varrho_{\lambda_{n}^{-1}}\big)*_{t}\vartheta_{\lambda_{n}^{-1}}.
\end{equation}

\textbf{Claim:} For $a$ sufficiently large, $(\Bu_{n}, \Bb_{n}, \run^{u}, \run^{B})$ satisfies 
the iterative estimates (\ref{2.11})-(\ref{2.13}) at level $q=n  (\geq1)$.

For this purpose, let us start with the 
most delicate decay estimate (\ref{2.13}).
Let
 \begin{align}\label{1.3.5.22}
 \epsilon_{n}:=\lambda_{n}^{-1},\quad \widetilde{M}:=\|u\|_{H_{t,x}^{\widetilde{\beta}}}+\|B\|_{H_{t,x}^{\widetilde{\beta}}}.
 \end{align}
Using the Minkowski inequality and the Slobodetskii-type norm of Sobolev spaces,
we have (see \citep[(6.35)]{ZZL-MHD})
\begin{equation}\label{1.3.59}
	\|u-u_{n}\|_{L_{t}^{2}L_{x}^{2}}\lesssim \lambda_{n}^{-\widetilde{\beta}}\|u\|_{H_{t,x}^{\widetilde{\beta}}}\lesssim \lambda_{n}^{-\widetilde{\beta}}\widetilde{M}.
\end{equation}
By   (\ref{1.3.5.22}) and   H\"{o}lder's inequality,  
\begin{align}\label{1.3.60}
	&\quad\|(\Bu_{n}+\widehat{z}_{1,n}^{(\epsilon_{n})})\mathring{\otimes}(\Bu_{n}+\widehat{z}_{1,n}^{(\epsilon_{n})})-u_{n}\mathring{\otimes}u_{n}\|_{L_{t}^{1}L_{x}^{1}}\notag\\
	&\lesssim\|(\widehat{z}_{1,n}^{(\epsilon_{n})}-z_{1,n}^{u})\mathring{\otimes}(u_{n}-z_{1,n}^{u}+\widehat{z}_{1,n}^{(\epsilon_{n})})\|_{L_{t}^{1}L_{x}^{1}}
	+\|u_{n}\mathring{\otimes}(\widehat{z}_{1,n}^{(\epsilon_{n})}-z_{1,n}^{u})\|_{L_{t}^{1}L_{x}^{1}}\notag\\
	&\lesssim
	\|\widehat{z}_{1,n}^{(\epsilon_{n})}-z_{1,n}^{u}\|_{L_{t}^{2}L_{x}^{2}}^{2}+\|u_{n}\|_{L_{t}^{2}L_{x}^{2}}\|\widehat{z}_{1,n}^{(\epsilon_{n})}-z_{1,n}^{u}\|_{L_{t}^{2}L_{x}^{2}}\notag\\
	&\lesssim
	\|\bbp_{\leq \lambda_{n}^{15}}Z_{1,n}^{(\epsilon_{n})}\|_{L_{t}^{2}L_{x}^{2}}^{2}+\widetilde{M}\|\bbp_{\leq \lambda_{n}^{15}}Z_{1,n}^{(\epsilon_{n})}\|_{L_{t}^{2}L_{x}^{2}}.
\end{align}
Similarly, we have
\begin{align}\label{1.3.61}
	\|(\Bb_{n}+\widehat{z}_{2,n}^{(\epsilon_{n})})\mathring{\otimes}(\Bb_{n}+\widehat{z}_{2,n}^{(\epsilon_{n})})-B_{n}\mathring{\otimes}B_{n}\|_{L_{t}^{1}L_{x}^{1}}
	\lesssim
	\|\bbp_{\leq \lambda_{n}^{15}}Z_{2,n}^{(\epsilon_{n})}\|_{L_{t}^{2}L_{x}^{2}}^{2}+\widetilde{M}\|\bbp_{\leq \lambda_{n}^{15}}Z_{2,n}^{(\epsilon_{n})}\|_{L_{t}^{2}L_{x}^{2}}.
\end{align} 
Moreover, by the mollification estimate {(see \citep[(6.36)]{ZZL-MHD})}, we get
\begin{align}\label{1.3.62}
		&\quad\|u_{n}\mathring{\otimes} u_{n}-\big((u \mathring {\otimes} u)*_{x}\varrho_{\lambda_{n}^{-1}}\big)*_{t}\vartheta_{\lambda_{n}^{-1}}\|_{L_{t}^{1}L_{x}^{1}}
	\lesssim	
	\lambda_{n}^{-2\widetilde{\beta}}\|u\|^{2}_{H_{t,x}^{\widetilde{\beta}}}\lesssim \lambda_{n}^{-2\widetilde{\beta}}\widetilde{M}^{2},
\end{align}
and
\begin{align}\label{1.3.63-0}
	\quad\|B_{n}\mathring{\otimes} B_{n}-\big((B \mathring {\otimes} B)*_{x}\varrho_{\lambda_{n}^{-1}}\big)*_{t}\vartheta_{\lambda_{n}^{-1}}\|_{L_{t}^{1}L_{x}^{1}}
	\lesssim \lambda_{n}^{-2\widetilde{\beta}}\widetilde{M}^{2}.
\end{align}
Therefore, using (\ref{1.3.56})  and Proposition \ref{Proposition stochastic} and combining (\ref{1.3.60})-(\ref{1.3.63-0}) together, we obtain
\begin{align}\label{1.3.63}
	\|\run^{u}\|_{\Lo^{r}_\om L_{t}^{1}L_{x}^{1}}
	&\lesssim	 \sum_{i=1}^{2}\big(\|\bbp_{\leq \lambda_{n}^{15}}Z_{i,n}^{(\epsilon_{n})}\|^{2}_{\Lo^{2r}_{\om}L_{t}^{2}L_{x}^{2}}+\widetilde{M}\|\bbp_{\leq \lambda_{n}^{15}}Z_{i,n}^{(\epsilon_{n})}\|_{\Lo^{r}_{\om}L_{t}^{2}L_{x}^{2}}\big)\notag\\
	&\quad
	+\|\mathcal{R}^{u}(\bbp_{\leq \lambda_{n}^{15}}Z_{1,n}^{(\epsilon_{n})})\|_{\Lo^{r}_{\om}L_{t}^{2}L_{x}^{2}}
	+\lambda_{n}^{-2\widetilde{\beta}}\widetilde{M}^{2}\notag\\
	&\lesssim \epsilon_{n}^{2}(2r-1)L^{2}+\epsilon_{n}(r-1)^{\frac{1}{2}}L+\lambda_{n}^{-2\widetilde{\beta}}\notag\\
	&\lesssim \lambda_{n}^{-1}+\lambda_{n}^{-2\widetilde{\beta}},
\end{align}
where we also used (\ref{1.3.59}) in the last step.

Analogous arguments yield that
\begin{align}\label{1.3.64}
	\|\run^{B}\|_{\Lo^{r}_\om L_{t}^{1}L_{x}^{1}}
     \lesssim \lambda_{n}^{-1}+\lambda_{n}^{-2\widetilde{\beta}}.
\end{align}
Thus,
estimate  (\ref{2.13}) is verified at level $q=n$.

Moreover, using the Sobolev embedding $H_{t,x}^{3}\hookrightarrow L_{t,x}^{\infty}$, we have that for $0\leq N\leq 4$,
\begin{align*}
	\|(\Bu_{n},\Bb_{n})\|_{\Lo^{m}_{\om}\mathcal{C}_{t,x}^{N}}
	&\lesssim
	\|(u_{n},B_{n})\|_{\Lo^{m}_{\om}\mathcal{H}_{t,x}^{3+N}}+	\|(z_{1,n}^{u},z_{2,n}^{B})\|_{\Lo^{m}_{\om}\mathcal{H}_{t,x}^{3+N}}\notag\\
&\lesssim
\lambda_{n}^{3+N}\|(u,B)\|_{\Lo^{m}_{\om}L_{t}^{2}\mathcal{L}_{x}^{2}}+\lambda_{n}^{3+N}\|(z_{1}^{u},z_{2}^{B})\|_{\Lo^{m}_{\om}L_{t}^{2}\mathcal{L}_{x}^{2}}\notag\\
&\lesssim
\lambda_{n}^{3+N}(\widetilde{M}+M)
\lesssim
\lambda_{n}^{3+N},
\end{align*} 
where $M$ is as in \eqref{in}. 
This yields (\ref{2.11}) at level $q=n$.

Regarding estimate (\ref{2.12}),  by  (\ref{1.3.5.17}), (\ref{1.3.5.22}) and H\"{o}lder's inequality,
\begin{align}\label{1.3.66}
	&\quad\|(\Bu_{n}+\widehat{z}_{1,n}^{(\epsilon_{n})})\mathring{\otimes}(\Bu_{n}+\widehat{z}_{1,n}^{(\epsilon_{n})})-u_{n}\mathring{\otimes}u_{n}\|_{C_{t}L_{x}^{1}}\notag\\
	&\lesssim\|(\widehat{z}_{1,n}^{(\epsilon_{n})}-z_{1,n}^{u})\mathring{\otimes}(u_{n}-z_{1,n}^{u}+\widehat{z}_{1,n}^{(\epsilon_{n})})\|_{C_{t}L_{x}^{1}}
	+\|u_{n}\mathring{\otimes}(\widehat{z}_{1,n}^{(\epsilon_{n})}-z_{1,n}^{u})\|_{C_{t}L_{x}^{1}}\notag\\
	&\lesssim
	\|\widehat{z}_{1,n}^{(\epsilon_{n})}-z_{1,n}^{u}\|_{C_{t}L_{x}^{2}}^{2}+\|u_{n}\|_{C_{t}L_{x}^{2}}\|\widehat{z}_{1,n}^{(\epsilon_{n})}-z_{1,n}^{u}\|_{C_{t}L_{x}^{2}}\notag\\
	&\lesssim
	\|\bbp_{\leq \lambda_{n}^{15}}Z_{1,n}^{(\epsilon_{n})}\|_{C_{t}L_{x}^{2}}^{2}+\widetilde{M}\|\bbp_{\leq \lambda_{n}^{15}}Z_{1,n}^{(\epsilon_{n})}\|_{L_{t}^{2}L_{x}^{2}}.
\end{align}
Similarly,  
\begin{equation}\label{1.3.67}
	\|(\Bb_{n}+\widehat{z}_{2,n}^{(\epsilon_{n})})\mathring{\otimes}(\Bb_{n}+\widehat{z}_{2,n}^{(\epsilon_{n})})-B_{n}\mathring{\otimes}B_{n}\|_{C_{t}L_{x}^{1}}
	\lesssim
	\|\bbp_{\leq \lambda_{n}^{15}}Z_{2,n}^{(\epsilon_{n})}\|_{L_{t}^{2}L_{x}^{2}}^{2}+\widetilde{M}\|\bbp_{\leq \lambda_{n}^{15}}Z_{2,n}^{(\epsilon_{n})}\|_{L_{t}^{2}L_{x}^{2}}.
\end{equation}
An application of the  Sobolev embedding
 $W_{t,x}^{5,1} \hookrightarrow L_{t,x}^{\infty}$  gives
 \begin{align}\label{1.3.68}
 	&\quad\|u_{n}\mathring{\otimes} u_{n}-\big((u \mathring {\otimes} u)*_{x}\varrho_{\lambda_{n}^{-1}}\big)*_{t}\vartheta_{\lambda_{n}^{-1}}\|_{C_{t}L_{x}^{1}}\notag\\
 	&\lesssim	
 	\|u_{n}\mathring{\otimes}u_{n}\|_{W_{t,x}^{5,1}}+\|\big((u \mathring {\otimes} u)*_{x}\varrho_{\lambda_{n}^{-1}}\big)*_{t}\vartheta_{\lambda_{n}^{-1}}\|_{W_{t,x}^{5,1}}\notag\\
 	&\lesssim	
 	\sum_{0\leq N_{1}+N_{2}\leq 5}
 	\|\p_{t}^{N_{1}}\nabla^{N_{2}}(u_{n}\mathring{\otimes}u_{n})\|_{L_{t}^{1}L_{x}^{1}}+\|\p_{t}^{N_{1}}\nabla^{N_{2}}\big((u \mathring {\otimes} u)*_{x}\varrho_{\lambda_{n}^{-1}}\big)*_{t}\vartheta_{\lambda_{n}^{-1}})\|_{L_{t}^{1}L_{x}^{1}}\notag\\
 	&\lesssim	
 	\lambda_{n}^{5}\|u_{n}\|^{2}_{L_{t}^{2}L_{x}^{2}}+\lambda_{n}^{5}\|u\|^{2}_{L_{t}^{2}L_{x}^{2}}\lesssim	
 	\lambda_{n}^{5}\widetilde{M}^{2}, 
 \end{align}
and 
\begin{align}\label{1.3.69}
\|B_{n}\mathring{\otimes} B_{n}-\big((B \mathring {\otimes} B)*_{x}\varrho_{\lambda_{n}^{-1}}\big)*_{t}\vartheta_{\lambda_{n}^{-1}}\|_{C_{t}L_{x}^{1}}
\lesssim	
\lambda_{n}^{5}\widetilde{M}^{2}.
\end{align}
Thus, combining Proposition \ref{Proposition stochastic}, (\ref{1.3.56})  and (\ref{1.3.66})-(\ref{1.3.69}) together, we obtain
\begin{align}\label{1.3.70-1}
	\|\run^{u}\|_{\Lo^{m}_{\om}C_{t}L_{x}^{1}}
&\lesssim	
\sum_{i=1}^{2}\big(\|\bbp_{\leq \lambda_{n}^{15}}Z_{i,n}^{(\epsilon_{n})}\|^{2}_{\Lo^{2m}_{\om}C_{t}L_{x}^{2}}+\widetilde{M}\|\bbp_{\leq \lambda_{n}^{15}}Z_{i,n}^{(\epsilon_{n})}\|_{\Lo^{2m}_{\om}C_{t}L_{x}^{2}}\big)\notag\\
&\quad
+\|\mathcal{R}^{u}(\bbp_{\leq \lambda_{n}^{15}}Z_{1,n}^{(\epsilon_{n})})\|_{\Lo^{2m}_{\om}C_{t}L_{x}^{2}}
+\lambda_{n}^{5}\widetilde{M}^{2}\notag\\
&\lesssim
\epsilon_{n}^{2}(2m-1)L^{2}+\epsilon_{n}(2m-1)^{\frac{1}{2}}L+\lambda_{n}^{5} \notag \\  
&\lesssim \lambda_{n}^{5}(2m-1)L^{2} .
\end{align}
Similar estimate also holds for $\run^{B}$.
Thus,  we prove estimate (\ref{2.12}) at level $q=n$, as claimed.

\medskip
Now, by virtue of Proposition \ref{Proposition Main iteration}, 
we obtain a sequence of approximate solutions $(\Bu_{n,q},\Bb_{n,q},\runq^{u},\runq^{B})_{q\geq n}$, 
which satisfy (\ref{2.11})-(\ref{2.13}),
$\{(\Bu_{n,q}, \Bb_{n,q})\}_{q\geq n}$ are  Cuachy sequences in $\Lo^{2r}_{\om}L_{t}^{2}\mathcal{L}_{x}^{2}$, 
and $\{(\runq^{u},\runq^{B})\}$ converges strongly to $(0, 0)$ in $\Lo^{r}_{\om}L_{t}^{1}\mathcal{L}_{x}^{1}$. 
Thus,
passing to the limit as 
$q\rightarrow \infty$ we obtain a solution $(u^{(\epsilon_{n})}, B^{(\epsilon_{n})})\in \Lo^{2r}_{\om}L_{t}^{2}\mathcal{L}_{x}^{2}$ to (\ref{1.1}) with $\epsilon_{n}W^{(i)}$ replacing $W^{(i)}$, $i=1,2$.

Moreover, by (\ref{1.3.5.17}),  
\begin{align}\label{1.3.73.01}
	\|u^{(\epsilon_{n})}-u_{n}\|_{\Lo^{2r}_{\om}L_{t}^{2}L_{x}^{2}}
	&\leq \|u^{(\epsilon_{n})}-(\Bu_{n}+z_{1}^{(\epsilon_{n})})\|_{\Lo^{2r}_{\om}L_{t}^{2}L_{x}^{2}}+\|(\Bu_{n}+z_{1}^{(\epsilon_{n})})-u_{n}\|_{\Lo^{2r}_{\om}L_{t}^{2}L_{x}^{2}}\notag\\
	&\leq\sum_{q\geq n}\|\Bu_{q+1}-\Bu_{q}\|_{\Lo^{2r}_{\om}L_{t}^{2}L_{x}^{2}}+\|z_{1,n}^{u}-z_{1}^{u}\|_{\Lo^{2r}_{\om}L_{t}^{2}L_{x}^{2}}+\|Z_{1}^{(\epsilon_{n})}\|_{\Lo^{2r}_{\om}L_{t}^{2}L_{x}^{2}},
\end{align}
where $z_{1}^{u}$ and $Z_{1}^{(\epsilon_{n})}$ are given by (\ref{1.3.5.12}) and (\ref{1.3.5.13}), 
respectively. 

In order to control  $\|z_{1,n}^{u}-z_{1}^{u}\|_{\Lo^{2r}_{\om}L_{t}^{2}L_{x}^{2}}$, we consider the temporal regimes $[0,\lambda_{n}^{-1/2}]$ and   $(\lambda_{n}^{-1/2}, T]$ separately. 
In the first case where $t\in(\lambda_{n}^{-1/2}, T]$, by the mollification estimates, (\ref{5.22.01}) and (\ref{5.22.04}), we have
 \begin{align*}
 	\|z_{1,n}^{u}-z_{1}^{u}\|_{L_{(\lambda_{n}^{-1/2}, T]}^{2}L_{x}^{2}}\lesssim \lambda_{n}^{-\frac{1}{2}}(\|z_{1}^{u}\|_{C_{(\lambda_{n}^{-1/2}-\lambda_{n}^{-1}, T]}^{\frac{1}{2}}L_{x}^{2}}+\|z_{1}^{u}\|_{C_{(\lambda_{n}^{-1/2}, T]}H_{x}^{\frac{1}{2}}})\lesssim \lambda_{n}^{-\frac{1}{2}} (1+\lambda_{n}^{\frac{1}{4}})M,
 	\end{align*}
 where $M$ is as in (\ref{in}). While in the  case where $t\in[0,\lambda_{n}^{-1/2}]$, 
  \begin{align*}
 	\|z_{1,n}^{u}-z_{1}^{u}\|_{L_{[0,\lambda_{n}^{-1/2}]}^{2}L_{x}^{2}}\leq \|z_{1,n}^{u}\|_{L_{[0,\lambda_{n}^{-1/2}]}^{2}L_{x}^{2}}+\|z_{1}^{u}\|_{L_{[0,\lambda_{n}^{-1/2}]}^{2}L_{x}^{2}}\lesssim \lambda_{n}^{-1/4}\|z_{1}^{u}\|_{C_{[0,\lambda_{n}^{-1/2}]}L_{x}^{2}}\lesssim \lambda_{n}^{-\frac{1}{4}}M.
 \end{align*}
Hence, we obtain
  \begin{align}\label{y3}
	\|z_{1,n}^{u}-z_{1}^{u}\|_{L_{t}^{2}L_{x}^{2}}\lesssim \lambda_{n}^{-\frac{1}{2}} (1+\lambda_{n}^{\frac{1}{4}})M+\lambda_{n}^{-\frac{1}{4}}M\lesssim \lambda_{n}^{-\frac{1}{4}}M.
\end{align}

In view of   Proposition \ref{Proposition stochastic}, (\ref{1.3.73.01}) and (\ref{y3}), we arrive at
\begin{align}\label{1.3.73.1}
	\|u^{(\epsilon_{n})}-u_{n}\|_{\Lo^{2r}_{\om}L_{t}^{2}L_{x}^{2}}
	&\lesssim 	\sum_{q\geq n} M^{*}\dqq^{\frac{1}{2}}+\lambda_{n}^{-\frac{1}{4}}M+\lambda_{n}^{-1}(2r-1)^{\frac{1}{2}}L 	\lesssim \sum_{q\geq n} \dqq^{\frac{1}{2}}+\lambda_{n}^{-\frac{1}{4}}, 
\end{align}  
which, along with   (\ref{2.2-2}) and (\ref{1.3.59}), 
yields that for $n> 10$,
\begin{align}\label{1.3.76-1}
	\|u^{(\epsilon_{n})}-u\|_{\Lo^{2r}_{\om}L_{t}^{2}L_{x}^{2}}
	&\leq
	\|u^{(\epsilon_{n})}-u_{n}\|_{\Lo^{2r}_{\om}L_{t}^{2}L_{x}^{2}} 	+\|u-u_{n}\|_{\Lo^{2r}_{\om}L_{t}^{2}L_{x}^{2}}\notag\\
	&\lesssim \sum_{q\geq n}\dqq^{\frac{1}{2}}+\lambda_{n}^{-\frac{1}{4}}+\lambda_{n}^{-\widetilde{\beta}}\widetilde{M}\notag\\
		&\lesssim \lambda_{1}^{\frac{3\beta}{2}-b^{n}\beta}+\sum_{q\geq 11}\dqq^{\frac{1}{2}}+\lambda_{n}^{-\frac{1}{4}}+\lambda_{n}^{-\widetilde{\beta}}
	 \leq \frac{1}{n},
\end{align}
where we choose $a$ sufficiently large.

Arguing in an analogous manner, 
we also have 
\begin{align}\label{1.3.76}
	\|B^{(\epsilon_{n})}-B\|_{\Lo^{2r}_{\om}L_{t}^{2}L_{x}^{2}}
	\leq \frac{1}{n}.
\end{align}

Therefore, the strong convergence \eqref{1.3.3} follows
and the  proof is completed.\hfill$\square$

\section{Appedix: Standard tools in convex integration}\label{Standard tools}

\begin{lemma} (\textbf{First Geometric Lemma}, \citep[Lemma 4.2]{BBV-IMHD})\label{Lemma First Geometric}
	 There exists a set $\Lambda_u \subset \mathbb{S}^2 \cap \mathbb{Q}^3$ that consists of vectors $k$ with associated orthonormal bases $\left(k, k_1, k_2\right), \varepsilon_u>0$, and smooth positive functions $\gamma_{(k)}$ : $B_{\varepsilon_u}(\mathrm{Id}) \rightarrow \mathbb{R}$, where $B_{\varepsilon_u}(\mathrm{Id})$ is the ball of radius $\varepsilon_u$ centered at the identity in the space of $3 \times 3$ symmetric matrices, such that for $S \in B_{\varepsilon_u}(\mathrm{Id})$ we have the following identity:
\begin{equation}\label{6.1}
	\begin{aligned}
S=\sum_{k \in \Lambda_u} \gamma_{(k)}^2(S) k_1 \otimes k_1.
	\end{aligned}
 \end{equation}
We may choose $\Lambda_u$ such that $\Lambda_B \cap \Lambda_u=\emptyset$.
\end{lemma}
\begin{lemma}
(\textbf{Second Geometric Lemma}, \citep[Lemma 4.1]{BBV-IMHD})) \label{Lemma Second Geometric}
There exists a set $\Lambda_B \subset \mathbb{S}^2 \cap \mathbb{Q}^3$ that consists of vectors $k$ with associated orthonormal bases $\left(k, k_1, k_2\right), \varepsilon_{B}>0$, and smooth positive functions $\gamma_{(k)}: B_{\varepsilon_{B}}(0) \rightarrow \mathbb{R}$, where $B_{\varepsilon_{B}}(0)$ is the ball of radius $\varepsilon_{B}$ centered at 0 in the space of $3 \times 3$ skew-symmetric matrices, such that for $A \in B_{\varepsilon_{B}}(0)$ we have the following identity:
\begin{equation}\label{6.2}
	\begin{aligned}
A=\sum_{k \in \Lambda_B} \gamma_{(k)}^2(A)\left(k_2 \otimes k_1-k_1 \otimes k_2\right) .
\end{aligned}
\end{equation}
\end{lemma}

Note that there exists $N_{\Lambda} \in \mathbb{N}$ such that
\begin{equation}\label{6.3}
	\begin{aligned}
\left\{N_{\Lambda} k, N_{\Lambda} k_1, N_{\Lambda} k_2\right\} \subset N_{\Lambda} \mathbb{S}^2 \cap \mathbb{Z}^3
\end{aligned}
\end{equation} 
Let  $M_*$  denote the geometric constant such that
\begin{equation}\label{6.4}
	\begin{aligned}
\sum_{k \in \Lambda_u}\|\gamma_{(k)}\|_{C^4(B_{\varepsilon_u}(\mathrm{Id}))}+\sum_{k \in \Lambda_B}\|\gamma_{(k)}\|_{C^4(B_{\varepsilon_{B}}(0))} \leq M_* .
\end{aligned}
\end{equation}

The inverse-divergence operator $\mathcal{R}^u$ and $\mathcal{R}^{B}$,   are defined by
	\begin{align}
&(\mathcal{R}^u v)^{k l}:=\partial_k \Delta^{-1} v^l+\partial_l \Delta^{-1} v^k-\frac{1}{2}(\delta_{k l}+\partial_k \partial_l \Delta^{-1}) \operatorname{div} \Delta^{-1} v, \label{6.5}\\
	&(\mathcal{R}^{B} f)_{i j}:=\varepsilon_{i j k}(-\Delta)^{-1}(\operatorname{curl} f)_k,\label{6.6}
\end{align}
where $v$ is mean-free, i.e., $\int_{\mathbb{T}^3} v d x=0$, $f$ is divergence-free, $\varepsilon_{i j k}$ is the Levi-Civita tensor, $i, j, k, l \in\{1,2,3\}$. Note that, the inverse-divergence operator $\mathcal{R}^{u}$ maps mean-free functions to symmetric and trace-free matrices, while the operator $\mathcal{R}^{B}$ returns skew-symmetric matrices. Moreover, it holds that
\begin{flalign*}
\operatorname{div} \mathcal{R}^u(v)=v, \quad \operatorname{div} \mathcal{R}^{B}(f)=f .
\end{flalign*}
Both $|\nabla| \mathcal{R}^u$ and $|\nabla| \mathcal{R}^{B}$ are Calderon-Zygmund operators, and thus, they are bounded in the spaces $L^p$, $1<p<+\infty$. (See \citep{BBV-IMHD}).

We have the following refined H\"{o}lder's inequlity.
\begin{lemma}(\citep[ Lemma 2.4]{CL-transport equation}; see also \citep[Lemma A.3]{ZZL-MHD-sharp} ) \label{Lemma Decorrelation1}
 Let $\theta \in \mathbb{N}$ and $f, g: \mathbb{T}^d \rightarrow \mathbb{R}$ be smooth functions. Then for every $p \in[1,+\infty]$,
\begin{equation}\label{6.7}
	\begin{aligned}
\left|\|f g(\theta \cdot)\|_{L^p(\mathbb{T}^d)}-\|f\|_{L^p(\mathbb{T}^d)}\|g\|_{L^p(\mathbb{T}^d)}\right| \lesssim \theta^{-\frac{1}{p}}\|f\|_{C^1(\mathbb{T}^d)}\|g\|_{L^p(\mathbb{T}^d)}.
\end{aligned}
\end{equation}
\end{lemma}
\begin{lemma} (\citep[Lemma 6]{LT-hyperviscous NS})\label{Lemma commutator estimate1}
 Let $a \in C^2(\mathbb{T}^3)$. For all $1<p<+\infty$ we have
\begin{equation}\label{6.8}
	\begin{aligned}
	\||\nabla|^{-1} \mathbb{P}_{\neq 0}(a \mathbb{P}_{\geq k} f)\|_{L^p(\mathbb{T}^3)} \lesssim k^{-1}\|\nabla^2 a\|_{L^{\infty}(\mathbb{T}^3)}\|f\|_{L^p(\mathbb{T}^3)},
\end{aligned}
\end{equation}
	holds for any smooth function $f \in L^p(\mathbb{T}^3)$.
\end{lemma}

We also recall from \citep{BLIS-Elurflows} the H\"{o}lder estimate, which is useful to prove Lemmas \ref{Lemma Magnetic amplitudes} and  \ref{Lemma Velocity amplitudes}.
\begin{lemma} (\citep[Proposition C.1]{BLIS-Elurflows})\label{composition estimate}
	 Let $\Psi: \Omega \longrightarrow \mathbb{R}$ and $f: \mathbb{R}^n \longrightarrow \Omega$ be two smooth functions, with $\Omega \subset \mathbb{R}^m$. Then, for every $N \in \mathbb{N}$, there is a constant $C=C(n, m, N)>0$ such that 
  \begin{align}
	&\|\Psi \circ f\|_{\dot{C}^N } \leq C\(\|\Psi\|_{\dot{C}^1}\|f\|_{\dot{C}^N }+\|\mathrm{D} \Psi\|_{C^{N-1} }\|f\|_{C^{0}}^{N-1}\|f\|_{\dot{C}^N }\),\label{6.9}\\	
	&\|\Psi \circ f\|_{\dot{C}^N }\leq C\(\|\Psi\|_{\dot{C}^1}\|f\|_{\dot{C}^N }+\|\mathrm{D} \Psi\|_{C^{N-1} }\|f\|_{\dot{C}^1 }^N\).\label{6.10}
\end{align}

\end{lemma}

\medskip 
\noindent{\bf Acknowledgment.}
The authors would like to thank Zirong Zeng for helpful discussions. Yachun Li thanks the support by NSFC (No. 112161141004, 2371221, and 11831011). Deng Zhang thanks the supports by NSFC (No. 12271352, 12322108, 12161141004). Yachun Li is also grateful for the supports by Fundamental Research Funds for the Central Universities. Yachun Li and Deng Zhang are also grateful for the supports by Shanghai Frontiers Science Center of Modern Analysis.

 \end{document}